\documentclass[11pt]{article}
\usepackage{amsmath,amssymb,latexsym,graphics,amsthm}
 \input diagxy
 \input xy
 \xyoption{all}
 \makeindex

\newtheorem{lemma}{Lemma}[section]
\newtheorem{theorem}[lemma]{Theorem}
\newtheorem{corollary}[lemma]{Corollary}
\newtheorem{claim}[lemma]{Claim}
\newtheorem{proposition}[lemma]{Proposition}

\theoremstyle{definition}
\newtheorem{definition}[lemma]{Definition}

\theoremstyle{remark}
\newtheorem*{remark}{Remark}
\newtheorem*{remarks}{Remarks}
\newtheorem*{notation}{Notation}

\newcommand{\s}{[\![}%left 'value' bracket
\newcommand{\y}{]\!]}%right 'value' bracket
\newcommand{\val}[1]{\s #1 \y}%'value of..'
\newcommand{\fr}[1]{\ensuremath{\langle #1 \rangle}}%free cat over..
\newcommand{\ff}[1]{\ensuremath{\fr{#1}^{*}}}%free cat assignment companion
\newcommand{\el}{\rightarrow}%compact arrow (el=to in Hebrew)

\newcommand{\bx}{\ensuremath{\mathbf{X}}}%capital boldfaces
\newcommand{\by}{\ensuremath{\mathbf{Y}}}
\newcommand{\bz}{\ensuremath{\mathbf{Z}}}
\newcommand{\ba}{\ensuremath{\mathbf{A}}}
\newcommand{\bb}{\ensuremath{\mathbf{B}}}
\newcommand{\bc}{\ensuremath{\mathbf{C}}}

\newcommand{\bs}{\ensuremath{\mathbf{S}}}
\newcommand{\bw}{\ensuremath{\mathbf{W}}}

\newcommand{\rep}{\raisebox{.6ex}{\ensuremath{\,\framebox[5pt]{\rule{0pt}{.5pt}}_{\,}}}}%replacement box symbol
\newcommand{\repq}{\raisebox{.6ex}{\ensuremath{\,\framebox[5pt]{\rule{0pt}{.5pt}}_{\,q}}}}
\newcommand{\repr}{\raisebox{.6ex}{\ensuremath{\,\framebox[5pt]{\rule{0pt}{.5pt}}_{\,r}}}}
\newcommand{\drep}{\raisebox{.6ex}{\ensuremath{\,\framebox[5pt]{\rule{0pt}{.5pt}.}_{\,}}}}%dotted replacemnet box

\newcommand{\sbs}[1]{_{\text{\tiny{\ensuremath{#1}}}}}%tiny subscript
\newcommand{\sps}[1]{^{\text{\tiny{\ensuremath{#1}}}}}%tiny superscript
\newcommand{\tc}[1]{\text{\tiny{\ensuremath{#1}}}}%tiny script

\newcommand{\eqv}[1]{#1/\!\!\approx}%equivalence class of..
\newcommand{\oc}[1]{\langle #1 \rangle}%<#1>
\newcommand{\doc}[1]{|\langle #1 \rangle|}%|<#1>|

\newcommand{\stc}{\mathcal{C}}%capital script letters
\newcommand{\stm}{\mathcal{M}}
\newcommand{\stt}{\mathcal{T}}

\newcommand{\stn}{\mathcal{N}}
\newcommand{\sta}{\mathcal{A}}

\newcommand{\dcc}{\mathbb{C}}%capital double characters
\newcommand{\dcd}{\mathbb{D}}

\newcommand{\dcr}{\mathbb{R}}
\newcommand{\dcw}{\mathbb{W}}
\newcommand{\dcz}{\mathbb{Z}}

\newcommand{\eps}{\varepsilon}%the 'normal' greek epsilon
\newcommand{\p}{\sps{\prime}}%prime
\newcommand{\pp}{\sps{\prime\prime}}%double prime
\newcommand{\hd}{\hat{d}}%'hatted' d
\newcommand{\hc}{\hat{c}}%'hatted' c
\newcommand{\hb}{\hat{\bullet}}%'hatted' bullet
\newcommand{\hrep}{\widehat{\rep}}%'hatted' rep
\newcommand{\hrepr}{\widehat{\rep}_r}%'hatted repr

\begin{document}
\def\xypic{\hbox{\rm\Xy-pic}}

\bibliographystyle{amsplain}

\title{Computads and Multitopic Sets}
\author{Victor Harnik\setcounter{footnote}{0} \footnote{corresponding author}
\\%Department of Mathematics,
University of Haifa
\\harnik@math.haifa.ac.il
\and Michael Makkai
\\%Department of Mathematics and Statistics,
 McGill University
 \\makkai@math.mcgill.ca
\and Marek Zawadowski
\\%Institute of Mathematics,
Warsaw University
\\zawado@mimuw.edu.pl
}

\maketitle

\begin{abstract} We compare \emph{computads} (as defined
in~\cite{S1},~\cite{S2},~\cite{B}) with \emph{multitopic sets}
(cf.~\cite{HMP1}-~\cite{HMP3}). Both these kinds of structures have
$n$-dimensional objects (called \emph{$n$-cells} for computads and
\emph{n-pasting diagrams} for multitopic sets), for each natural number $n$.
In both cases, the set of $n$-dimensional objects is freely generated by one
of its subsets. The computads form a subclass of the more familiar collection
of \emph{$\omega$-categories} while multitopic sets are of a more novel
nature, being based on an iteration of free \emph{multicategories}.
Multitopic sets have been devised as a vehicle for a definition of the
concept of \emph{weak $\omega$-category}. Our main result states that the
category of multitopic sets is equivalent to that of \emph{many-to-one}
computads, which is a certain full subcategory of the category of all
computads.
\end{abstract}

\section*{Introduction and preliminaries}\label{S:introduction} The notion
of \emph{free structure} has penetrated all parts of modern algebra. It has
the following abstract generalization. Given categories $\bc$ and $\bs$ and a
functor $U: \bc \el \bs$, we say that an object $A$ of $\bc$ is \emph{free}
with respect to $U$ iff for some object $I$ and arrow $\iota: I \el UA$ in
$\bs$, the following \emph{universal property} holds: for every object $B$ of
$\bc$ and arrow $\phi: I \el UB$ of $\bs$, there is a \emph{unique}
$\bc$-arrow $f: A \el B$ such that the following diagram commutes:
$$\bfig \Vtriangle/>`>`.>/[{I}`{UA}`{UB};{\iota}`{\phi}`{Uf}] \efig$$
We say that $I$ \emph{generates} the free object $A$ (\emph{via} the arrow
$\iota$). In the familiar cases, the objects of $\bc$ and $\bs$ are
mathematical structures, $I$ is a substructure of $UA$ with $\iota$ being the
\emph{inclusion} map, and the elements of (the universe of) $I$ are called
\emph{generators} of the free structure $A$.

For example, if $\bc$ is the \emph{category of commutative rings}, $\bs$ the
\emph{category of sets} and $U$ the \emph{forgetful functor}, the free ring
generated by a set $X$ is nothing but the ring $\dcz[X]$ of polynomials with
integral coefficients and \emph{indeterminates} from the set $X$. Borrowing
terminology from this example, we will usually refer to the generators of a
free structure as \emph{indeterminates} or, in short, \emph{indets}.

Another familiar example is that of a \emph{free category} generated by a
\emph{directed graph} (see, e.g.~\cite{CWM}, $\S7$ of Chapter II). In this
case, $\bc$ is the category $Cat$ of (small) categories, $\bs$ is the
category $Grph$ of directed graphs and $U$ is, again, the forgetful functor.

The notion of free category has been generalized to higher dimensional
categories by Street, leading to the concept of \emph{computad} which is
central for the present work (cf.~\cite{S1} and~\cite{S2} for the
$2$-dimensional case and~\cite{B} for the general definition).

To fix our notations, we now recall the structure of higher dimensional
categories. An $n$-dimensional category or, in short, an \emph{$n$-category}
$\bc$ has a set of \emph{$k$-cells} $C_k$, for each $k \leqslant n$. For $k >
0$, it has domain and codomain functions $d,c: C_k \el C_{k-1}$; thus, a
$k$-cell $u$ is envisaged as an arrow $du \to^u cu$, linking its domain $du
\in C_{k-1}$ to its codomain $cu \in C_{k-1}$. For $k \geqslant 2$, we also
require $du$ and $cu$ to be \emph{parallel}, meaning that $ddu=dcu$ and
$cdu=ccu$, i.e. $du$ and $cu$ have the same domain and the same codomain. For
the sake of uniformity, we say that any two $0$-cells are parallel, so that
we can say that $du \parallel cu$ whenever $u \in C_k$, with $k>0$. If $u$ is
an $l$-cell and $k < l$ we let $d^{(k)}u$ be the $k$-cell obtained from $u$
by $l-k$ successive applications of the domain function $d$; the $k$-cell
$c^{(k)}u$ is defined similarly. The $n$-category $\bc$ is also equipped with
\emph{partial composition operations} $\bullet_k$ for $k<n$. If $u,v \in C_l$
and $k < l$, then $u \bullet_k v$ is an $l$-cell that is defined iff
$d^{(k)}u= c^{(k)}v$. Finally, with each $l$-cell $u$, $l < n$, $\bc$ has an
\emph{identity $(l+1)$-cell} $u \to^{1_u} u$. The concepts that we mentioned,
satisfy certain axioms. For a precise definition, see~\cite{Leinster}, as
well as section~\ref{S:clanguage} below.

An $\omega$-category is one that has $n$-cells for each natural number $n <
\omega$ (as customary in set theory, $\omega$ is the first infinite ordinal
number). An $n$-category can be seen as an $\omega$-category in which all
cells of dimension $>n$ are identities. An $\omega$-functor $F: \bc \el
\bc\p$ between $\omega$-categories is a map from the cells of $\bc$ to those
of $\bc\p$ that preserves the $\omega$-categorical structure. The category
$\omega Cat$ of $\omega$-categories is the one that has the small
$\omega$-categories as objects and the $\omega$-functors as arrows.

For an $N$-category $\bc$, with $N \leqslant \omega$, and $n < N$, let
$\bc_n$ be the $n$-category whose $k$-cells are the same as those of $\bc$,
for all $k \leqslant n$. $\bc_n$ is called the $n$th truncation of $\bc$. If
$\ba$ is an $n$-category, we say that $\bc$ \emph{extends} $\ba$ iff
$\bc_n=\ba$.

Fix an $(n-1)$-category $\ba$. An \emph{extension} of $\ba$ will be any
$n$-category that extends $\ba$. A \emph{pre-extension} $(I,d,c)$ of $\ba$
will be a set $I$ together with functions $d,c: I \el A_{n-1}$, such that $dx
\parallel cx$ for $x \in I$ (remember that $A_{n-1}$ is the set of $(n-1)$-cells of $\ba$).
For the sake of this preliminary discussion, let us introduce the categories
$Ext(\ba)$ and $Preext(\ba)$; the former has the extensions of $\ba$ as
objects while its arrows are the $\omega$-functors that extend the identity
functor on $\ba$. The latter category has the pre-extensions of $\ba$ as
objects and the structure preserving maps as arrows. There is an obvious
\emph{forgetful} functor $U: Ext(\ba) \el Preext(\ba)$. An extension $\bb$ of
$\ba$ is called \emph{free} if it is a free object of $Ext(\ba)$ with respect
to $U$, in the sense of the definition that opened this introduction. This
concept is at the heart of Street's definition. An $\omega$-category $\ba$ is
called a \emph{computad} iff $\ba_{n+1}$ is a free extension of $\ba_n$, for
each $n < \omega$.

In the first part of this paper
(sections~\ref{S:clanguage}-\ref{S:comparing}), we concentrate on the study
of free extensions of finite dimensional categories. We start by presenting a
construction of a free extension, using a  method familiar from universal
algebra (cf. e.g.~\cite{Gr}): given a pre-extension $(I,d,c)$ of an
$(n-1)$-category $\ba$, we set up a formal \emph{equational language} $\stc$
which has \emph{terms} denoting all the cells that can be constructed from
the elements of $I$ by applications of the partial operations defined among
$n$-cells in an extension of $\ba$. The language $\stc$ has also \emph{a
deductive system} that allows us to prove equalities among terms. Two terms
are called \emph{equivalent} if their equality is provable in $\stc$. The
elements of the free extension constructed by this method, will be the
equivalence classes of $\stc$-terms.

The same method could be used to construct the free ring generated by a set
of indeterminates $X$. However, a simplifying circumstance occurs in this
case. The terms of the corresponding formal language are algebraic
expressions that use indeterminates, constants for $0$ and $1$, binary
operation symbols $+,\,\cdot,\,-$ as well as parentheses. Each such term $t$
can be proven to be equal to a polynomial, which is \emph{unique} (assuming
that the monomials that are the terms of the polynomial occur in a canonical
ordering induced by a given ordering of the set of indeterminates). We shall
call this polynomial the \emph{reduced} form of the term $t$. This situation
allows us to replace the equivalence class of $t$ by the unique polynomial
which is the common reduced form of the members of this class. The free
structure generated by $X$ becomes, in this way, a \emph{term model}, i.e. a
structure whose elements are individual terms, rather than equivalence
classes. This is how the polynomial ring $\dcz[X]$ is obtained.

Can the free extension $\bx$ of an $n$-category $\bb$, generated by a given
pre-extension, be also construed as a term model? In other words, can we
substitute each equivalence class of terms by a "canonical" representative, a
common "reduced form" of its elements? \emph{Under certain conditions}, the
answer to this question is positive. This result is just one corollary, a
side benefit, of the study that we conduct in
sections~\ref{S:placedcomp}-\ref{S:comparing}. We are now going to describe
the content of these sections, in rough terms.

Assume that $\bb$ itself is a \emph{free} extension of an $(n-1)$-category
$\ba$, generated by a set $I$ of $n$-dimensional indets (i.e., generated by a
pre-extension of the form $(I,d,c)$) and let $\bx$ be any extension of $\bb$.
Call an $(n+1)$-cell of $\bx$, $u \in X_{n+1}$,  \emph{many-to-one} iff its
codomain $cu$ is an indet, i.e. $cu \in I$. We define, in
section~\ref{S:placedcomp}, a partial binary operation, called \emph{placed
composition}, between the many-to-one cells of $\bx$. As it turns out, the
many-to-one cells of $\bx$, together with the operation of partial
composition yield a structure $\dcc_{\bx}$ which is a \emph{multicategory}.
The abstract notion of multicategory, described in section~\ref{S:multicats},
is a generalization, introduced in~\cite{HMP2}, of a notion due to Lambek
(cf.~\cite{Lambek}). Free multicategories do have \emph{term models}, as
shown in~\cite{HMP2} and briefly sketched in section~\ref{S:planguage}. The
main technical result of this paper, theorem~\ref{L:main}, states that if
$\bx$ is a free extension of $\bb$, generated by a set $J$ of many-to one
indets, then the multicategory $\dcc_{\bx}$ is also free (and, actually,
generated by the same set of indets $J$).

As we stated already, a computad is obtained by starting with a barren set
and iterating the free extension construction indefinitely. If, at each
stage, the generating indets are many-to-one cells, then we get a
\emph{many-to-one} computad. These are the objects of a category $m/1Comp$
described in section~\ref{S:computads}. The many-to-one cells of a
many-to-one computad $\ba$ together with the (partial) operations of placed
composition and the domain/codomain functions, form a structure $S_{\ba}$.
This structure is a \emph{multitopic set}, an abstract notion introduced
in~\cite{HMP3}. Roughly speaking, a multitopic set is a structure obtained by
iterating indefinitely the construction of free multicategory. The precise
setup, as well as the description of the category $mltSet$ of multitopic
sets, are presented also in section~\ref{S:computads}. In
section~\ref{S:equivalence} we show, using the results of
section~\ref{S:comparing}, that actually, \emph{all} multitopic sets are of
the form $S_{\ba}$ for some many-to-one computad $\ba$. We then infer that
the categories $m/1Comp$ and $mltSet$ are equivalent. More colorfully said,
\emph{multitopic sets are the same as many-to-one computads}. This is the
main result of our paper.

Multitopic sets have been introduced in the sequence of papers \cite{HMP1},
\cite{HMP2}, \cite{HMP3} as a vehicle for producing the "right" definition
for the notion of \emph{weak higher dimensional category}. This approach was
inspired by an earlier attempt of Baez and Dolan (cf. \cite{BD1},
\cite{BD2}). See~\cite{Leinster} for a survey of the competing definitions of
weak higher dimensional categories, including the one of~\cite{M1}, based on
multitopic sets. Our main result shows that the definition of~\cite{M1} could
be rephrased using the more familiar notion of many-to-one computad.

An alternative approach for defining weak higher dimensional categories,
based on a concept called \emph{dendrotopic sets}, has been devised by Palm
in~\cite{P}. In addition, Palm shows that the category of dendrotopic sets is
equivalent to that of many-to-one computads, thus concluding that the
categories of multitopic sets and of dendrotopic sets are also equivalent.

\medskip

We conclude the preliminaries by recalling one more notation. If $u$ is a
$k$-cell of an $\omega$-category $\ba$ and $k<n$, then we let $1_u^{(n)}$ be
the $n$-cell obtained from $u$ by $n-k$ successive applications of the $x
\mapsto 1_x$ operation.

\section{Free extensions}\label{S:clanguage}

Let $(I,d,c)$ be a pre-extension of an $(n-1)$-category $\ba$, meaning, as we
recall, that $I$ is a set and $d,c: I \rightarrow A_{n-1}$ are functions such
that $dx\parallel cx$ for each $x \in I$. As in the introduction, the
elements of $I$ will be called \emph{ $n$-indets}. We should think of an
$n$-indet $x$ as denoting an arbitrary $n$-cell belonging to an
$\omega$-category \emph{extending} $\ba$ (i.e. an $\omega$-category whose
$(n-1)$th truncation is $\ba$), having domain and codomain $dx$ and $cx$. We
now define an equational language $\mathcal{C}=\mathcal{C}(\ba,I,d,c)$,
dealing with the $n$-cells obtained from the (cells denoted by) $n$-indets,
by repeated compositions. The symbols of $\stc$ will be the $n$-indets, the
\emph{composition symbols} $\bullet_k$, for $k<n$, as well as the
\emph{identity symbols} $1_a$, for each $(n-1)$-cell $a \in A_{n-1}$. Besides
these, $\stc$ will employ left and right parentheses as \emph{auxiliary}
symbols.

\begin{definition}\label{D:cterms} The set $\stt (\stc)$ of
$\stc$-terms and the domain and codomain functions \newline \mbox{$d$,$c:\stt
(\stc) \rightarrow A_{n-1}$} are defined as follows:
\begin{enumerate}
\item Every $n$-indet $x$ is a $\stc$-term with $dx$, $cx$ as
specified by the given functions \mbox{$d$,$c:I \rightarrow A_{n-1}$}.

\item For each $a \in A_{n-1}$, $1_a$ is a $\stc$-term with
$d1_a=c1_a=a$.

\item If $t$, $s$ are $\stc$-terms and $d^{(k)}t=c^{(k)}s$, then
$(t)\bullet_k (s)$ is a $\stc$-term (the parentheses around $t$ and $s$
insure unique readability; usually, we just write $t\bullet_k s$) and we have

\[d(t\bullet_k s)=\begin{cases}ds & \text{if }k=n-1\\dt \bullet_k
ds& \text{if } k<n-1 \end{cases}
\]

\[c(t\bullet_k s)=\begin{cases}ct & \text{if }k=n-1\\ct \bullet_k
cs& \text{if } k<n-1 \end{cases}
\]

\item There are no $\stc$-terms besides those mentioned in 1-3.
\end{enumerate}
\end{definition}

The meaning of the $\stc$-terms should be clear. If $\bx$ is an
$\omega$-category extending $\ba$ and if \mbox{$\varphi: I \rightarrow X_n$}
is an \emph{assignment} which is \emph{correct}, meaning that $d\varphi x=dx$
and $c\varphi x=cx$ for all $x\in I$, then we can \emph{evaluate} any
$\stc$-term t \emph{under} the said assignment and get the value
$val_{\varphi}(t)\in X_n$. Remember that when saying that $\bx$ extends
$\ba$, we mean that $\ba=\bx_{n-1}$ (the $(n-1)$th  truncation of $\bx$).
More generally, if $\bx$ is \emph{any} $\omega$-category,
\mbox{$F:\ba\rightarrow \bx$} an $\omega$-functor and
\mbox{$\varphi:I\rightarrow X_n$} an assignment that is \emph{consistent}
with $F$, in the sense that $d\varphi x=Fdx$, $c\varphi x=Fcx$ for $x\in I$,
we can evaluate $t$ \emph{under} $F,\varphi$ and get the value
$val_{F,\varphi}(t)\in X_n$ . The formal definition runs as follows.

\begin{definition}\label{D:val} Under the assumptions that we just
mentioned, we define the function \mbox{$val=val_{F,\varphi}:\stt (\stc)
\rightarrow X_n$} , by induction on $\stc$-terms:
\begin{enumerate} \item $val(x)=\varphi x$, for $x\in I$.
\item $val(1_a)=1_{Fa}$, for $a\in A_{n-1}$.
\item $val(t\bullet_k s)=val(t)\bullet_k val(s)$.
\end{enumerate}

If $\ba=\bx_{n-1}$ and $\varphi$ is a correct assignment, we let
$val_{\varphi}=val_{i_{\ba},\varphi}$, where $i_{\ba}$ is the inclusion
$\omega$-functor of $\ba$ into $\bx$.
\end{definition}

It may so happen, that for terms $t$ and $s$ we have
$val_{F,\varphi}(t)=val_{F,\varphi}(s)$ for \emph{all} $F$ and $\varphi$.
This occurs whenever $t$ and $s$ \emph{must} be equal in virtue of the axioms
of $\omega$-category. We can describe this situation precisely, by setting up
a deductive system for proving equality of terms. This is done in the
definition below, which completes the presentation of the equational logical
system $\stc=\stc(\ba,I,d,c)$. Let us mention that the axioms of the notion
of $\omega$-category are the associativity, exchange and identity axioms of
this definition.

\begin{definition}\label{D:axioms} We define the deductive system
$\stc$ as follows, where, in the axioms and rules below, $t,s,w,t_1,s_1$ are
arbitrary $\stc$-terms and all compositions are supposed to be well defined
(according to definition~\ref{D:cterms}).

\vspace{3pt}

\emph{\textbf{Axioms.}}

\vspace{-9pt}

\begin{enumerate}
\item $t=t$ (equality axioms).
\item $(t\bullet_k s)\bullet_k w=t\bullet_k(s\bullet_k w)$
(associativity axioms).
\item $(t\bullet_k t_1)\bullet_l(s\bullet_k s_1)=(t\bullet_l s) \bullet_k
(t_1 \bullet_l s_1)$, where $l<k<n$ (exchange axioms).
\item $1^{(n)}_b \bullet_k t=t=t \bullet_k 1^{(n)}_a$, where $k <
n$, $c^{(k)}t=b$ and $d^{(k)}t=a$.

Also, $1_a\bullet_k 1_b=1_{a\bullet_k b}$, where $a,b \in A_{n-1},\,
d^{(k)}a=c^{(k)}b$ (identity axioms). \end{enumerate}

\vspace{-9pt}

 \emph{\textbf{Rules.}}

\vspace{-3pt}

\begin{enumerate}
\item $\dfrac{t=s}{s=t} \qquad \dfrac{t=s \quad s=w}{t=w}$
(equality rules)
\item $\dfrac{t=s}{t\bullet_k w=s\bullet_k w} \qquad
\dfrac{t=s}{w\bullet_k t=w\bullet_k s}$ (congruence rules)
\end{enumerate}
\end{definition}

We will write `$\vdash t=s$' or, sometimes, $\vdash_{\stc} t=s$ to indicate
that $t=s$ is provable in this system.

\medskip

Is this system complete? In other words are we sure that, whenever $\nvdash
t=s$, there are $\bx$,$F$ and $\varphi$ for which $val(t) \neq val(s)$? The
\emph{positive} answer to this question, follows from the existence of
\emph{free} extensions.

\begin{theorem}\label{L:freeext}
Given $\ba$, $I,d,c$ as above, there exists an $n$-category $\ba[I]$
satisfying:

\begin{enumerate}
\item $\ba[I]$ is an extension of $\ba$, i.e its $(n-1)$-th truncation is $\ba$,
$\ba[I]_{n-1}=\ba$.
\item Each $x \in I$ is an $n$-cell of $\ba[I]$ with domain
$dx$ and codomain $cx$.
\item $\ba[I]$ has the following \emph{universal property}: if $\bx$ is any
$\omega$-category extending $\ba$ and $\varphi: I \el X_n$ a function
satisfying that $d\varphi x=dx$ and $c\varphi x=cx$ for $x \in I$, then there
is a \emph{unique} $\omega$-functor $G: \ba[I] \el \bx$ such that $Ga=a$ when
$a$ is a cell of $\ba$ and $Gx=\varphi x$ for $x \in I$.

\medskip

Moreover, $\ba[I]$ has the following \emph{strong universal property}:
whenever $\bx$ is an \mbox{$\omega$-category}, $F:\ba \rightarrow \bx$ an
$\omega$-functor and $\varphi:I \rightarrow X_n$ a function such that
$d\varphi x=Fdx$, $c\varphi x=Fcx$ for $x \in I$, there is a \emph{unique}
\mbox{$\omega$-functor $G:\ba[I]\! \rightarrow \!\bx$} such that $Ga=Fa$
whenever $a$ is a cell of $\ba$ and $Gx=\varphi x$ for \mbox{$x \in I$}.
\end{enumerate}
\end{theorem}

\begin{remark} The universal property means that $\ba[I]$ is a free extension
of $\ba$ in the sense explained in the introduction. The \emph{strong}
universal property means that $\ba[I]$ is free with respect to a forgetful
functor $U_1: \bc_1 \el \bs_1$, where $\bc_1$ is $\omega Cat$ while $\bs_1$
is a category whose objects are pairs $\langle \bb, (Z,d,c) \rangle$ with
$\bb$ an $(n-1)$-category and $(Z,d,c)$ a pre-extension of $\bb$ (the
interested reader should have no problems in identifying the arrows of
$\bs_1$ and the definition of $U_1$).
\end{remark}

\begin{proof} As outlined in the introduction,
the $n$-cells of $\ba$ will be \emph{equivalence classes} of $\stc$-terms,
under a suitable equivalence relation.

\begin{claim}\label{L:vdash} \begin{enumerate}
\item[\textup{(a)}] If we define, for $\stc$-terms $t$ and $s$, $t\approx s
\text{ iff } \vdash{t=s}$, then $\approx$ is an equivalence relation that is
a congruence with respect to $\bullet_k,\ k < n$.
\item[\textup{(b)}] If $\vdash t=s$ then $dt=ds,\ ct=cs$.
\item[\textup{(c)}] If $\vdash t=s$ then
$val_{F,\varphi}(t)=val_{F,\varphi}(s)$, for all $F$ and $\varphi$.
\end{enumerate} \end{claim}

\begin{proof} (a) is immediate (congruence with respect to
$\bullet_k$ means that $t\approx s$ implies \mbox{$t\bullet_kw\approx
s\bullet_kw$} and $w\bullet_kt\approx w\bullet_ks$).

(b) and (c) are easily checked by induction on proofs. \end{proof}

We can now describe the $n$-category $\ba[I]$. The cells of $\ba[I]$ of
dimension $\leqslant n-1$ are those of $\ba$, while the $n$-cells are the
equivalence classes $t/\!\!\approx$ for $t \in \stt (\stc)$, where
$d(t/\!\!\approx)=dt$, $c(t/\!\!\approx)=ct$ and
\mbox{$(t/\!\!\approx)\bullet_k(s/\!\!\approx)=t\bullet_k s/\!\!\approx$}.

\medskip

Claim~\ref{L:vdash}(a)(b), insures that the definitions of $c, d \text{ and }
\bullet_k$  are correct, and the axioms of our deductive system insure that
we defined, indeed, an $n$-category. However, we wanted the elements of $I$
to be $n$-cells of $\ba[I]$ and what we have, instead, is that
$x/\!\!\approx$ is a such, for every $x \in I$. To correct this, we only have
to \emph{identify} $x$ with $x/\!\!\approx$. To be sure that we do not make
unwanted identifications in this way, we have to check that $x\not\approx y$,
whenever $x \neq y$ for $x,y \in I$. This is easily seen, however. It should
be clear when do we say that an indet $x$ \emph{occurs} in a term $t \in \stt
(\stc)$. A straightforward verification shows that the following is true.

\begin{claim}\label{L:xydistinct} If $\vdash t=s$ then any
indet $x \in I$ occurs in $t$ iff it occurs in $s$. Hence, if $x \text{ and }
y$ are \emph{distinct} indets, then $\not\vdash x=y$, which means that $x
\not\approx y$. \end{claim}

This shows that we can, indeed, identify $x$ with $x/\!\!\approx$ and assume
that the elements of $I$ are $n$-cells of $\ba[I]$.

To conclude the proof, it is enough to show that $\ba[I]$ has the
\emph{strong universal property} stated in part \emph{3} of~\ref{L:freeext}.
Given an $\omega$-functor $\ba \to^F \bx$ and a function $\varphi:I
\rightarrow X_n$ such that $d\varphi x=Fdx, \ c\varphi x=Fcx$, we define
$G:\ba[I] \rightarrow \bx$ by

\vspace{-9pt}

 \[ Ga=Fa \text{ for $a$ a cell of $\ba$ and }
G(t/\!\!\approx)=val_{F,\varphi}(t) \text{ for } t \in \stt (\stc) .\]
Claim~\ref{L:vdash}(c) implies that this definition is correct and
definition~\ref{D:val} of $val_{F,\varphi}(-)$ insures that $G$ is an
$\omega$-functor. It follows immediately that $G$ extends both, $F$ and
$\varphi$ and that any such $G$ \emph{has} to be defined as above. Thus, $G$
is unique and we have proved~\ref{L:freeext}.
\end{proof}

If we now let $i_{\ba}:\ba \rightarrow \ba[I]$ be the inclusion functor and
$\varphi$ be the inclusion function from $I$ into the $n$-cells of $\ba[I]$,
then an easy induction on terms shows that
\mbox{$val_{i_{\ba},\varphi}(t)=\eqv{t}$.} This fact yields immediately the
following.

\begin{corollary}\label{L:completeness} The deductive system $\stc$
is complete, namely, if $\nvdash t=s$, then for some $\bx$, $F$ and $\varphi$
we have $val_{F,\varphi}(t) \neq val_{F,\varphi}(s)$.
\end{corollary}

\begin{remark} As easily seen, the \emph{universal property}
of~\ref{L:freeext}, part \emph{3}, determines $\ba[I]$ uniquely up to an
isomorphism (actually, up to a \emph{unique} isomorphism that is the identity
for the cells of $\ba$ and for the elements ($n$-indets) of $I$). It follows
that the universal property actually implies the \emph{strong} universal
property. \end{remark}

An $n$-category $\bb$ will be called a \emph{free extension} of $\ba$ iff it
extends $\ba$ and for some $I \subset B_n$, $\bb$ has the universal property
of $\ba[I]$ (and hence, it is isomorphic to $\ba[I]$, as just remarked). We
also say, in this situation, that $\bb$ is freely generated by the set $I$
(an abbreviated terminology that suppresses $\ba$).

\medskip

\noindent\textbf{An important convention.} A $0$-\emph{category} $\bb$
consists of the set $B\sbs{0}$ of its $0$-cells, and nothing more. Thus, a
$0$-category is just a barren set (this is a customary point of view). An
$\omega$-functor from such a $\bb$ to any $\omega$-category $\bx$ is just a
function from $B\sbs{0}$ to the set $X\sbs{0}$ of $0$-cells of $\bx$. We will
say that \emph{any} $0$-category $\bb$ is \emph{freely generated} by the set
$B\sbs{0}$ of its $0$-cells. This is justified because the obvious universal
property holds trivially. Also, we will sometimes refer to the $0$-cells of
\emph{any} $\omega$-category as $0$-\emph{indets}.

This terminology will turn out to be convenient in the sequel, as it will
allow the inclusion of the case $n=0$ in several statements.

\medskip

We conclude this section with a remarkable property of free extensions. As
the statement and, even more so, the proof, involve some technical details,
the reader may wish to skip this on first reading and return to it when it is
invoked in later sections.

In analogy with the notion of free group, one might expect that the same free
extension of $\ba$ might be generated by several distinct sets of $n$-indets.
In many important instances, this is not so, however. As it turns out, under
certain conditions, the set of $n$-indets of a free extension is uniquely
determined.

\begin{definition}\label{D:indec} An $n$-cell $u$ of an
$\omega$-category $\bx$ is \emph{indecomposable} if whenever $u=v \bullet_k
w$, with $k<n$, then either $u=1^{(n)}_a$ or $v=1^{(n)}_a$, where
$a=d^{(k)}v=c^{(k)}w$.
\end{definition}

Identity cells are, in general, decomposable in many obvious ways. For
example, if $u,v$ are \emph{non-identity} $m$-cells such that $u\bullet_{m-1}
v=a$ is defined, then \mbox{$1_a=1_u \bullet_{m-1} 1_v$}, showing that $1_a$
is \emph{decomposable}. More generally, if $l<k<m$ and \mbox{$a=u\bullet_l
v$}, where $u,v$ are $k$-cells, then it is easy to see that \mbox{$1^{(m)}_a=
1^{(m)}_u\bullet_l 1^{(m)}_v$}. We will consider this kind of decompositions
of $k$-identities to be trivial. A formal definition, which is wider in a
certain respect, will now be given. A cell of the form $1^{(m)}_a$, with $a$
a $k$-cell and $m>k$, will be called a \emph{$k$-identity of dimension $m$}.

\begin{definition}\label{D:indecid} A $k$-identity $e$ of
dimension $m>k$ is called \emph{essentially indecomposable} if whenever
$e=u\bullet_l v$ with $l\leqslant k$, then both $u$ and $v$ are $k$-identity
cells of dimension $m$. \end{definition}

\begin{remark} In the case of $l=k<m$, the condition of essential
indecomposability just means that if $e=1^{(m)}_a=s\bullet_k w$ with $a$ of
dimension $k$, then $s=w=1_a^{(m)}$. \end{remark}

\begin{definition}\label{D:wellbehaved} An $n$-category $\bx$ is
\emph{well-behaved} if, for all $k<m \leqslant n$, all $k$-identities of
dimension $m$ are essentially indecomposable. \end{definition}

Notice that any $0$-category is trivially well-behaved. Also, as \emph{free
categories}, i.e. free extensions of $0$-categories, have a very simple
structure (cf. e.g. section 7 of chapter I in \cite{CWM}) and are easily seen
to be well-behaved. The remarkable result that we want to prove is the
following.

\begin{theorem}\label{L:indetisindec} If $\ba$ is a well-behaved
$(n-1)$-category and $I$ is a set of $n$-indets over it, then for any
$n$-cell $x$ of $\ba[I]$, $x \in I$ iff $x$ is indecomposable and is
\emph{not} an identity cell. Furthermore, $\ba[I]$ is also well
behaved.\end{theorem}

Thus, an $n$-dimensional extension $\bb$ of a well-behaved $(n-1)$-category
$\ba$ is free iff it is freely generated by the set of its non-identity
indecomposable cells.

For $n=1$, this theorem is easily checked, due to the above mentioned simple
structure of free categories. For $n>1$, the proof involves a deeper analysis
of the deductive system $\stc$. We begin with a definition.

\begin{definition}\label{D:terms} \begin{enumerate}
\item A term $t \in \stt (\stc)$ is called \emph{constant} iff no
variable occurs in $t$.
\item $t$ is called an \emph{identity} iff for some $(n-1)$-cell $a$
of $\ba$, $\vdash t=1_a$. An identity $t$ is called a \emph{$k$-identity}
(where $k < n$) iff $\vdash t=1^{(n)}_a$ for some $k$-cell $a$ of $\ba$.
\item A term $t$ is called \emph{indecomposable} iff whenever $\vdash
t=s\bullet_k w$ (with $k<n$) then one of $s,\,w$ is a $k$-identity.
\end{enumerate} \end{definition}

Thus, a term $t$ is indecomposable iff $t/\!\!\approx$ is an indecomposable
$n$-cell of $\ba[I]$.

The following simple statement implies immediately the ``if'' direction
of~\ref{L:indetisindec}.

\begin{proposition}\label{L:indecisindet} If $t \in \stt (\stc)$ is an
indecomposable term, then  $t$ is an identity or $\vdash t=x$ for some
variable $x \in I$. \end{proposition}

\begin{proof} By induction on $t$. If $t$ is
an identity or a variable then we have nothing to prove. If $t=s\bullet_k w$,
then either $s$ is a $k$-identity and then $\vdash t=w$ or $w$ is an identity
and $\vdash t=s$; in either case, the claim follows by the induction
hypothesis.\end{proof}

Next, we point out a very simple fact.

\begin{claim}\label{L:constantterm} A term $t$ is constant iff it
is an identity. \end{claim}

\begin{proof} By induction on $t$. If $t$ is an identity or an $n$-indet, this is
immediate. If $t=s\bullet_k w$ and $t$ is constant then so are $s,\,w$,
hence, by the induction hypothesis, we can find $(n-1)$-cells $a,\,b$ such
that $\vdash s=1_a,\ \vdash w=1_b$.

If $k=n-1$, then we have $a=d1_a=ds=cw=c1_b=b$, hence $\vdash t=s\bullet_n
w=1_a\bullet_n 1_a=1_a$, so $t$ is an identity. If $k<n-1$, then
$d^{(k)}s=d^{(k)}1_a=d^{(k)}a$ and likewise, $c^{(k)}w=c^{(k)}b$. As
$s\bullet_k w$ is defined, we have that $d^{(k)}a=c^{(k)}b$, hence
$a\bullet_k b$ is defined and we have, by one of the identity axioms, $\vdash
t=1_a \bullet_k 1_b=1_{a\bullet_k b}$. this completes the proof of the ``only
if'' direction of the claim. The ``if'' direction follows immediately by
claim~\ref{L:xydistinct}. \end{proof}

This allows us to infer the "Furthermore" part of \ref{L:indetisindec}.

\begin{claim}\label{L:extiswb} If $\ba$ is well-behaved then
so is $\ba[I]$.\end{claim}

\begin{proof} We have to show that
if $t \in \stt(\stc)$ is a $k$-identity and $\vdash t=s\bullet_k w$, for $k <
n$, then both $s$ and $w$ are k-identities. Indeed, in this case, $t,s,w$
must all be constant, hence identities, by~\ref{L:constantterm}. So, assume
that $\vdash t=1^{(n)}_a$ and $\vdash s=1_u, w=1_v$ with $a$ being a $k$-cell
and $u,v$ being $(n-1)$-cells of $\ba$. If $k=n-1$, we immediately infer that
$a=u=v$. If $k<n-1$, then $dt=ds\bullet_k dw$ which means that
$1^{(n-1)}_a=u\bullet_k v$ and as $1^{(n-1)}_a$ is a $k$-identity in $\ba$,
it is essentially indecomposable, which means that $u,v$ are also
$k$-identities hence, so are $s$ and $w$. \end{proof}

Each occurrence of a composition symbol $\bullet_k$ in a term $t \in \stt
(\stc)$ has a definite \emph{scope} which is a subterm of $t$ of the form
$s\bullet_k w$.

\begin{definition}\label{D:inesscomp} An occurrence of $\bullet_k$
with scope $s\bullet_k w$ in a term $t$ is called \emph{inessential} iff
\emph{either} one of $s,w$ is a $k$-identity \emph{or} both, $s$ and $w$ are
identities.
\end{definition}

To put it more colorfully, a composition occurrence in $t$ is inessential iff
it can be ``wiped out'' by the use of one of the \emph{identity} axioms of
our deductive system. The next lemma, which is crucial for the proof
of~\ref{L:indetisindec}, says, in effect, that this is the only way in which
a composition symbol can be made to disappear from a term of $\mathcal{T}$.

\begin{lemma}\label{L:mainlemma} Under the assumption of~\ref{L:indetisindec},
if $\vdash t=s$ and one of $\ t,s$ has only inessential occurrences of
composition, then so does the other .\end {lemma}

\medskip

\noindent \emph{Proof.} By induction on proofs. We have to show, first, that
the statement of the lemma is true for all the axioms and, second, that if
the statement is true for the premise, or the premises, of a rule then it is
true for its conclusion as well.

We start with the associativity axioms. Let $(t\bullet_k s)\bullet_k
w=t\bullet_k (s\bullet_k w)$ be a such and assume, e.g., that the left hand
side only inessential compositions. This means that the same is true for
$t,s,w$, so all we have to show is that the two $\bullet_k$-occurrences
indicated on the right are inessential. As the rightmost indicated occurrence
of $\bullet_k$ on the left hand side is inessential, we have three cases, and
we examine each of them separately.

Assume first that $t\bullet_k s$ is a $k$-identity. If so, then so are $t$
and $s$ and this implies that the two indicated occurrences of $\bullet_k$ on
the right are inessential. Second, assume that $w$ is a $k$-identity. But
then, the left hand side of the axiom is provably equal to $t\bullet_k s$ and
by assumption, this occurrence of $\bullet_k$ is also inessential and we
easily conclude that the compositions on the right hand side are also
inessential. Finally, if both $t\bullet_k s$ and $w$ are identities, then so
are $t$ and $s$ and all compositions on the right are inessential. This
completes the examination of the associativity axiom.

The case of exchange axioms is more complex. The argumentation is not hard,
but is somewhat tedious. Consider the instance
\[ (t\bullet_k t_1)\bullet_l(s\bullet_k s_1)=(t\bullet_l s) \bullet_k (t_1
\bullet_l s_1) \] where $l<k\leqslant n$.

Assume first that the \emph{left} side has no essential composition
occurrences. Then, certainly, $t,t_1,s,s_1$ have no such occurrences and so,
all we have to show this that, on the \emph{right} side, the three indicated
composition occurrences are inessential. As $\bullet_l$ on the left is
inessential, \emph{either} both terms that it binds are identities \emph{or}
one of these terms that is an $l$-identity. In the first case, $t, t_1,s,s_1$
are all identities and hence, all compositions on the right are inessential.
In the second case assume, e.g., that $t\bullet_k t_1$ is an $l$-identity. As
$l<k$, any $l$-identity is also a $k$-identity hence, by~\ref{L:extiswb}, is
essentially indecomposable and both $t$ and $t_1$ are $k$-identities. But
then, $\vdash t\bullet_k t_1=t=t_1$ and so, $t$ and $t_1$ are actually
$l$-identities. Taking into consideration that the second $\bullet_k$ on the
left is also inessential, we now easily conclude that all compositions on the
right are inessential.

Now assume that the right side of the exchange axiom has no essential
composition occurrences and let's show that the three compositions indicated
on the left are also inessential. \emph{Either} both terms bound by
$\bullet_k$ on the right are identities \emph{or} one of these is a
$k$-identity. The first case is, again, trivial, so let us consider the
second. Assume, e.g., that $t\bullet_l s$ is a $k$-identity and hence, is
essentially indecomposable. Then $t$ and $s$ are both $k$-identities and, as
the second $\bullet_k$ on the right is inessential, we conclude immediately
that all compositions on the left are inessential.

\smallskip

Checking the statement of the lemma for the other axioms is trivial and so is
for the rules.  \qed

\medskip

\noindent \emph{Proof of~\ref{L:indetisindec}.} It remains to show that any
indeterminate $x \in I$ is indecomposable. assume that $\vdash x=t\bullet_k
s$. As the basic term $x$ has no essential compositions, it follows
by~\ref{L:mainlemma} that $\bullet_k$ on the right is inessential.
By~\ref{L:xydistinct}, $x$ must occur in $t\bullet_k s$, which means that $t$
and $s$ cannot be both constant, i.e., by~\ref{L:constantterm} cannot be both
identities. We conclude that one of $t,s$ must be a $k$-identity. \qed

\section{Indet occurrences}\label{S:occurrence}
The notion of occurrence of an indet in (a $\stc$-term denoting) an $n$-cell
$u$ of $\ba[I]$ is surprisingly complex and will be discussed in the present
section.

We start by pointing out that the same indet may occur several times in a
term denoting $u$. A simple example: assuming that $a=dx=cy$ and $b=dy=cx$,
the $\stc$-term \mbox{$t=(x\bullet_{n-1}y)\bullet_{n-1}x$} denotes an
$n$-cell $a \to^u b$, the composite of the diagram
\[a\to^x b\to^y a\to^x b ,\] and the indet $x$ has two distinct
occurrences in $t$. As we shall see, in certain situations we will be
interested in replacing \emph{one} of these occurrences of $x$ by a cell $v$
of dimension $n$ or \emph{higher}~(!), such that $d^{(n-1)}v=a$ and
$c^{(n-1)}v=b$. Therefore, we must have a mean of indicating a particular
occurrence of an indet $x$ in an $n$-cell $u$. One solution could be to
arrange the occurrences of the indets in a sequence, in the order in which
they occur in $t$. In the example that we just considered, we are speaking of
the sequence $\langle x,y,x \rangle$. \emph{Unfortunately}, the same cell $u$
is denoted by several terms and the order of indet occurrences may vary from
one such term to the other. For instance, the terms \mbox{$t=(x\bullet_k
x_{\sbs{1}})\bullet_l (y\bullet_k y_{\sbs{1}})$} and \mbox{$s=(x\bullet_l
y)\bullet_k (x_{\sbs{1}}\bullet_l y_{\sbs{1}})$} denote the same $n$-cell,
where $l<k<n$. \emph{Fortunately}, whenever $t$ and $s$ denote the same cell,
i.e. whenever $\vdash t=s$, the same indets occur in both, each occurring the
same number of times in $t$ and in $s$ and, \emph{moreover}, each proof of
$t=s$ yields, in an obvious way, a one-to-one correspondence between the
indet occurrences in $t$ and those in $s$.

To deal with this situation, we start by attaching to each $n$-cell $u$ an
\emph{indexed set} $\oc{u}$ of indet occurrences; this is a function
\mbox{$\oc{u}:\doc{u} \el I$} whose domain is a finite set $\doc{u}$. An
indet $x \in I$ has an occurrence in $u$ iff it is in the range of $\langle
u\rangle$ and if this is the case, then  the number of occurrences of $x$ in
$u$ is the cardinality of the set \mbox{$\{r \in \doc{u}: \oc{u}(r)=x \}$}.

It will be useful to assume that all the domains $\doc{u}$ are subsets of a
given infinite set $\mathcal{N}$. Following \cite{HMP2}, we let $I^{\#}$ be
the category whose objects are the finite indexed subsets of $I$ (i.e. the
functions from finite subsets of $\mathcal{N}$ into $I$) and arrows are
defined in the obvious way. Let us mention, for further use, that in
$I^{\#}$, $\oc{u_{\sbs{1}}\bullet_k u_{\sbs{2}}}$ is a \emph{coproduct} of
$\oc{u_{\sbs{1}}}$ and $\oc{u_{\sbs{2}}}$, with several possible pairs of
coprojections \mbox{$\kappa_i:\doc{u_i} \el
\doc{u_{\sbs{1}}\bullet_k u_{\sbs{2}}}$}, $i=1,2$.

\medskip

\begin{remark} The choice of the finite set $\doc{u}$ is totally arbitrary,
apart from the fact that the number of its elements should equal that of
\emph{distinct} occurrences of indets in $u$ and we may, if we wish so,
\emph{reparametrize} $\oc{u}$, meaning that we replace its domain $\doc{u}$
by any subset of $\stn$ of the same cardinality.
\end{remark}

$\oc{u}$ is just an abstract object that carries the basic information about
the indets occurring in $u$ and the number of occurrences of each. We still
have to attach every $r \in \doc{u}$ to a particular occurrence of
$x=\oc{u}(r)$ in $u$. This is done with the help of an \emph{indet-occurrence
specification} or, in short, a \emph{specification} for $u$. Such a
specification is given by a $\stc$-term $t$ denoting $u$ (i.e. such that
$u=\eqv{t}$) together with a one-to-one function $\theta$ whose domain is
$\doc{u}$ and such that for each $r \in\doc{u}$, $\theta(r)$ will be a place
in the string of symbols $t$, in which $x=\oc{u}(r)$ occurs. This occurrence
will be referred to as the \emph{r-occurrence} of $x$ in $u$, \emph{as
specified by} $\theta$. We will denote $\theta:\doc{u} \el t$, to indicate
that $\theta$ is a specification as described.

As mentioned above, every $\stc$-proof $\pi$ of $t=s$ generates a bijection
between the indet occurrences in $t$ and those in $s$. This is a one-to-one
function $\chi=\chi_{\pi}$ which maps every location in the string of symbols
$t$ at which a certain indet occurs to a location in $s$ occupied by the same
indet.  We denote this situation by $\chi:t\el s$. We let the reader figure
out the obvious definition of $\chi_{\pi}$. With the help of this notion, we
can now define when two specifications for $u$ are the same.

\begin{definition}\label{D:speceqv} Two specifications $\theta_i:
\doc{u} \el t_i$, $i=1,2$, are called \emph{equivalent} if there exists a
$\stc$-proof $\pi$ of $t_{\sbs{1}}=t_{\sbs{2}}$ such that
$\theta_{\sbs{2}}=\chi_{\pi} \theta_{\sbs{1}}$.
\end{definition}

\begin{remarks} 1. If every indet that occurs in $u$, occurs there precisely
\emph{once}, then we have a unique specification $\theta:\doc{u} \el t$, for
every $t$ denoting $u$. If this is the case, then any two specifications for
$u$ are equivalent. If, however, there are indets with multiple occurrences
in $u$, then there are several possible specifications for $u$ into the same
$t$. In this case, $u$ may have \emph{inequivalent} specifications. To show
how delicate the issue of indet occurrences may be, let us also mention that
we might have two \emph{distinct} specifications $\theta_i:\doc{u} \el t$
into the same $t$ that are \emph{equivalent}! Indeed, as remarked by Eckmann
and Hilton, if $a$ is a $0$-cell and $u,v:1_a \el 1_a$ are $2$-cells, then
one can prove that \[u\bullet_0 v=v\bullet_0 u=u\bullet_{\sbs{1}}
v=v\bullet_{\sbs{1}} u.\] Thus, if we let $x$ be a 2-indet with $dx=cx=1_a$,
then substituting $x$ for $u$ and $v$ in, e.g. the proof of the first
equality, we get a non-trivial $\stc$-proof $\pi$ of $x\bullet_0 x=x\bullet_0
x$ which yields a $\chi_{\pi}$ that interchanges the two occurrences of $x$.

2. An alternative, more picturesque and less formal, point of view is this. A
specification $\theta:\doc{u} \el t$ actually \emph{relabels} the distinct
occurrences of any indet $x$ by different symbols $x',x'',...$. In this way,
we transform $t$ into a term $t^{\star}$, all of whose indets have unique
occurrences. Any $\stc$-proof of $t=s$ yields a suitable relabelling
$s^{\star}$ and a $\stc$-proof of $t^{\star}=s^{\star}$. In this way, by
looking at the proof, we can follow the rearrangement in $s$ of the indets
that occur in $t$. \end{remarks}

\medskip

We now \emph{choose}, for every $n$-cell $u$ of $\ba$, a \emph{preferred}
specification $\theta_u:\doc{u} \el t_u$. From this point on, when speaking
of \emph{the} $r$-occurrence of $x=\oc{u}(r)$ in $u$, we will mean the
occurrence specified by $\theta_u$. Occasionally, we might have to use
another term $t$ denoting $u$, and in such a case, it should always be
considered together with a proof $\pi$ of $t_u=t$. Then, the above mentioned
$r$-occurrence in $u$ is also the one specified by the equivalent
specification $\theta=\chi_{\pi}\theta_u:\doc{u} \el t$. One typical context
in which such a situation occurs naturally, will now be described.

If $u_{\sbs{1}}$, $u_{\sbs{2}}$ are $k$-composable $n$-cells of $\ba[I]$,
then $u_{\sbs{1}}\bullet_k u_{\sbs{2}}$ is denoted by both
$t_{u_{\sbs{1}}\bullet_k u_{\sbs{2}}}$ and $t_{u_{\sbs{1}}}\bullet_k
t_{u_{\sbs{2}}}$. Let's write, for simplicity, $u_{\sbs{1}}\bullet_k
u_{\sbs{2}}=u$, $t_u=t$, $t_{u_i}=t_i$, $i=1,2$. Select a proof $\pi$ of the
equality $t=t_{\sbs{1}}\bullet_k t_{\sbs{2}}$. It will yield a map
$\chi_{\pi}: t \el t_{\sbs{1}}\bullet_k t_{\sbs{2}}$. Let $\iota_i$ be the
``embedding'' of the term $t_i$ into $t_{\sbs{1}}\bullet_k t_{\sbs{2}}$,
$i=1,2$. By this we mean that $\iota_i$ maps every location in the string of
symbols $t_i$ into the corresponding location in the the larger string
$t_{\sbs{1}}\bullet_k t_{\sbs{2}}$. Remember that
$\oc{u}=\oc{u_{\sbs{1}}\bullet_k u_{\sbs{2}}}$ is a coproduct of
$u_{\sbs{1}}$ and $u_{\sbs{2}}$ in the category $I^{\#}$. A pair of
coprojections $\kappa_i$, $i=1,2$, will be called \emph{appropriate} (with
respect to the selected proof $\pi$), if the following diagrams commute:
%$$\bfig \morphism(0,0)|a|/>/<1200,0>[{\doc{u_{\sbs{1}}}}`{t_{\sbs{1}}};{\theta_{\sbs{1}}}]
%\morphism(0,0)|l|/>/<0,-500>[{\doc{u_{\sbs{1}}}}`{\doc{u}};{\kappa_{\sbs{1}}}]
%\morphism(1200,0)|r|/>/<0,-500>[{t_{\sbs{1}}}`{t_{\sbs{1}}\bullet_kt_{\sbs{2}}};{\iota_{\sbs{1}}}]
%\morphism(0,-500)|a|/>/<600,0>[{\doc{u}}`{t};{\theta}]
%\morphism(600,-500)|a|/>/<600,0>[{t}`{t_{\sbs{1}}\bullet_kt_{\sbs{2}}};{\chi_{\pi}}]
%\morphism(0,-1000)|l|/>/<0,500>[{\doc{u_{\sbs{2}}}}`{\doc{u}};{\kappa_{\sbs{2}}}]
%\morphism(1200,-1000)|r|/>/<0,500>[{t_{\sbs{2}}}`{t_{\sbs{1}}\bullet_kt_{\sbs{2}}};{\iota_{\sbs{2}}}]
%\morphism(0,-1000)|b|/>/<1200,0>[{\doc{u_{\sbs{2}}}}`{t_{\sbs{2}}};{\theta_{\sbs{2}}}]
%\efig$$
$$\bfig \morphism(0,0)|a|/>/<1200,0>[{\doc{u_{\sbs{1}}}}`{t_{\sbs{1}}};{\theta_{\sbs{1}}}]
\morphism(0,0)|l|/>/<0,-400>[{\doc{u_{\sbs{1}}}}`{\doc{u}};{\kappa_{\sbs{1}}}]
\morphism(1200,0)|r|/>/<0,-400>[{t_{\sbs{1}}}`{t_{\sbs{1}}\bullet_kt_{\sbs{2}}};{\iota_{\sbs{1}}}]
\morphism(0,-400)|a|/>/<600,0>[{\doc{u}}`{t};{\theta}]
\morphism(600,-400)|a|/>/<600,0>[{t}`{t_{\sbs{1}}\bullet_kt_{\sbs{2}}};{\chi_{\pi}}]
\morphism(0,-800)|l|/>/<0,400>[{\doc{u_{\sbs{2}}}}`{\doc{u}};{\kappa_{\sbs{2}}}]
\morphism(1200,-800)|r|/>/<0,400>[{t_{\sbs{2}}}`{t_{\sbs{1}}\bullet_kt_{\sbs{2}}};{\iota_{\sbs{2}}}]
\morphism(0,-800)|b|/>/<1200,0>[{\doc{u_{\sbs{2}}}}`{t_{\sbs{2}}};{\theta_{\sbs{2}}}]
\efig$$ where $\theta, \,\theta_{\sbs{1}},\,\theta_{\sbs{2}}$, are the
\emph{preferred} specifications for $u,\,u_{\sbs{1}},\,u_{\sbs{2}}$. As all
maps in this diagram are one-to-one, we immediately conclude that, given
$u_{\sbs{1}}$, $u_{\sbs{2}}$ and $\pi$, there is a unique pair of appropriate
coprojections $\kappa_{\sbs{1}}$, $\kappa_{\sbs{2}}$. These coprojections
will relate each indet occurrence in $u_i$ to the corresponding one in
$u=u_{\sbs{1}}\bullet_ku_{\sbs{2}}$.

\medskip

\textbf{An important convention.} As we remarked already, we can
reparametrize any given $\do{u}$ by changing its domain at will. We will use
this flexibility and \emph{always assume} that, whenever we are considering
the cell $u\bullet_k v$, the index sets $\doc{u},\,\doc{v}$ were so chosen as
to be \emph{disjoint} and to have $\doc{u\bullet_kv}=\doc{u} \dot{\cup}
\doc{v}$ (where the customary notation ``$\dot{\cup}$'' comes to emphasize
that the two terms of the union are disjoint sets), with the \emph{inclusion}
maps of $\doc{u},\,\doc{v}$ in $\doc{u\bullet_kv}$ being \emph{appropriate}
coprojections. This convention will simplify notations in the sequel.

\section{Placed composition}\label{S:placedcomp}

In this section, we will assume that $\bb$ is an $n$-category freely
generated by a set $I$ of indets. This means that $n>0$ and $\bb=\ba[I]$,
where $\ba=\bb\sbs{n-1}$, the $(n-1)$th truncation of $\bb$ or else, $n=0$
and $I=B\sbs{0}$. Let $\bx$ be an $\omega$-category extending $\bb$.

We are going to describe several operations involving cells of dimension
$\geqslant n$ of the $\omega$-category $\bx$. The most important, for the
present article, is the operation of \emph{placed composition}, that will be
presented later in this section.

\medskip

The first operation to be described is the \emph{$n$-cell replacement}
operation. If $u$ is an $n$-cell of $\bx$, $r \in
\doc{u}$, $\oc{u}(r)=x \in I$ and if $v$ is any $n$-cell of $\bx$
parallel to $x$, then we can \emph{replace} the $r$-occurrence of the
$n$-indet $x$ in $u$ by the $n$-cell $v$, producing an $n$-cell $u\repr v$,
as result. Notice that $u\repr v$ is parallel to $u$. Let us recall that an
$n$-cell $v$ is parallel to $x$ iff $n>0$ and $dv=dx,\,cv=cx$ or else, $n=0$
(as any two $0$-cells are considered to be parallel).

We can generalize this operation by allowing $v$ to be any cell of
dimension$\geqslant n$, provided that, if $n>0$ then $d^{(n-1)}v=dx$ and
$c^{(n-1)}v=cx$. Indeed, if $u$ is an $n$-cell and $v$ an $m$-cell, where
$k<n<m$, such that $d^{(k)}u=c^{(k)}v$ it is customary to define $u\bullet_k
v=1^{(m)}_u\bullet_kv$. Similarly, $u\bullet_k v= u\bullet_k 1^{(m)}_v$, if
$u$ is of dimension $m$ and $v$ of dimension $n$. These operations that yield
an $m$-cell when applied to cells of dimensions $n$ and $m$, are called
\emph{whiskerings}. As any $n$-cell $u$ is obtained from indets by means of
compositions, we conclude that it makes sense to replace the $r$-occurrence
of $x$ in $u$, by any cell $v$ of $\bx$ of dimension $m\geqslant n$, provided
that $d^{(n-1)}v=dx$ and $c^{(n-1)}v=cx$. The result is an $m$-cell $u\repr
v$ and this kind of replacement will be called a \emph{generalized
whiskering} operation. The $n$-cell replacement operation is just the
generalized whiskering, restricted to $n$-cells.

For $n=0$, the generalized whiskering operations is trivial: if $u$ is a
$0$-cell, then it is an indet, and $u \repr v = v$, hence $u \repr -$ is the
identity function on the set of all cells of $\bx$.

For $n>0$, given \emph{parallel} $(n-1)$-cells $a,b \in X_{n-1}$ of $\bx$, we
let $\bx(a,b)$ be the $\omega$-category whose $k$-cells are those
$(n+k)$-cells $v \in X_{n+k}$ that satisfy $d^{(n-1)} v=a,\,c^{(n-1)} v=b$.
With this notation, we see that, for $u \in X_n$ and $r \in \doc{u},\,x=
\oc{u}(r)$, $u \repr{-}$ is a function from the set of cells of $\bx(dx,cx)$
to the cells of $\bx(du,cu)$. As a clue to a precise definition of this
function, we note that the following three conditions should be met.

\begin{enumerate} \item If $u$ is an $n$-cell of $\bx$, $r \in \doc{u}$ and
$\oc{u}(r)=x$ then $u\repr v$ is defined iff $d^{(n-1)}v=dx$ and
$c^{(n-1)}v=cx$. If this is the case, then $d^{(n-1)}(u\repr v)=du$ and
$c^{(n-1)}(u\repr v)=cu$.
\item If $u=x\in I$ (remember that we identified
$x$ with the $n$-cell $\eqv{x}$) then $u\repr v=v$ (where, of course,
$\doc{u}=\{r\}$).
\item $(u\sps{\prime}\bullet_ku\sps{\prime\prime})\repr v=\begin{cases}
(u\sps{\prime}\repr v)\bullet_k u\sps{\prime\prime}& \text{if } r\in
\doc{u\sps{\prime}}\\u\sps{\prime}\bullet_k (u\sps{\prime\prime}\repr v)&
\text{if } r\in
\doc{u\sps{\prime\prime}}\end{cases}$

(remember that, by the convention established at the end of
section~\ref{S:occurrence}, we have
\mbox{$\doc{u\sps{\prime}\bullet_ku\sps{\prime\prime}}=\doc{u\sps{\prime}}
\dot{\cup}
\doc{u\sps{\prime\prime}}$}).\end{enumerate}

We want to associate with every $n$-cell $u$ an indexed set of partial
functions $\langle u\repr - \rangle_{r\in\doc{u}}$ so as to have conditions
1-3 met. One might think that these conditions can be used to define the
partial function $u\repr-$ by recursion on the $n$-cell $u$. However, the
same composite $u$ might be represented in more than one way as a composition
of two other cells. Conditions 1-3 allow us to define, by recursion on the
\emph{$\stc$-term} $t$, a partial function $t\repr-$ and we still have to
show that all terms denoting a given $u$ yield the same function. This can be
done by induction on proofs. However, we prefer another route.

We will use the universal property of $\ba[I]$ (cf. theorem~\ref{L:freeext})
and construe the mapping \mbox{$u \mapsto \langle u\repr -
\rangle_{r\in\doc{u}}$} as a functor into an $n$-category $\bw$.

\begin{definition}\label{D:W} $\bw$ is the $n$-category satisfying
the following requirements: \begin{enumerate} \item $\bw_{n-1}=\ba$, i.e. the
$k$-cells of $\bw$ are those of $\ba$ for $k<n$. \item The $n$-cells of $\bw$
are pairs \mbox{$U=(u,\, \langle H_r\rangle_{r\in\doc{u}})$}, with $u$ an
$n$-cell of $\bx$ and $H_r$ a function from the set of cells of $\bx
(dx_r,cx_r)$ to the set of cells of $\bx(du,cu)$, where $x_r= \oc{u}(r)$. The
domain and codomain are $dU=du,\ cU=cu$. \item For $a \in A_{n-1}=W_{n-1}$,
the identity over $a$ in $\bw$ will be $(1_a,\langle \, \rangle)$ (where
$\langle\,\rangle$ is, of course, the empty indexed set of functions).
\item if $U$ is as above and $V=(v,\langle K_r\rangle_{r\in\doc{v}})$ is such
that \mbox{$d^{(k)}U=d^{(k)}u=c^{(k)}v=c^{(k)}V$}, then we have
\[U\bullet_k V= (u\bullet_k v,\langle L_r
\rangle_{r\in\doc{u\bullet_k v}}),\]where \[L_r(-)=\begin{cases}
H_r(-)\bullet_k v & \text{if } r\in \doc{u}\\u\bullet_k K_r(-)& \text{if }
r\in \doc{v} \end{cases}\] \end{enumerate}
\end{definition}

A straightforward verification shows that $\bw$ is, indeed, an $n$-category.

\medskip

For $x\in I$, seen as an $n$-cell of $\ba[I]$ with $\doc{x}=\{r\}$, we have
that $\varphi x =_{def} (x,\langle H_r \rangle)$ is an $n$-cell of $\bw$,
where $H_r=id\sbs{\bx(dx,cx)}$ is the identity function from $\bx(dx,cx)$ to
itself. Thus we defined a function $\varphi :I \el W_n$ and we have $d\varphi
x=dx,\, c\varphi x=cx$. By theorem~\ref{L:freeext}, there is a unique
$\omega$-functor $G:\ba[I] \el \bw$ such that $Ga=a$, for $a$ a cell of $\ba$
and $Gx=\varphi x$, for $x\in I$.

\begin{claim}\label{L:uingu} For every $n$-cell $u$ of $\ba[I]$, the first component of $Gu
\in W_n$ is $u$ itself. \end{claim}

\begin{proof} Let $\Pi:\bw \el \ba[I]$ be defined as $\Pi a=a$ for $a$ a cell
of $\ba=\bw_{n-1}$ and $\Pi U=u$ for $U= (u,\langle H_r \rangle_{r \in
\doc{u}})=u$. Then $\Pi$ is an $\omega$-functor, hence so is the composite
\mbox{$\Pi G:\ba[I] \el \ba[I]$} and we must have that $\Pi
G=\mathbf{1}_{\ba[I]}$, the identity $\omega$-functor on $\ba[I]$, because,
by~\ref{L:freeext} there is a unique functor $\ba[I] \el \ba[I]$ which is the
identity for the cells of $\ba$ and for the indets $x\in I$. \end{proof}

\begin{definition}\label{D:rep} For $u$ an $n$-cell of $\bx$, if $Gu= (u,\langle H_r \rangle_{r \in
\doc{u}})$, then we define $u\repr -=H_r(-)$, for $r \in \doc{u}$.
\end{definition}

It follows immediately that the partial functions $u\repr -$ satisfy
conditions 1-3 stipulated just before definition~\ref{D:W}. Actually,
conditions 1-3 determine these functions uniquely, as summed up in the
following statement.

\begin{theorem}\label{L:reprexists} Given $\bx_n=\ba[I]$, there exists a unique
system of partial functions $\{u\repr -: u\in X_n,\, r \in \doc{u} \}$
satisfying conditions 1-3.\end{theorem}

\begin{proof} The \emph{existence} of a system as stipulated has been just
proven, so we have only to prove \emph{uniqueness}. This done by
\emph{induction on $n$-cells}. Let us emphasize that, while \emph{definitions
by recursion} on $n$-cells cells require special caution, as we just saw,
\emph{proofs by induction} are unproblematic, as the set of $n$-cells of
$\bx$, being the same as the set of $n$-cells of $\ba[I]$, is the least that
contains the indets and the identity $n$-cells and is closed under
composition. We are now going to see a first instance of such a proof.

Assuming that $\{u\ast_r -: u \in X_n,\, r \in \doc{u} \}$ is another system
of functions satisfying 1-3, an induction on $u$ shows that $u\ast_r- =
u\repr-$. We leave the straightforward argument to the reader. Many more
instances of proofs by induction on cells will be met soon. \end{proof}

\begin{remark} All these involved statements are relevant for the case $n>0$
only. If $n=0$ then every $n$-cell is an indet and $u \repr -$ is always the
identity function. \end{remark}

\medskip

It is well known and easily seen that the whiskering operations are
\emph{functorial} in the following sense: if $x,\,v,\,u$ are $n$-cells such
that $u=x\bullet_k v$ for some $k<n$, then the function $-\bullet_k v:
\bx(dx,cx) \el \bx(du,cu)$ is an $\omega$-functor (and, of course, a similar
statement holds for $v\bullet_k-$). The same is true for \emph{generalized}
whiskering.

\begin{theorem}\label{L:reprfunctor} If $\bx_n=\ba[I]$, $u\in X_n$, $r
\in \doc{u}$ and $\oc{u}(r)=x$ then the function \mbox{$u \repr-: \bx(dx,cx)
\el \bx(du,cu)$} is an $\omega$-functor. \end{theorem}

\begin{proof} By induction on $u$.

If $u$ is an indet $x\in I$, then $u\repr-$ is an identity map and there is
nothing to prove.

$u$ cannot be an identity, as $\doc{u} \neq \emptyset$.

If $u=u\p \bullet_k u\pp$ and, say, $r \in \doc{u\p}$, then  $u\p \repr-$ is
an $\omega$-functor by the induction hypothesis, hence so is the composition
$u\repr- = (u\p \repr-) \bullet_k u\pp$ of the $\omega$-functors $-\bullet_k
u\pp$ and $u\p \repr-$.
\end{proof}

If $u,v$ are $n$-cells of $\bx$, then so is $u\repr v$, if defined. Again, an
easy proof by induction on $u$, will show that $\oc{u\repr v}$ is a coproduct
of $\oc{u}\setminus r$ (i.e. $\oc{u}$ restricted to $\doc{u}-\{r\}$) and~
$\oc{v}$. The coprojections of this coproduct are induced by those of the
$\bullet_k$ operations involved, and if we stick to our convention of
choosing disjoint index sets for the arguments of these composition
operations, we will always have that $\doc{u\repr v}=(\doc{u}-\{r\})
\dot{\cup} \doc{v}$, with the \emph{inclusion} maps being the \emph{induced}
coprojections. Again, this will greatly simplify notations in the sequel.

\begin{theorem}\label{L:rep} If $u$ is an
$n$-cell then:
\begin{enumerate}\item (``Commutativity'') If $r,\,q \in \doc{u}$, $r\neq q$
 such that $u\repr v,\,u\repq w$ are defined where $v$, $w$ are also $n$-cells,
 then $(u\repr v)\repq w=(u\repq w)\repr v$.
\item (``Associativity'') If $r\in\doc{u}$, and $u\repr v$ is defined, $v$ an $n$-cell,  $q\in \doc{v}$ and $v\repq
w$ is defined with $w$ a cell of dimension $\geqslant n$, then
\[(u\repr v)\repq w= u\repr (v\repq w).\]
\item (Identity rule) If $r\in\doc{u}$ and $\oc{u}(r)=x$ then $u\repr x=u$.
\end{enumerate} \end{theorem}

\begin{proof} By induction on $u$. We sketch the proofs of parts
\emph{1,2} and leave the proof of~\emph{3} to the reader.

Proof of part \emph{1}: As $\doc{u}$ is assumed to have at least two distinct
elements, $u$ is neither an indet nor an identity. Assume that
$u=u\sps{\prime}\bullet_k u\sps{\prime\prime}$. Then
$\doc{u}=\doc{u\sps{\prime}}\dot{\cup}\doc{u\sps{\prime\prime}}$. If $r,q$
belong to different summands, e.g. if $r\in
\doc{u\sps{\prime}},\,q\in\doc{u\sps{\prime\prime}}$, then both sides of the stipulated equality are
seen to be equal to $(u\sps{\prime}\repr v)\bullet_k(u\sps{\prime\prime}\repq
w)$ (this case doesn't require any induction hypothesis). If both $r$ and $q$
belong to the same summand, e.g. $r,q\in
\doc{u\sps{\prime}}$ then the statement follows from the induction hypothesis for $u\sps{\prime}$.

Proof of \emph{2}: If $u$ is an indet, then both sides equal $v\repq w$. If
$u=u\sps{\prime}\bullet_k u\sps{\prime\prime}$ and, say, $r\in
\doc{u\sps{\prime}}$ then the left side equals $((u\sps{\prime}\repr v)\repq
w)\bullet_k u\sps{\prime\prime}$, while the right one equals
$(u\sps{\prime}\repr (v\repq w))\bullet_ku\sps{\prime\prime}$ and the
statement follows from the induction hypothesis for $u\sps{\prime}$.
\end{proof}

Assume that, not only is $\bx_n=\bb$ a free extension of $\bb_{n-1}=\ba$, but
also $\bx_{n+1}$ is a  free extension of $\bx_n$. Let's say that $\bx_{n+1} =
\bx_n[J] =\bb[J]$, for a set $J$ of $(n+1)$-indets. This situation will be
encountered from section~\ref{S:comparing} on. If so, then we can define
generalized whiskering functors for $n$-cells, as well as for $(n+1)$-cells.
The following simple technical lemma, linking these two kinds of operations,
will be useful later.

\begin{lemma}\label{L:repqrepr} Assume that $\bx$ is an $\omega$-category as
just described. If we have $w \in X_n$, $q \in \doc{w}$, $u \in X_{n+1}$, $r
\in \doc{u}$ and $v$ is any cell of $\bx$ of dimension $m \geqslant n+1$ then
the following equality holds, provided that the expressions involved are
defined: \[ (w \repq u) \repr v = w \repq (u \repr v).\] \end{lemma}

\begin{proof} By induction on $w$. If $w$ is an $n$-indet, then $w \repq -$
is an identity functor, and there is nothing to prove. $w$ cannot be an
identity, as $\doc{w} \neq \emptyset$. Assume that $w = w\p \bullet_k w\pp$
and , e.g., $q \in \doc{w\p}$. Then \[ (w \repq u) \repr v = ((w\p \repq u)
\bullet_k w\pp) \repr v .\] By the induction hypothesis, $(w\p \repq u) \repr
v = w\p \repq (u \repr v)$, and we conclude \[=(w\p \repq (u \repr v))
\bullet_k w\pp = (w\p \bullet_k w\pp) \repq (u\repr v) = w \repq (u \repr
v).\] \end{proof}

We now go \emph{one dimension higher} and define the operations of
\emph{placed composition} that involve $(n+1)$-cells of $\bx$. Let $u$ be
such a cell. Its domain $du$ is an $n$-cell of $\bx$, hence of $\ba[I]$.
Assume that $r\in \doc{du}$ and $\oc{du}(r)=x\in I$. Schematically, the
situation may be represented as in the figure below, where we indicated the
$r$-occurrence of $x$ in $du$.

$$\bfig \morphism(0,0)|a|/>/<1200,0>[{}`{};{cu}]
\morphism(0,0)|b|/{@{}@/^-1em/_(0.8){du}}/<450,-400>[{}`{};{}]
\morphism(450,-400)|a|/@{{*}->}^(0.5){x}/<300,0>[{}`{};{}]
\morphism(750,-400)|b|/{@{>}@/^-1em/}/<470,455>[{}`{};{}]
\morphism(600,-300)|r|/=>/<0,250>[{}`{};{\,\,u}] \efig$$

Let, in addition, $v$ be another $(n+1)$-cell of $\bx$ with codomain $cv=x$.
The two cells can be represented as in the figure at left below and it is a
natural thought to combine the two cells into a single one, $u\circ_r v$,
whose domain will be $du\repr dv$, the result of replacing the $r$-occurrence
of $x$ in $du$ by $dv$. The new cell is represented schematically in the
figure at right and is called the \emph{placed composition of $u$ and $v$ at
$r$}.

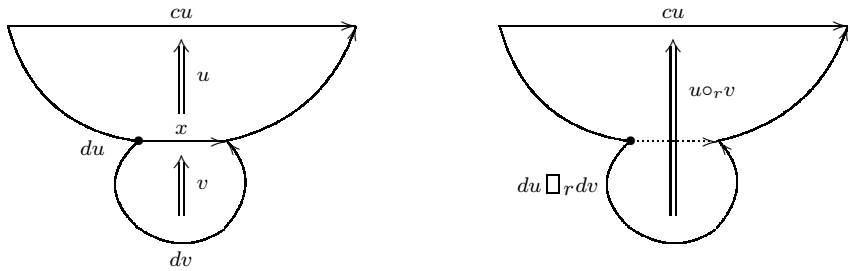
$$\bfig \morphism(0,0)|a|/>/<1200,0>[{}`{};{cu}]
\morphism(0,0)|b|/{@{}@/^-1em/_(0.8){du}}/<450,-400>[{}`{};{}]
\morphism(450,-400)|a|/@{{*}->}^(0.5){x}/<300,0>[{}`{};{}]
\morphism(750,-400)|b|/{@{>}@/^-1em/}/<470,455>[{}`{};{}]
\morphism(600,-300)|r|/=>/<0,250>[{}`{};{\,\,u}]
\morphism(450,-400)|l|/{@{}@/^-0.8em/}/<0,-300>[{}`{};{}]
\morphism(450,-700)|b|/{@{}@/^-0.6em/}/<325,15>[{}`{};{dv}]
\morphism(750,-700)|r|/{@{>}@/^-0.8em/}/<-18,325>[{}`{};{}]
\morphism(600,-650)|r|/=>/<0,200>[{}`{};{\,\,v}]
\morphism(1700,0)|a|/>/<1200,0>[{}`{};{cu}]
\morphism(1700,0)|b|/{@{}@/^-1em/}/<450,-400>[{}`{};{}]
\morphism(2150,-400)|a|/@{{*}.>}^(0.3){}/<300,0>[{}`{};{}]
\morphism(2450,-400)|b|/{@{>}@/^-1em/}/<470,455>[{}`{};{}]
\morphism(2300,-650)|r|/@{=>}_(0.7){\,\,u\circ_r v}/<0,600>[{}`{};{}]
\morphism(2150,-400)|l|/{@{}@/^-0.8em/}/<0,-300>[{}`{};{du\repr dv}]
\morphism(2150,-700)|b|/{@{}@/^-0.6em/}/<325,15>[{}`{};{}]
\morphism(2450,-700)|r|/{@{>}@/^-0.8em/}/<-18,325>[{}`{};{}] \efig$$

What is the precise definition of placed composition? The cells $u$ and $v$
cannot be composed as they are, because the domain of $u$ doesn't match the
codomain of $v$. This, however, can be corrected with the help of the
generalized whiskering functor $du\repr -$. Indeed, as we have $dv \to^v
cv=x=\oc{du}(r)$, we get, after applying $du\repr -$, \[ du\repr dv
\to^{du\repr v} du\repr cv=du\repr x=du \] and thus, $du\repr v$ is an
$(n+1)$-cell with codomain $du$, matching the domain of $u$. This motivates
the following definition:

\begin{definition}\label{D:placedcomp} For $u,\,v\in X_{n+1}$ with $r\in
\doc{du}$ and $\oc{du}(r)=x=cv$, we define the \emph{placed composition of
$u$ and $v$ at $r$} to be the $(n+1)$-cell \[u\circ_r v=u\bullet_n (du\repr
v) \] with domain $du\repr dv$ and codomain $cu$. \end{definition}

Again, we can generalize this operation further, by allowing $v$ to be any
$\bx$-cell of dimension $\geqslant n+1$ such that $c^{(n)}v=x$.
Definition~\ref{D:placedcomp} makes sense for such a $v$, with $\bullet_n$
indicating a whiskering, and produces a cell $u\circ_r v$, of dimension equal
to that of $v$, which will be called the placed \emph{whiskering} of $u$ and
$v$ at $r$.

\medskip

\noindent\emph{Remark concerning the case $n=0$.} In this situation, $du$ is
a $0$-cell, i.e. an indet, so that $du\repr -$ is the identity function,
$\doc{u}$ is a singleton, say $\{r\}$, and the placed composition is defined
only when $c\sps{(0)}v=du$ and we have, therefore, $u\repr v=u \bullet\sbs{0}
v$.

\medskip

The placed whiskering operations in general, and placed compositions in
particular, have properties similar to those of the operations of replacement
and generalized whiskering.

\begin{theorem}\label{L:pcomp}
If $u$ is an $(n+1)$-cell then:
\begin{enumerate}
\item(``Commutativity'') If $r,q\in
\doc{du}$, $r\neq q$ and $v,\,w$ are $(n+1)$-cells for which \mbox{$u\circ_r
v$}, $u\circ_q w$ are defined, then $(u\circ_r v)\circ_q w=(u\circ_q
w)\circ_r v$. \item (``Associativity'') If $r\in \doc{du}$, $v$ is an
$(n+1)$-cell such that $u\circ_r v$ is defined, $q\in \doc{dv}$ and $w$ is
any $\bx$-cell of dimension $\geqslant n+1$ with $v\circ_q w$ defined, then
$(u\circ_r v)\circ_q w=u\circ_r(v\circ_q w)$.
\item (Identity rules) If $\oc{du}(r)=x$ then $u\circ_r 1_x =u$. If $cv=x$ and
$\doc{x}=\{r\}$, then $1_x\circ_r v=v$.
\end{enumerate}
\end{theorem}

\begin{proof}
Proof of part \emph{1}: We have \begin{multline}\notag (u\circ_r v)\circ_q w
= (u\bullet_n (du\repr v)) \circ_q w= u\bullet_n (du\repr v)\bullet_n
(d(u\bullet_n(du\repr v))\repq w) =\\=u\bullet_n(du\repr
v)\bullet_n(d(du\repr v)\repq w) = u\bullet_n (du\repr v) \bullet_n ((du\repr
dv)\repq w) \end{multline} (remember that $du\repr-$ is functorial, therefore
$d(du\repr-)=du\repr d-$)

In the same way, $(u\circ_q w)\repr v=u\bullet_n(du\repq w) \bullet_n
((du\repq dw)\repr v)$, hence the desired conclusion will follow from the
following:

\begin{lemma}\label{L:vdw-wdv} If $u\sps{\prime}$ is an $n$-cell, $r,q\in \doc{u\sps{\prime}}$,
$r \neq q$ and $v,\,w$ are $(n+1)$-cells satisfying
$cv=\oc{u\sps{\prime}}(r)$, $cw=\oc{u\sps{\prime}}(q)$ then
\[ (u\sps{\prime}\repr v)\bullet_n((u\sps{\prime}\repr dv)\repq w)=(u\sps{\prime}\repq w)\bullet_n ((u\sps{\prime}\repq
dw)\repr v)
\]
\end{lemma}

\begin{proof} By induction on $u\sps{\prime}$. $u\sps{\prime}$ can be neither an indet nor an
identity, so assume that $u\sps{\prime}=u\sps{\prime}\sbs{1}\bullet_k
u\sps{\prime}\sbs{2}$, $k<n$.

Case 1: $r,\,q$ belong to the same one of
$\doc{u\sps{\prime}\sbs{1}},\,\doc{u\sps{\prime}\sbs{2}}$, e.g. $r,q\in
\doc{u\sps{\prime}\sbs{1}}$. Then
\begin{multline}\notag
(u\sps{\prime}\repr v)\bullet_n((u\sps{\prime}\repr dv)\repq w)=
((u\sps{\prime}\sbs{1}\repr v) \bullet_k
u\sps{\prime}\sbs{2})\bullet_n (((u\sps{\prime}\sbs{1}\repr dv)\repq w)\bullet_k u\sbs{2}\sps{\prime})=\\
=((u\sps{\prime}\sbs{1}\repr v)\bullet_n ((u\sps{\prime}\sbs{1}\repr dv)\repq
w))\bullet_k u\sps{\prime}\sbs{2}, \end{multline}
 where the second equality is just an instance of the exchange
axiom (axiom 3 of definition~\ref{D:axioms}). To see this, one should notice
that the $\bullet_k$ compositions stand for whiskerings and, therefore,
$u\sps{\prime}\sbs{2}$ is just short for $1_{u\sps{\prime}\sbs{2}}$.

Similarly, $(u\sps{\prime}\repq w)\bullet_n((u\sps{\prime}\repq dw)\repr v)=
((u\sps{\prime}\sbs{1}\repq w)\bullet_n((u\sps{\prime}\sbs{1}\repq dw)\repr
v))\bullet_k u\sps{\prime}\sbs{2}$ and the equality follows from the
induction assumption for $u\sps{\prime}\sbs{1}$.

Case 2: $r\in \doc{u\sps{\prime}\sbs{1}}$, $q\in\doc{u\sps{\prime}\sbs{2}}$.
Then,
\begin{multline} \notag (u\sps{\prime}\repr v)\bullet_n((u\sps{\prime}\repr dv)\repq w)=
((u\sps{\prime}\sbs{1}\bullet_k u\sps{\prime}\sbs{2})\repr
v)\bullet_n(((u\sps{\prime}\sbs{1}\bullet_ku\sps{\prime}\sbs{2})\repr
dv)\repq w)=\\=((u\sps{\prime}\sbs{1}\repr v)\bullet_k
u\sps{\prime}\sbs{2})\bullet_n ((u\sps{\prime}\sbs{1}\repr
dv)\bullet_k(u\sps{\prime}\sbs{2}\repq w)) = ((u\sps{\prime}\sbs{1}\repr
v)\bullet_n(u\sps{\prime}\sbs{1} \repr dv))\bullet_k
(u\sps{\prime}\sbs{2}\bullet_n (u\sps{\prime}\sbs{2}\repq w))=\\ =
(u\sps{\prime}\sbs{1} \repr v)\bullet_k (u\sps{\prime}\sbs{2} \repq w),
\end{multline} where, again, the equality before the last is an instance of
the exchange axiom, while the last equality follows by identity axioms (the
first line of axiom 4, definition~\ref{D:axioms}), taking into consideration
that $u\sps{\prime}\sbs{1} \repr dv,\, u\sps{\prime}\sbs{2}$ are just short
for $1_{u\sps{\prime}\sbs{1} \repr dv},\,1_{u\sps{\prime}\sbs{2}}$,
respectively .

A similar computation shows that $(u\sps{\prime}\repq w)\bullet_n
((u\sps{\prime}\repq dw)\repr v)$ equals $(u\sps{\prime}\sbs{1} \repr
v)\bullet_k (u\sps{\prime}\sbs{2} \repq w)$ as well. No need for any
induction hypothesis for this case.
\end{proof}

The proof of part \emph{1} is now complete.

\medskip

Proof of part \emph{2}: A computation shows that \[ u\sps{\prime}\circ_r(v
\circ_q w)= u\sps{\prime}\bullet_n (du\repr v) \bullet_n (du\repr (dv\repq
w)), \text{ while}
\]
\[ (u\sps{\prime}\repr v)\repq w= u\sps{\prime}\bullet_n (du\repr v) \bullet_n ((du\repr dv)
\repq w)
\] and the desired equality follows by part \emph{2} of theorem~\ref{L:rep}.

\medskip

The proof of \emph{3} is easy (for the first statement, one should only
notice that \\$du\repr 1_x=1_{du}$).
\end{proof}

\medskip

Theorem~\ref{L:reprexists} and definition~\ref{D:placedcomp} show that the
operations of placed composition are uniquely determined by the
$\omega$-categorical composition operations $\bullet_k$. The next statement
describes the behavior of a $\bullet_k$ composition operation when one of its
arguments is a placed composition. It will allow us to show, in
section~\ref{S:comparing}, that under certain conditions a \emph{converse}
also holds, namely, the placed compositions determine uniquely the
$\omega$-categorical ones.

\begin{proposition}\label{L:bulletviacirc}  If $\bx_n=\ba[I]$, then the
following identities hold, where $u, u\p, u\pp, v, v\p, v\pp$ are
$(n+1)$-cells of $\bx$ such that the left hand side expressions are defined,
then:
\begin{enumerate}
\item $u\bullet_k (v\p \circ_r v\pp)= (u \bullet_k v\p) \circ_r v\pp$, for
$k \leqslant n$. \item $(u\p \circ_r u\pp) \bullet_k v = (u\p \bullet_k v)
\circ_r u\pp$, for $k<n$.
\end{enumerate}\end{proposition}

\begin{proof} Part \emph{1}: As $c(v\p \circ_r v\pp)= cv\p$, we see that
$u\bullet_k v\p$ is defined and, as $\doc{u \bullet_k v} = \doc{u} \dot{\cup}
\doc{v\p}$, the right hand side expression is defined, whenever the left is.

If $k=n$, then we have $u\bullet_n (v\p \circ_r v\pp) = u\bullet_n (v\p
\bullet_n (d v\p \repr v\pp))= (u\bullet_n v\p) \bullet_n (dv\p \repr v\pp)$
and the desired identity follows once we notice that $dv\p = d(u\bullet_n
v\p)$.

If $k< n$, then $u \bullet_k (v\p \circ_r v\pp) = u\bullet_k (v\p \bullet_n
(dv\p \repr v)= (u\bullet_n 1\sbs{du}) \bullet_k (v\p \bullet_n (dv\p \repr
v)$. We can now use an instance of the exchange axiom and conclude that $u
\bullet_k (v\p \circ_r v\pp)= (u\bullet_k v\p) \bullet_n (1\sbs{du} \bullet_k
(dv\p \repr v\pp))= (u\bullet_k v\p) \bullet_n (du \bullet_k (dv\p \repr
v\pp))= (u\bullet_k v\p) \bullet_n ((du\bullet_k dv\p) \repr v\pp)$ (notice
that the second $\bullet_k$ in the third expression represents a whiskering)
and the desired identity follows if we notice that $du\bullet_k dv\p=
d(u\bullet_k v\p)$.

Part \emph{2}: To see that the right hand side is defined if the left is,
notice that $d(u\p \circ_r u\pp)= du\p \repr du\pp \parallel du\p$, hence
$d^{(k)}(u\p \circ_r u\pp)= d^{(k)} u\p$. The proof of the identity is
similar to that of the case $k<n$ of part \emph{1}. \end{proof}

\medskip

Theorems~\ref{L:pcomp} and~\ref{L:rep} point out common properties of the
placed composition operations on one hand, and the $n$-cell replacement ones,
on the other. Actually, these two families of operations are particular
instances of a general concept that forms the subject of the next section.

\section{Multicategories}\label{S:multicats} The notion of multicategory that
we are about to present, has been introduced in \cite{HMP2} and extends a
notion defined previously, under the same name, by Lambek (cf.
\cite{Lambek}). It is an abstract concept that, as we just hinted, displays
the common features of the placed composition operations, on one hand, and
the $n$-cell replacement ones, on the other.

A multicategory has a set of \emph{objects} and a set of \emph{arrows}. Each
arrow $u$ has a \emph{source} $Su$ and a \emph{target} $Tu$. $Su$ is an
indexed set of objects, a function from a finite set of indices $|Su|$ into
the set of objects. The multicategory has also partial
\emph{multicomposition} operations, which we denote $\odot_r$, $r$ being any
index. If $u$, $v$ are arrows then $u\odot_r v$ is defined whenever $r\in
|Su|$ and the target of $v$ is ``appropriate'' (in a sense to be made precise
shortly) for the object $Su(r)$ that occurs in the $r$-position in the source
of $u$. If such is the case, we will say that $v$ is \emph{multicomposable}
(or, $r$-\emph{multicomposable}) \emph{into $u$}.

One kind of examples of multicategories is based on the operations of placed
compositions playing the role of multicompositions. In this context, the
objects are the $n$-indets while the arrows are certain $(n+1)$-cells. The
source of an arrow $u$ will be $\oc{du}$ and its target will be $cu$. Thus,
the target of $v$ is ``appropriate'' for the object $Su(r)=\oc{du}(r)$ iff it
\emph{equals} it.

The situation is a bit different in a multicategory based on the $n$-cell
replacements. The objects are, again, the $n$-indets and the arrows are the
$n$-cells, the source of $u$ being $\oc{u}$. This time, the
$r$-multicomposition of $v$ into $u$ will be defined iff we have the equality
of ordered pairs $(dv,cv)=(dx,cx)$ where $x=Su(r)$. We will call $(dx,cx)$
the \emph{type} of the object $x$ and let the target of $v$ be $Tv=(dv,cv)$.
Hence, in this case, the target of $v$ is ``appropriate'' for  the object
$Su(r)=\oc{u}(r)$ iff it equals its type.

In preparation for a formal definition, let us specify a few conventions and
notations. As we mentioned already, given a set $O$, we let $O^{\#}$ be the
category whose objects are finite indexed sets of elements of $O$, i.e.
functions from finite subsets of a given infinite set of indices $\stn$, and
arrows defined in the obvious way (see also \cite{HMP2}). Recall that, given
an object $f$ of $O^{\#}$, $f:|f| \el O$, we allow ourselves to
\emph{reparametrize} $f$ replacing, at will, the domain $|f|$ by any subset
of $\stn$ of equal cardinality. To be more precise, if $s \subset \stn$ and
$\sigma: s \el |f|$ is a bijection, then we regard $f'=f\sigma: s \el O$ as
being \emph{the same} as $f$. Of course, when we do this, we also identify
the $O^{\#}$-arrows from and to $f$ with the corresponding maps (e.g.
$\gamma:g \el f$ should be identified with $\gamma'=\sigma^{\tc{-1}}\gamma:g
\el f'$). Finally, if $x\in O$, we let $\oc{x}$ be the object of $O^{\#}$
whose domain is a singleton and whose range is $\{x\}$.

\begin{definition}\label{D:multcat} A multicategory $\dcc$ consists of;
\begin{enumerate}\item An \emph{object system}, which is a triple
$\Omega=\Omega(\dcc)=(O,\,\dot{O},\,(-)^{\cdot})$ where $O$ is a set of
\emph{objects}, $\dot{O}$ a set of \emph{object types} and $(-)^{\cdot}: O
\el \dot{O}$ a map that associates with every $x \in O$ its \emph{type}
$\dot{x} \in \dot{O}$. We say that $\dcc$ is \emph{based} on $\Omega$. If $O
=\dot{O}$ and $(-)^{\cdot}$ is the identity, then $\Omega$ is called a
\emph{simple} object system and is denoted $\Omega = (O)$.
\item A set $A=A(\dcc)$ of \emph{arrows} together with \emph{source} and
\emph{target} functions \mbox{$S:A \el Ob(O^{\#})$}  and \mbox{$T:A \el
\dot{O}$}. \item Partial \emph{multicomposition operations} that associate
with each pair of arrows $u,\,v\in A$ and each $r\in |Su|$ such that
$Tv=(Su(r))^{\cdot}$, an arrow $u\odot_r v$ such that $S(u\odot_r v)$ is a
coproduct of $Su \setminus r$ and $Sv$ with \emph{specified coprojections}
and $T(u\odot_r v)=Tu$ (following our practice, we will always assume that
$Su,\,Sv$ have been so reparametrized as to have $|S(u\odot_r v)|=
(|Su|\setminus \{r\})\dot{\cup} |Sv|$ with the inclusion maps being the
specified coprojections).

$u\odot_r v$ will be referred to as the \emph{multicomposition of $v$ into
$u$ at place $r$}.
\item An \emph{identity} arrow $1_x$, for each $x\in O$, such that
$S(1_x)=\oc{x}$, $T(1_x)=\dot{x}$. \end{enumerate}

These components are subject to the following conditions:
\begin{description} \item[(a)] (Identity rules) If $Tu=\dot{x}$ then $1_x
\odot_r u=u$, where, of course, $|S1_x|=\{r\}$. If $Su(r)=x$, then $u\odot_r
1_x=u$. \item[(b)] (``Commutativity'') If $r,\,q\in |Su|$, $r\neq q$,
$Tv=(Su(r))^{\cdot}$ and $Tw=(Su(q))^{\cdot}$ then $(u\odot_r v)\odot_q
w=(u\odot_q w)\odot_r v$. \item[(c)] (``Associativity'') If $r\in |Su|$,
$Tv=(Su(r))^{\cdot}$, $q\in |Sv|$ and $Tw=(Sv(q))^{\cdot}$ then
\mbox{$(u\odot_r v)\odot_q w=u\odot_r (v \odot_q w)$}.
\end{description} \end{definition}

\medskip

We now reexamine the examples that motivated this definition.

As it turns out, there are \emph{two} important examples based on placed
composition.

\medskip

The \emph{first (and} main) \emph{example}: If $\bb$ is an $n$-category
generated by a set $I$ of indets, as we considered in
section~\ref{S:placedcomp}, and $\bx$ is an $(n+1)$-category extending $\bb$,
i.e. $\bx_n=\bb=\ba[I]$, then we define the multicategory $\dcc=\dcc_{\bx}$
of \emph{placed-composition}, whose object system is simple, with set of
objects $I$. The set of arrows will be $A=\{u:\,u\in X_{n+1},\,cu\in I\}$,
i.e. the set of those $(n+1)$-cells of $\bx$ that were called
\emph{many-to-one} in the introduction. For $u\in A$, $Su=\oc{du}$ and
$Tu=cu$. The multicomposition operation at place $r$ will be, of course,
$\circ_r$. Finally, for $x\in I$, the identity arrow will be the identity
\emph{cell} $1_x$.

\medskip

\emph{Remark concerning the terminology.} An arbitrary $(n+1)$-cell $u \in
X_{n+1}$ can be seen as \emph{linking} between the finite indexed sets of
$n$-indets $\oc{du}$ and $\oc{cu}$. In general, both these indexed sets have
(finitely) \emph{many} components. If it so happens that $cu \in I$, i.e.
$\oc{cu}$ contains just \emph{one} component, then it is only natural to say
that $u$ is a \emph{many-to-one} cell.

\medskip

A moment of thought will show that we do not \emph{have} to take the arrows
to be just the many-to-one $(n+1)$-cells. By deciding that \emph{all}
$(n+1)$-cells of $\bx$ are arrows we get another example of multicategory
based on placed composition.

\medskip

The  \emph{second example} of multicategory: We enlarge the
placed-composition multicategory $\dcc_{\bx}$ into an \emph{extended}
placed-composition multicategory $\dcc^+=\dcc^+_{\bx}$ whose set of objects
$O$ is still the set of $n$-indets $I$, but the set of arrows $A$ equals
$X_{n+1}$, the set of \emph{all} $(n+1)$-cells of $\bx$. To accommodate this
situation, the object system of $\dcc_{\bx}^+$ is \emph{not} simple anymore.
The set of object types is $\dot{O}=B_n$, the set of all $n$-cells of
$\bb=\ba[I]$ and $(-)^{\cdot}$ is the inclusion map. The source and the
target of $u$ are $Su=\oc{du}$ and $Tu=cu$. The multicomposition operations
and the identity arrows are defined as in the case of $\dcc_{\bx}$.

\medskip

The definition of $\dcc^+_{\bx}$ is made possible by the fact that, in the
abstract concept of multicategory, the map $(-)^{\cdot}:O \el \dot{O}$ is not
necessarily onto $\dot{O}$. Hence, we might have arrows whose target is not
the type of any object; such arrows cannot be multicomposed into any other
arrow (but, of course, other arrows can be multicomposed \emph{into} it).
This possibility was not ruled out in \cite{HMP2}, but it seems that it had
no relevance in that paper. It is, however, useful in the present work as the
notion of extended placed-composition multicategory will turn out to be
valuable in section~\ref{S:comparing} below.

\medskip

\noindent\emph{Remark concerning the case $n=0$.} In this case, $\bx$ is a
$1$-category (i.e., just an ordinary category) and all its $1$-cells are
many-to-one. Furthermore, as we remarked after definition~\ref{D:placedcomp},
placed composition is the same as categorical composition and so, we have in
this case that $\dcc_{\bx}=\dcc_{\bx}^+=\bx$. Hence, an ordinary category,
can be seen at the same time as a multicategory of a very particular kind.
Actually, the ordinary categories are precisely those multicategories whose
object system is simple and the source of any arrow is a singleton.

\medskip

We now turn to the replacement context.

\medskip

\emph{Third example}: Given $\bb=\ba[I]$ of dimension $n>0$, we construct a
multicategory $\dcr=\dcr_{\bb}$ of \emph{cell replacement} as follows. The
set of objects of $\dcr$ will be $O=I$, the set of $n$-indets. The set of
types $\dot{O}=\{(du,cu):\, u \text{ an } n\text{-cell of } \ba\}$ and for
$x\in O=I$, $\dot{x}=(dx,cx)$. The set of arrows $A$ will be $B_n$, the set
of $n$-cells of $\bb=\ba[I]$ and for $u\in A$, $Su=\oc{u}$, $Tu=(du,cu)$. The
placed multicomposition operation at $r$ will be $\repr$ and for $x\in O=I$,
the identity arrow $1_x$ will be $x$ itself.

\medskip

We now define the obvious notions of morphisms of object systems and of
multicategories.

\begin{definition}\label{D:mcmorph}\begin{enumerate} \item A \emph{morphism}
$\gamma: \Omega \el \Lambda$ between object systems
$\Omega=(O,\dot{O},(-)^{\cdot})$ and $\Lambda=(L,\dot{L},(-)^{\cdot})$ is a
pair of functions $\gamma=(\gamma\sbs{o}, \gamma\sbs{t})$, where
$\gamma\sbs{o}: O \el L$, $\gamma\sbs{t}: \dot{O} \el \dot{L}$ and we have,
for $x \in O$, $(\gamma\sbs{o} x)^{\cdot}=\gamma\sbs{t} \dot{x}$. Thus, if
$\Omega$ is simple, then $\gamma\sbs{o}=\gamma\sbs{t}$ and, if such is the
case, we denote $\gamma=\gamma\sbs{o}$. \item A \emph{morphism} $\chi :\dcc
\el \dcd$, where $\dcc,\,\dcd$ are multicategories, is a pair $\chi
=(\chi\sbs{\Omega}, \chi_a)$ such that:
\begin{description}
\item[i.] $\chi\sbs{\Omega}: \Omega(\dcc) \el \Omega(\dcd)$ is a morphism of
object systems. \item[ii.] $\chi\sbs{a}: A(\dcc) \el A(\dcd)$ and for each
$u\in A(\dcc)$, $\chi\sbs{a} Tu=T\chi\sbs{a} u$ and there is a bijection
$\theta_u: |Su| \el |S\chi\sbs{a} u|$ such that $Su=(S\chi\sbs{a} u)\theta_u$
(and we will usually assume that an appropriate reparametrization has been
made, so that $\theta_u$ is an identity map).
\item[iii.] If $u,v\in A(\dcc)$ and $u\odot_r v$ is defined, then
$\chi\sbs{a} (u\odot_r v)=(\chi\sbs{a} u)\odot_r(\chi\sbs{a} v)$ \item[iv.]
$\chi\sbs{a} 1_x = 1_{\chi\sbs{o}x}$, for $x\in O$.
\end{description}
\end{enumerate}
\end{definition}

\begin{remark} Stipulation \textbf{iii} has been made under the assumption that
the $\theta$ bijections of \textbf{ii} are identity maps. Otherwise, we have
to say that $\chi (u\odot_r v)=(\chi u)\odot_{r'} (\chi v)$, where
$r'=\theta_u r$ and must add obvious requirements concerning the links
between $\theta_u$, $\theta_v$, $\theta_{u\odot_r v}$ and the coprojections
related to the sources $S(u\odot_r v)$, $S((\chi u)\odot_r (\chi v))$. For
example, if the coprojections are inclusion maps, as we usually assume, then
we must just require that $\theta_{u\odot_r v}=\theta_u \dot{\cup}
\,\theta_v$.
\end{remark}

\section{Free multicategories}\label{S:planguage}
We follow a path analogous to the one taken in section~\ref{S:clanguage}. We
will design a language that allows to specify arrows built from given
\emph{indeterminates} by means of multicompositions in a multicategory. Given
an object system $\Omega =(O,\dot{O},(-)^{\cdot})$, let $J$ be a set of
\emph{arrow}-indeterminates, together with source and target functions $S:J
\el Ob(O^{\#})$, $T:J \el \dot{O}$. The elements of J will be also called
a-\emph{indets} or, simply, \emph{indets}, and will denote arbitrary arrows
in a multicategory based on $\Omega $. We will define an equational language
$\stm=\stm(\Omega ,J,S,T)$. The symbols of $\stm$ will be the a-indets, the
\emph{multicomposition} symbols $\odot_r$, for $r\in \stn$, the
\emph{identity} symbols $1_x$ for $x\in O$, as well as left and right
parentheses, as auxiliary symbols.

\begin{definition}\label{D:pterms} The set $\stt(\stm)$ of $\stm$-terms
and the source and target functions \newline\mbox{$S:\stt(\stm) \el
Ob(O^{\#})$}, $T:\stt(\stm) \el \dot{O}$ are defined as follows:
\begin{enumerate} \item Each indet $f\in J$ is an $\stm$-term with
$Sf,\,Tf$ as specified by the given source and target functions. \item For
each $x\in I$, $1_x$ is an $\stm$-term with $S1_x=\oc{x}$ and $T1_x=\dot{x}$.
\item If $t,\,s$ are $\stm$-terms and $r\in |St|$, $Ts=(St(r))^{\cdot}$, then
$(t)\odot_r(s)$ is an $\stm$-term (usually written just as $t\odot_r s$),
with $T(t\odot_r s)=Tt$ and $S(t\odot_rs)$ being a coproduct, with
\emph{specified coprojections}, of $St\setminus r$ and $Ss$. We will follow
our simplifying practice and assume that $St,\,Ss$ have been so
reparametrized as to have $|S(t\odot_rs)|=(|St|-\{r\}) \dot{\cup} |Ss|$, with
the inclusion maps being the specified coprojections.
\item There are no $\stm$-terms besides those mentioned in 1-3.
\end{enumerate}
\end{definition}

The semantics of the $\stm$-terms is analogous to that of the $\stc$-terms of
section~\ref{S:clanguage}. For $\dcc$ a multicategory based on $\Omega $ and
an assignment $\varphi:J \el A(\dcc)$ which is \emph{correct}, in the sense
that $S\varphi f=Sf,\,T\varphi f=Tf$, one defines the value
$val(t)=val_{\varphi}(t)\in A(\dcc)$ of any term $t\in \stt(\stm)$, under the
assignment $\varphi$. More generally, if $\gamma: \Omega \el \Omega(\dcc)$ is
a morphism of object structures for any multicategory $\dcc$ and $\varphi: J
\el A(\dcc)$ an assignment that is \emph{consistent} with $\gamma$ (in the
sense that $S\varphi f=\gamma Sf,\,T\varphi f=\gamma Tf$) , we can evaluate
$t$ \emph{under} $\gamma,\varphi$ and get $val_{\gamma,\varphi}(t) \in
A(\dcc)$. The definition of the evaluation function $val_{\gamma,\varphi}$ is
most natural and similar to definition~\ref{D:val}, so that we do not present
it formally.

Next, we define the \emph{axioms} and \emph{rules} of the equational logic
$\stm$ as we did in definition~\ref{D:axioms}:

\begin{definition}\label{D:paxioms} The deductive system
$\stm$ has the following axioms and rules, where, $t,s,w$ are arbitrary
$\stm$-terms and all multicompositions are supposed to be well defined
(according to definition~\ref{D:pterms}).

\vspace{3pt}

\emph{\textbf{Axioms.}}

\vspace{-9pt}

\begin{enumerate}
\item $t=t$ (equality axioms).
\item $1_x\odot_rt=t$ and
$t\odot_r1_x=t$ (identity axioms).
\item $(t\odot_rs)\odot_qw=(t\odot_qw)\odot_rs$, if $r\neq q$ (commutativity
axioms).
\item $(t\odot_rs)\odot_qw=t\odot_r(s\odot_qw)$ (associativity axioms).
\end{enumerate}

\vspace{-9pt}

 \emph{\textbf{Rules.}}

\vspace{-3pt}

\begin{enumerate}
\item $\dfrac{t=s}{s=t} \qquad \dfrac{t=s \quad s=w}{t=w}$ (equality rules).

\item $\dfrac{t=s}{t\odot_r w=s\odot_r w} \qquad
\dfrac{t=s}{w\odot_r t=w\odot_r s}$ (congruence rules).

\end{enumerate}
\end{definition}

Again, we will write `$\vdash t=s$' or, sometimes, `$\vdash_{\stm} t=s$', to
indicate that $t=s$ is provable in system $\stm$.

As in section~\ref{S:clanguage}, we are now able to prove the existence of
\emph{free} multicategories.

\begin{theorem}\label{L:freemc} Given $\Omega ,\,J,\,S,\,T$, there exists a
multicategory $\Omega [J]$ based on $\Omega $, with \newline\mbox{$J \subset
A(\Omega [J])$}, such that for $f\in J$, $Sf,\,Tf$ are the source and target
of $f$ in $\Omega [J]$ and the following \emph{universal property} holds:

Whenever $\dcc$ is a multicategory based on $\Omega $ and $\varphi:J \el
A(\dcc)$ a function such that $S\varphi f=Sf,\, T\varphi f=Tf$ for all $f\in
J$, there is a \emph{unique} morphism $\chi:\Omega [J] \el \dcc$ which is the
identity on objects and object-types and satisfies $\chi f=\varphi f$ for $f
\in J$.

\medskip

Moreover, $\Omega [J]$ has also the following \emph{strong universal
property}: whenever $\dcc$ is any multicategory, $\gamma: \Omega \el
\Omega(\dcc)$ a morphism of object systems and $\varphi:J \el A(\dcc)$ a
function such that $S\varphi f=\gamma Sf,\,T\varphi f=\gamma Tf$, there is a
\emph{unique} morphism $\chi:\Omega [J] \el \dcc$ extending both, $\gamma$
and $\varphi$ in the sense that $\chi\sbs{\Omega}=\gamma$ and $\chi\sbs{a}
f=\varphi f$ for $f\in J$.
\end{theorem}

\begin{remark} Here we used abbreviated notations, that will be adopted in
the sequel. We wrote just $\chi$ for $\chi\sbs{a}$ and , likewise, $\gamma$
for $\gamma\sbs{o}$ or $\gamma\sbs{t}$, as the subscripts are understood for
the context. Also, when applying a function to a finite sequence (like in
$\gamma Sf$), we understand that the function is applied to each component of
the sequence. \end{remark}

\noindent\emph{First proof} (Sketch). As in the proof of~\ref{L:freeext}, we
define, for $\stm$-terms $t,\,s$, $t\approx s$ iff $\vdash t=s$, and take the
arrows of $\Omega [J]$ to be equivalence classes $\eqv{t}$, of $\stm$-terms,
identifying $f\in J$ with $\eqv{f}$. The details are similar to those of the
proof of~\ref{L:freeext}. In particular,
$\chi(\eqv{t})=val_{\gamma,\varphi}(t)$. \qed

The multicategory $\Omega [J]$ will be called \emph{free} or, more
specifically, \emph{freely generated by $J$ over $\Omega $}. This terminology
is justified, as both universal properties show that $\Omega [J]$ is a free
object with respect to suitable functors $U$, in the sense described in the
introduction.

\medskip

\emph{\textbf{An important example.}} Let $\Omega\sbs{0} = (\{0,1\})$ be the
simple object system having $\{0,1\}$ as set of objects and object types.
Given \emph{any} set $J$, make $J$ it into a set of a-indets over
$\Omega\sbs{0}$ by letting $Sx = \langle 0 \rangle$ and $Tx = 1$, for each $x
\in J$, and consider the multicategory $\Omega\sbs{0}[J]$. A moment of
thought will show that there are no non trivial arrow compositions in this
multicategory and hence its set of arrows will contain, besides the two
identity arrows $1\sbs{0}, 1\sbs{1}$, only the elements of $J$. We can,
therefore, identify the set $J$ with the free multicategory
$\Omega\sbs{0}[J]$.  Hence, \emph{any barren set can be viewed as a free
multicategory}.

\medskip

The notion of a-\emph{indet occurrence in an arrow} $u \in A(\Omega[J]$, can
be developed precisely as we did in section~\ref{S:occurrence} for the
similar notion of indet occurrence in an $n$-cell of an $n$-category which is
a free extension of its $(n-1)$th truncation. Thus, each $u$ as above has a
finite indexed set $\oc{u}: \doc{u} \el J$ of a-indet occurrences and $\oc{u
\odot_r v}$ is a coproduct of $\oc{u}$ and $\oc{v}$ with specified
\emph{appropriate} coprojections and we will always assume that the index
sets $\doc{u},\, \doc{v}$ were so chosen as to have $\doc{u \odot_r v} =
\doc{u} \dot{\cup} \doc{v}$, with the inclusion maps being the appropriate
coprojections.

As we mentioned in the introduction, there is, \emph{however}, a basic
difference between free extensions, on one hand, and free multicategories on
the other. The latter is \emph{simpler}, in the sense that the free
multicategory $\Omega [J]$ can also be described as a true \emph{term} model,
whose arrows are certain terms (and \emph{not} equivalence classes of terms)
in `Polish' notation. This is the way free multicategories are constructed
in~\cite{HMP2} and we reproduce the description here.

\medskip

\noindent\emph{Second proof of~\ref{L:freemc}} (Sketch). The arrows of
$\Omega [J]$ will be certain strings of elements of $O \dot{\cup} J$.
\emph{For the following construction only}, it will be useful to depart from
the convention adopted elsewhere in this paper and to assume, first, that the
index set $\stn$ is the set of natural numbers and, second, that for an arrow
$u$, the finite set $|Su|$ will always be of the form $[k]=\{0,..,k-1\}$, for
some natural number $k$ (thus, the objects of $O^{\#}$ will be strings of
symbols (i.e. elements) from $O$). We also assume that each $f\in J$ has a
uniquely specified source, with no reparametrizations allowed. By the way,
theses are the conventions adopted throughout \cite{HMP2}. As a result, the
specified coprojections associated with multicompositions will no longer be
assumed to be inclusion maps.

\medskip

\textbf{Definition} of $\sta=A(\Omega [J])$ and of the target function T:
\begin{enumerate}
\item If $x\in O$ then $x\in \sta$ and $Tx=\dot{x}$.
\item If $f\in J$, $|Sf|=[k]$, $u_r \in \sta$ and $Tu_r=(Sf(r))^{\cdot}$ for
$r<k$, then $u=fu\sbs{0}u\sbs{1}..u\sbs{k-1} \in \sta$ and $Tu=Tf$ (here, $u$
is the concatenation of the one symbol string $f$ and the strings
$u\sbs{0},\, u\sbs{1},\,..,u\sbs{k-1}$). \item There are no arrows in $\sta$
besides those mentioned in 1-2. \end{enumerate}

\medskip

The elements of $\sta$ will sometimes be called \emph{reduced $\stm$-terms}
or, simply, \emph{reduced terms}.

\medskip

\textbf{Definition} of the source function, multicomposition and identity
arrows:

For $u\in \sta$, $Su$ will be the \emph{substring} of $u$ consisting of the
$O$-symbols only.

If $u,v \in \sta$, $Su(r)=x\in O$ and $Tv=\dot{x}$, then the $r$th $O$-symbol
occurrence in the string $u$ is an occurrence of $x$ and $u\odot_rv$ will be
the string obtained from $u$ by substituting the said occurrence of $x$ by an
occurrence of $v$. Thus, if $u=u^{\tc{\prime}} x u^{\tc{\prime\prime}}$ with
$x$ indicating the said $O$-symbol occurrence, then
$u\odot_rv=u^{\tc{\prime}} v u^{\tc{\prime\prime}}$ (this explicit way of
writing, should be useful when checking that the multicategory laws are
fulfilled for this definition). The specified coprojections associated with
this multicomposition are obvious.

Finally, for $x\in O$, $1_x$ will be $x$ itself.

\medskip

We leave the reader the tedious but routine task of checking that we did,
indeed, construct a multicategory.

In order to have $J\subset \sta$, we have to identify $f\in J$ with
$fx\sbs{0}x\sbs{1}..x\sbs{k-1}$, where $x\sbs{r}=Sf(r)$ for $r<k$.

Finally, the universal property of $\Omega [J]$ is also routinely checked,
using the fact that
$fu\sbs{0}u\sbs{1}..u\sbs{k-1}=(..((f\odot\sbs{k-1}u\sbs{k-1})\odot\sbs{k-2}u\sbs{k-2})
..)\odot\sbs{0}u\sbs{0}$. \qed

\medskip

As an immediate corollary of this second proof of~\ref{L:freemc}, we conclude
a simple but important statement. We say that a multicategory $\dcc$ is a
\emph{submulticategory} of $\dcc^{\prime}$, $\dcc \subset \dcc^{\prime}$, iff
\mbox{$O(\dcc) \subset O(\dcc^{\prime})$},
$\dot{O}(\dcc)\subset\dot{O}(\dcc^{\prime})$, $A(\dcc) \subset
A(\dcc^{\prime})$ and the inclusion maps of the components of $\dcc$ into
those of $\dcc^{\prime}$ form a multicategory morphism $\chi:\dcc \el
\dcc^{\prime}$. We also say, in such a situation, that $\Omega(\dcc)$ is an
object \emph{subsystem} of $\Omega(\dcc^{\prime})$, $\Omega(\dcc) \subset
\Omega(\dcc^{\prime})$.

\begin{proposition}\label{L:submc} If $\Omega \subset \Omega^{\prime}$ and
$J,\, J^{\prime}$ are sets of a-indets over $\Omega ,\, \Omega^{\prime}$ such
that $J\subset J\p$ and  the source and target functions on $J$ are the
restrictions of those on $J^{\prime}$, then $\Omega[J] \subset
\Omega[J^{\prime}]$.
\end{proposition}

Strictly speaking, the $\sta$-terms are \emph{not} $\stm$-terms, but can be
easily \emph{translated} into terms of the latter kind. Indeed, the last
remark of the second proof of~\ref{L:freemc} implies that each $u \in \sta$
is the value $val_{\varphi}(u\sps{\star})$ of a recursively defined
$\stm$-term $u\sps{\star}$, where $\varphi: J \el \sta$ is the inclusion map
of $J$ into $\sta$ (actually, the map $u \mapsto u\sps{\star}$ is primitive
recursive).

The $\stm$-terms $u\sps{\star}$ are of a special form. Call an $\stm$-term
$t$ \emph{normal}, if $t$ is an identity term or else, is of the form
\[t=(..((f\odot\sbs{k-1}t\sbs{k-1})\odot\sbs{k-2}t\sbs{k-2})
..)\odot\sbs{0}t\sbs{0}\] with $f\in J$, $|Sf|=[k]$ and $t\sbs{0},\,
..,t\sbs{k-2},\,t\sbs{k-1}$ normal terms (we still cling to the convention of
the second proof of~\ref{L:freemc}, according to which the index sets are
initial segments of the natural numbers, and each a-indet has a uniquely
specified source). Obviously, $u\sps{\star}$ is a normal $\stm$-term for all
$u \in \sta$. Conversely, every normal $\stm$-term can be seen to be
$u\sps{\star}$ for a \emph{unique} $u \in \sta$. Thus, the free multicategory
$\Omega[J]$ can be described as a \emph{term} model whose arrows are the
normal $\stm$-terms.

It follows that every $\stm$-term $t$ is $\stm$-provably equivalent to a
unique normal term $\hat{t}$ (namely, the only normal term satisfying
$\hat{t} \in \eqv{t}$). It is not hard to establish this fact directly and to
show that the function $t \mapsto \hat{t}$ is primitive recursive.
Incidentally, this implies that we have a primitive recursive algorithm for
deciding whether $t=s$ is $\stm$-provable or not, for given $t,s$. This fact
is usually described as saying that \emph{the word problem for $\stm$} is
decidable.

These circumstances allow a simpler treatment of the notion of a-indet
occurrence, as we can define $\oc{u}$ \emph{canonically}, as the sequence of
a-indets arranged in the order in which they occur in the unique
\emph{normal} $\stm$-term that denotes $u$. Still, we prefer to think of
$\doc{u}$ as a finite indexed set with domain $\doc{u} \subset \stn$, which
can be reparametrized to our convenience.

We now return to the analogy that exists, nevertheless, between  free
extensions of \mbox{$(n-1)$-categories} on one hand, and free multicategories
on the other. Given an arrow $u \in A(\Omega[J])$ and $r \in \doc{u}$, with
$\oc{u}(r)=f \in J$, if $v$ is another arrow such that $Sv=Sf,\,Tv=Tf$, we
can \emph{replace} the $r$-occurrence of $f$ in $u$ by an occurrence of $v$
and get an arrow $u \drep_r v$. The precise definition is worked our
similarly to that of cell replacement, as done in section~\ref{S:placedcomp}.

\begin{theorem}\label{L:arrowrep} There is a unique system $\{ u \drep_r -: u
\in A(\Omega[J]),\, r\in \doc{u}\}$ of partial functions , satisfying the
following conditions:
\begin{enumerate} \item If $\oc{u}(r)=f \in J$ then $u \drep_r v$ is defined
iff $v \parallel f$, meaning that $Sv = Sf,\, Tv = Tf$. If this is the case,
then $u\drep_r v \in A(\Omega[J])$ and $S(u\drep_r v)=Su,\,T(u\drep_r v)=Tu$.
\item If $u=f \in J$ and $\doc{u}=\{r\}$, then $u \drep_r v=v$. \item If $u=
u\p \odot_j u\pp$  then \[ u\drep_r v= \begin{cases} (u\p\drep_r v)\odot_j
u\pp& \text{if } r\in \doc{u\p}\\ u\p\odot_j (u\pp \drep_r v)& \text{if } r
\in
\doc{u\pp} \end{cases}\] \end{enumerate} \end{theorem}

\begin{proof} (Sketch) The uniqueness is easily seen by induction on $u$.

Let us use the following notations: $A=A(\Omega[J])$ and $A(Su,Tu)=\{v \in A:
v \parallel u\}$, for $u \in A$. We construe the function $u \mapsto \langle
u\drep_r - \rangle \sbs{r\in
\doc{u}}$ as a morphism into a multicategory $\dcw$, whose definition is
based on the idea that was used also in definition~\ref{D:W}:
\begin{enumerate}
\item The object system is $\Omega(\dcw)= \Omega$.\item  The arrows are
pairs $U=(u, \langle H_r \rangle \sbs{r \in \doc{u}})$, where $H_r: A(Sf_r,
Tf_r) \el A(Su,Tu)$, $f_r= \oc{u}(r)$. Also, $SU= Su,\, TU=Tu$.
\item If $x \in O$ then the identity arrow over $x$ in $\dcw$ is $(1_x,
\langle \rangle)$.
\item If $U$ is as above and $V= (v,\langle K_r \rangle\sbs{r\in
\doc{v}})$, $j \in |SU|=|Su|$ and $TV=Tv= (Su(j))^{\cdot}=(SU(j))^{\cdot}$ then $U
\odot_j V = (u\odot_j v, \langle L_r \rangle\sbs{r\in \doc{u\odot_j v}})$
where \[ L_r(-)=
\begin{cases} H_r(-)\odot_j v& \text{if } r \in \doc{u}\\ u\odot_j K_r(-)&
\text{if } r \in
\doc{v} \end{cases}\]
\end{enumerate}

It is easy to verify that $\dcw$ is, indeed, a multicategory. We can define
$\varphi:J \el A(\dcw)$ by letting $\varphi f =(f,\langle H_r \rangle)$,
where $\doc{f}=\{ r \}$ and $H_r= id\sbs{A(Sf,Tf)}$, the identity map of
$A(Sf,Tf)$ onto itself. By the universal property of $\Omega[J]$, there is a
unique morphism $\chi: \Omega[J] \el \dcw$ which is the identity on the
object system and extends $\varphi$. As in~\ref{L:uingu}, we see that for $u
\in A$, we have $\chi_a u= (u, \langle H_r \rangle\sbs{r \in \doc{u}})$ and
we define $u\drep_r-= H_r(-)$.
\end{proof}

Given $\Omega$ and $J$ as above, one can define a multicategory $\dcd=
\dcd\sbs{\Omega,J}$ of \emph{arrow replacement} as follows:

$\Omega(\dcd)=(O\sbs{\dcd},\dot{O}\sbs{\dcd},(-)^{\cdot}\sbs{\dcd})$, where
$O\sbs{\dcd}=J$, $\dot{O}\sbs{\dcd}= \{(Su,Tu):u \in A\}$ and $\dot{f} =
(Sf,Tf)$.

The arrows of $\dcd$ are those of $\Omega[J]$, while the source and target
functions are defined by $S\sbs{\dcd}u=\oc{u},\, T\sbs{\dcd}u= (Su,Tu)$. The
multicomposition operation at $r \in \doc{u}$ is $u\drep_r-$ and the identity
arrow over $f \in J$ is $f$ itself.

The proof that $\dcd$ is a multicategory is similar to that of
theorem~\ref{L:rep}.

\medskip

A morphism $\chi: \Omega[J] \el \Omega\p[J\p]$ between \emph{free}
multicategories is said to be \emph{indet preserving} if $\chi f \in J\p$
whenever $f \in J$.  If $\chi$ is such a morphism then it easy to see that,
for every $u \in A(\Omega[J])$ there is a bijection $\theta: \doc{u} \el
\doc{\chi u}$ such that $\chi(\oc{u}(r)) = \oc{\chi u}(\theta r)$. We will
assume that an appropriate reparametrization was made such that $\doc{u} =
\doc{\chi u}$ and $\theta$ is the identity. If so, then we have the following
useful statement:

\begin{proposition}\label{L:arrowreppres} If a morphism $\chi: \Omega[J] \el
\Omega\p[J\p]$ preserves indets, then it preserves also arrow replacement.
This means that for $u,v \in A(\Omega[J]$, if $u\drep_r v$ is defined the so
is $(\chi u) \drep_r (\chi v)$ and $\chi(u\drep_r v) = (\chi u) \drep_r (\chi
v)$. \end{proposition}

\begin{proof} A straightforward induction on $u$. \end{proof}

We now return to the comparison between the languages of composition and
multicomposition. As we saw, $\stm$-terms have normal forms and two terms are
$\stm$-provably equal iff they have the same normal form. Is a similar result
true for $\stc$-terms? It does not seem to be so, especially in view of
\cite{M}. However, in the restricted \emph{many-to-one} situation, the $\stc$
and $\stm$ equational logics can be linked to each other in a beneficial way
that displays useful similarities. This is the subject of the next section.

\section{Comparing $\stm$ and $\stc$ in the many-to-one
case}\label{S:comparing} Consider, again, an $n$-category $\bb$ generated by
a set $I$ of $n$-indets. In this section we make the following

\medskip

\noindent\textbf{Assumption.} $J$ is a set of many-to-one indets over
$\bb=\ba[I]$. In other words, $J$ is a set together with domain and codomain
functions $d,c:J \el B_n$ such that $cf \in I$ for all $f \in J$ (and, of
course, $df \parallel cf$).

\medskip

Thus, the indets in $J$ denote arbitrary many-to-one cells in
$\omega$-categories extending $\bb$. Once we have such a $J$, we can
construct three distinct structures:

\emph{First}, there is the free $(n+1)$-category $\bx=\bb[J]$, which is the
$n$-category $\bb$ augmented by the set $X_{n+1}$ of the $(n+1)$-cells
generated from $J$.

\emph{Second}, we have the multicategory $\dcc_{\bx}$ based on the the simple
object system $\Omega $ with set of objects $O=\dot{O}=I$. The arrows of
$\dcc_{\bx}$ are, as we recall, the \emph{many-to-one} $(n+1)$-cells of of
$\bx$ and the source and target functions are $Su= \oc{du},\,Tu= cu$. In
particular, \emph{all} indets $f \in J$ are arrows of $\dcc_{\bx}$.

\emph{Finally}, we construct the free multicategory $\Omega [J]$ generated by
$J$ over the \emph{same} object system $\Omega $ on which $\dcc_{\bx}$ is
based. The arrows of $\Omega [J]$ can be construed either as equivalence
classes $\eqv{t}$ of $\stm$-terms or, else, as reduced $\stm$-terms $u \in
\sta$.

By the universal property of $\Omega [J]$, there is a unique morphism $\chi:
\Omega [J] \el \dcc_{\bx}$ which is the identity on both, the set of objects
(and object-types) $O$ and the set of indets $J$. This map deserves a closer
look. As remarked at the end of the proof of~\ref{L:freemc}, for any
$\stm$-term $t$, $\chi(\eqv{t})=val\sbs{i\sbs{J}}(t)$, where $i\sbs{J}$ is
the inclusion map of $J$ into the set of arrows of $\dcc_{\bx}$, which is
nothing but the set of many-to-one $(n+1)$-cells of $\bx$. Thus, $\chi$ maps
every arrow of $\Omega [J]$, which is described by an $\stm$-term, to a
many-to-one $(n+1)$-cell of $\bx$, which is described by a $\stc$-term.
Actually, by carefully following the proofs of~\ref{L:freeext}
and~\ref{L:freemc}, one can exhibit a primitive recursive function that sends
every term $t \in \stt(\stm)$ to a term $\tilde{t} \in \stt(\stc)$ such that
$\chi(\eqv{t}) = \eqv{\tilde{t}}$. The function $t \mapsto \tilde{t}$ is,
therefore, a \emph{translation} of $\stm$-terms into $\stc$-terms.

The considerations above point to the fact that the map $\chi$ is a very
important one. It deserves a special notation and name.

\begin{notation}  If $\chi:\Omega [J] \el \dcc_{\bx}$ is the unique morphism of
multicategories that is the identity on $O=I$ and on $J$, then we denote
$\chi=\val{-}$. This morphism will be referred to as \emph{the canonical
morphism of $\Omega[J]$ into $\dcc_{\bx}$}.
\end{notation}

Thus, we have $\chi u=\chi\sbs{a} u=\val{u}$ for $u \in A(\Omega [J])$ and
$\val{x} =x,\,\val{f}=f$ for $x \in I,\,f \in J$.

As we remarked in section~\ref{S:multicats}, if $n=0$ then the category $\bx$
is the same as the multicategory $\dcc_{\bx}$ and, as in the present case
$\bx$ is a \emph{free} category, it is also identical with the free
multicategory $\Omega[J]$. Moreover, the canonical morphism $\val{-}$ is the
identity map.

In the case $n>0$, however, the situation is much more complex and
interesting. Not every $(n+1)$-cell of $\bx$ is of the form $\val{u}$ for
some arrow $u$ of $\Omega [J]$, simply because the latter is always a
many-to-one cell. But are all many-to-one $(n+1)$-cells of $\bx$ of the form
$\val{u}$? Furthermore, is the $\val{-}$ map one-to-one? In other words, is
$\val{u} \neq \val{u\sps{\prime}}$ whenever $u \neq u\sps{\prime}$? The
answer to both these questions is positive, as it follows from the following
statement which is the main technical result of this paper:

\begin{theorem}\label{L:main} $\val{-}:\Omega[J] \el \dcc_{\bx}$ is an
isomorphism of multicategories. \end{theorem}

Thus, if $\bx$ is an $(n+1)$-category \emph{freely generated} by a set $J$,
then $\dcc_{\bx}$ is a multicategory \emph{freely generated} by the same set
$J$. As a result, we have the following corollary that will be extremely
useful in the sequel.

\begin{corollary}\label{L:cxyieldsx} Assume that $\bx$, $\bb$ and $J$ are as
above. If $\bz$ is any other $(n+1)$-category extending $\bb$ and $\chi:
\dcc_{\bx} \el \dcc_{\bz}$ is a morphism of multicategories which is the
identity on objects and satisfies, for all $x\in J$, $d\chi\sbs{a}
x=dx,\,c\chi\sbs{a} x=cx$, then there is a \emph{unique} $\omega$-functor $F:
\bx \el \bz$ which is the identity on the cells of $\bb$ and extends $\chi$,
in the sense that $Fu= \chi\sbs{a} u$ whenever $u$ is a many-to-one
$(n+1)$-cell of $\bx$ (which means that $u$ is also an arrow of
$\dcc_{\bx}$). If $\bz$ is also a free extension of $\bb$ and $\chi$ is an
isomorphism, then $F$ is an isomorphism as well.
\end{corollary}

The significance of the last statement of this corollary is that in a free
extension $\bx$ of $\bb$ generated by many-to-one indets, the
\emph{many-to-one} $(n+1)$-cells of $\bx$ (i.e. the arrows of $\dcc_{\bx}$)
determine the entire $(n+1)$-cell structure of $\bx$.

\begin{proof} Due to the freeness of the $(n+1)$-category $\bx$, there is a
unique $\omega$-functor $F: \bx \el \bz$ which is the identity on the
$\bb$-cells and such that $Fx = \chi\sbs{a} x$ for $x \in J$. All we have to
show is that $F$ extends $\chi\sbs{a}$ on \emph{all} many-to-one $(n+1)$
cells of $\bx$. As these cells are also the arrows of $\dcc_{\bx}$ and,
by~\ref{L:main}, $\dcc_{\bx}$ is a free multicategory, we may prove that $Fu=
\chi\sbs{a} u$ by induction on the arrows of $\dcc_{\bx}$. If $u$ is an indet
or an identity, there is nothing to prove. To handle the induction step $u =
u\p \circ_r u\pp$, notice first that for any $n$-cell $w$ of $\bb$ and $r \in
\doc{w}$, $F$ preserves the generalized whiskering operation $w \repr -$.
This is seen by induction on $w$, using conditions 1-3 of~\ref{L:reprexists}
which, as stated by that theorem, characterize the generalized whiskering
operations. Once this is done, we infer
\[ Fu= F(u\p \circ_r u\pp) = F(u\p \bullet_n (du\p \repr u\pp)) = Fu\p \bullet_n
F(du\p \repr u\pp) = Fu\p \bullet_n (du\p \repr Fu\pp) \] By the induction
assumption, $F u\p= \chi\sbs{a} u\p,\, Fu\pp = \chi\sbs{a} u\pp$. Also, as
$F$ is the identity on $\bb$-cells, we have $d u\p= Fdu\p= dFu\p =d
\chi\sbs{a} u\p$, hence we can go on with our sequence of equalities  and
conclude
\[= \chi\sbs{a} u\p \bullet_n (d \chi\sbs{a}u\p \repr \chi\sbs{a} u\pp) = \chi\sbs{a} u\p
\circ_r \chi\sbs{a} u\pp = \chi\sbs{a} (u\p \circ_r u\pp) = \chi\sbs{a} u.\]

The last statement of the corollary now follows immediately. If $\bz$ is free
as well, then we have also a unique $\omega$-functor $G: \bz \el \bx$ which
is the identity on $\bb$-cells and extends $\chi\sbs{a}\sps{-1}$. Hence, both
$GF,\,FG$ are identity functors, as they are identities on the cells of $\bb$
as well as on the \emph{many-to-one} $(n+1)$-cells (which include the
$(n+1)$-indets).
\end{proof}

Before turning to the proof of~\ref{L:main}, let's point out the significance
of this theorem at the level of $\stm$-terms. If $t \in \stt(\stm)$ then
$\eqv{t}$ is an arrow of $\Omega(J)$. Let's denote $\val{\eqv{t}}=\val{t}$.
The significance of $\val{t}$ is clear: $t$ describes a way of constructing
an arrow from a-indets and identity arrows by means of repeated
multicomposition operations; $\val{t} \in A(\dcc_{\bx}) \subset X_{n+1}$ is
the $(n+1)$-cell described by $t$ when we interpret the a-indets as the
corresponding $(n+1)$-indets in $\bx$, while the multicomposition operations
$\odot_r$ are interpreted as the $(n+1)$-cell placed compositions $\circ_r$.
Theorem~\ref{L:main} states, first, that $t$ and $s$ denote \emph{distinct}
cells $\val{t}\neq \val{s}$, whenever $\nvdash_{\stm} t=s$.
Furthermore,~\ref{L:main} tells us that an $(n+1)$-cell $u\in X_{n+1}$ is of
the form $\val{t}$ for some $t\in\stt({\stm})$ \emph{iff} $u$ is a
many-to-one cell.

\medskip

Theorem~\ref{L:main} will follow from a stronger and somewhat surprising one
that will be stated after the preliminary discussion below.

The multicategory $\dcc_{\bx}$ has the extension $\dcc^+_{\bx}$ based on the
object system $\Omega^+= (I,B_n,i\sbs{I})$, where $i\sbs{I}$ is the inclusion
map of $I$ into the set $B_n$ of all $n$-cells of $\bb=\ba[I]$. If
$\Omega[J]$ is, indeed, isomorphic to $\dcc_{\bx}$, then it must have an
extension based on $\Omega^+$ which is isomorphic to $\dcc^+_{\bx}$ and we
now set out to identify such an extension. The set $A(\dcc^+_{\bx})$ of
arrows of $\dcc^+_{\bx}$ is also the set of \emph{all} $(n+1)$-cells of $\bx$
and has the following characterization that will assist us in our endeavor:

$A(\dcc^+_{\bx})$ is the least set of arrows containing the indets and the
$(n+1)$-\emph{identity cells} (of $\bx$) and closed under the placed
composition operations $\circ_r$.

(As we use this fact only as a guiding principle, we will not give a full
proof, but only indicate how a categorical composition $\bullet_k$ can be
expressed by means of multicategorical composition in a simple case: assuming
that $u$ and $v$ are many-to-one $(n+1)$-cells such that $u \bullet_k v$ is
defined for some $k<n$, then $u \bullet_k v=(1\sbs{x\bullet_ky}
\circ\sbs{2}v) \circ\sbs{1}u$, where $x=cu,\,y=cv$ and $1,\,2$ are the
indices indicating the occurrences of $x,\,y$ in $x\bullet_ky$.)

We conclude that the set of arrows of $\dcc_{\bx}$, i.e. the set of
many-to-one $(n+1)$-cells of $\bx$, fails to encompass all $(n+1)$-cells,
just because it lacks the identity cells $1_w$ for the $n$-cells $w \in B_n
\setminus I$ that are not $n$-indets. Likewise, the multicategory $\Omega[J]$
lacks arrows that would naturally correspond to the same identity cells. This
observations leads us to the idea of augmenting $J$ by adding \emph{new}
a-indets that will denote these missing items. To be more precise:

\medskip

We extend the set of a-indets $J$ over $\Omega$ to a set $J^+$ of a-indets
over $\Omega^+$ by letting $J^+=J \dot{\cup} \{e_w:\,w \in B_n \setminus I\}$
with the source and target functions extended by setting $Se_w=\oc{w}$ and
$Te_w=w$. The \emph{new} indets $e_w$ will be called, also,
\emph{predeterminates} or, in short, \emph{predets}. From a
\emph{syntactical} point of view, the predets are indets like all the others,
but \emph{semantically} they are predetermined to denote identity cells or
arrows.

\medskip

Consider the multicategory $\Omega^+[J^+]$ freely generated by $J^+$ over
$\Omega^+$. It extends the free multicategory $\Omega[J]$, cf.~\ref{L:submc}.
Let $\varphi: J^+ \el X_{n+1}=A(\dcc^+_{\bx})$ be defined by $\varphi f=f$
for $f\in J$ and $\varphi e_w=1_w$ for $w \in B_n \setminus I$. By the
universal property of free multicategories, there is a unique morphism $\chi:
\Omega^+[J^+] \el \dcc^+_{\bx}$ which is the identity on the object system
$\Omega^+$ and such that $\chi g= \varphi g$ for $g \in J^+$. We denote, for
any $u \in A(\Omega^+[J^+])$, $\chi u=\val{u}^+$. The main property of the
map $\val{-}^+$ is that $\val{u\odot_r v}^+=\val{u}^+\circ_r\val{v}^+$. Using
this, it is easy to infer that $\val{-}^+$ extends the canonical morphism
$\val{-}$. This means that $\val{u}^+=\val{u}$ whenever $u \in A(\Omega[J])$.

We can now state the stronger result to which we alluded above.

\begin{theorem}\label{L:strongmain} $\val{-}^+: \Omega^+[J^+] \el
\dcc^+_{\bx}$ is an isomorphism of multicategories. \end{theorem}

An unexpected feature of this statement is that $\dcc^+_{\bx}$ turns out to
be a free multicategory some of whose generating arrows are, at the same
time, identity cells in a related category.

\medskip

To get a better grasp of the significance of this result, it will be useful
to have a closer look at the structure of the arrows of $\Omega^+[J^+]$. To
shorten terminology, these arrows will be called $J^+$-arrows, while those of
$\Omega[J]$ will be referred to as $J$-arrows.

\begin{claim}\label{L:Jarrows} A $J^+$-arrow $u$ is a $J$-arrow iff $Tu \in
I$. Consequently, if $u=u\sps{\prime}\odot_r u\sps{\prime\prime}$ then
$u\sps{\prime\prime}$ is always a $J$-arrow.\end{claim}

\begin{proof} The "only if" direction is immediate. For the "if" direction,
assume that $Tu \in I$ and prove by induction on arrows that $u$ is a
$J$-arrow. If $u$ is an indet, then it cannot be a predet, hence is a
$J$-arrow. If $u$ is an identity, it must be $1_x$, where $x=Tu$. If
$u=u\sps{\prime}\odot_r u\sps{\prime\prime}$ then $Tu=Tu\sps{\prime} \in I$
and $Tu\sps{\prime\prime} \in I$ as well since otherwise,
$u\sps{\prime\prime}$ could not possibly be composed into another arrow.
Therefore, both $u\sps{\prime}$ and $u\sps{\prime\prime}$ are $J$-arrows, by
the induction hypothesis, hence so is $u$. \end{proof}

We can now show that~\ref{L:strongmain} implies immediately our important
theorem~\ref{L:main}.

\noindent\emph{Proof of~\ref{L:main}}. All we have to show is that $\val{-}$
is a one-to-one mapping from the arrows of $\Omega[J]$, i.e. the $J$-arrows,
\emph{onto} those of $\dcc_{\bx}$. But this follows immediately from the fact
that, by~\ref{L:strongmain}, $\val{-}^+$ is bijective. As $\val{-}^+$ is the
identity on the object system $\Omega^+$, it will map bijectively the arrows
of $\Omega^+[J^+]$ whose targets belong to $I$ onto those of $\dcc^+_{\bx}$
with the same property. \qed

\medskip

\noindent\emph{Proof of~\ref{L:strongmain}}. The advantage of working with
the multicategory $\dcc^+_{\bx}$, rather than $\dcc_{\bx}$, is that its
arrows have an additional structure embodied by the partial categorical
composition operations. If $\val{-}^+$ is, indeed, an isomorphism then its
inverse map will induce a similar additional structure on the arrows of
$\Omega^+[J^+]$ and we ought to be able to identify it.

We will define a new $(n+1)$-category $\by$ such that $\by_n=\bb$ and
$Y_{n+1}=A(\Omega^+[J^+])$. Thus, in particular, $J \subset Y_{n+1}$ and we
will show that, on one hand, \emph{$\by$ is freely generated over $\bb$ by
$J$} and hence, $\by$ is isomorphic to $\bx=\bb[J]$, while, on the other
hand, $\Omega^+[J^+]$ is \emph{identical} with $\dcc^+_{\by}$. From this
follows that $\Omega^+[J^+]$ is isomorphic to $\dcc^+_{\bx}$ and it will be
very easy to show that the canonical morphism $\val{-}^+$ is the isomorphism
that we exhibited.

\medskip

By setting $\by_n=\bb$, we already defined the $\leqslant n$-dimensional
structure of $\by$. Also, as we decided that the $(n+1)$-dimensional cells of
$\by$ are the arrows of $\Omega^+[J^+]$, all that remains to be done is to
define the domain/codomain functions for $(n+1)$-cells, the $(n+1)$-
dimensional identity cells and the compositions of $(n+1)$-cells at all
dimensions$\leqslant n$.

\emph{The domain/codomain functions of $\by$} will be denoted
$\hat{d},\,\hat{c}$ and are defined simply by $\hat{d}u=d\val{u}^+$,
$\hat{c}u=c\val{u}^+$. Thus, we get $\hat{d}u,\hat{c}u \in B_n=Y_n$ and
$\hat{d}u
\parallel \hat{c}u$, as required. Also, for $k<n$, we have $\hat{d}^{(k)}u=
d^{(k)}\val{u}^+=d^{(k)}\hat{d}u,\quad
\hat{c}^{(k)}u=c^{(k)}\val{u}^+=c^{(k)}\hat{c}u$, where $d,\,c$ are the
domain/codomain functions in $\bb$.  Remember that $\val{-}^+$ is the
identity on object systems, hence it preserves sources and targets. As the
source and target of $\val{u}^+$, as an arrow of $\dcc^+_{\bx}$, are
$\oc{d\val{u}^+}$ and $c\val{u}^+$, we infer the following useful equalities:
$Su=\oc{\hat{d} u}$ and $Tu=\hat{c}u$, for all $u\in Y_{n+1}$. Also, $\hat{d}
(u\odot_rv)=\hat{d}u\repr\hat{d}v$ and $\hat{c}(u\odot_rv)= \hat{c}u$, as is
easily seen.

\emph{The identity cells} are easy to define: if $w=x\in I$, then the
identity over $w$ will be the identity arrow $1_x$ and if $w \in B_n
\setminus I$ then the identity cell over $w$ will be the predet $e_w$. We
introduce a helpful notation: for $w \in B_n$, we let $\eps_w=1_x$ if $w=x
\in I$ and $\eps_w=e_w$ when $w \notin I$. Thus, the identity cell over $w
\in B_n=Y_n$ will be, in any case, $\eps_w$.

Before going on, let us remark that, as a consequence of~\ref{L:Jarrows}, the
set of all $J^+$-arrows is the least set $P\subset A(\Omega^+[J^+])$ such
that: (a) $P$ contains all predets and identity arrows (in other words,
$\eps_w \in P$ for all $w \in B_n$) and (b) $u\odot_rv \in P$ whenever $u\in
P$ and $v$ is a $J$-arrow such that $u\odot_r v$ is defined. This observation
will allow us to prove statements by induction on $J^+$-arrows.

We now turn to the definition of the composition operations of $\by$, which
will be denoted $\hat{\bullet}_k$, for $k\leqslant n$. We have to define
these only for cells of dimension $n+1$. This is done through the following
two claims that are strongly suggested by proposition~\ref{L:bulletviacirc}.

\begin{claim}\label{L:ncomp} There is a unique partial binary operation
$\hat{\bullet}_n$ over $Y_{n+1}$, satisfying the following requirements:
\begin{enumerate} \item $u\hat{\bullet}_n v$ is defined iff $\hat{d}u=\hat{c}v$.
\item $\hat{d}(u\hat{\bullet}_n v)= \hat{d} v$ and $\hat{c}(u\hat{\bullet}_n v)= \hat{c}u$.
\item $u\hat{\bullet}_n \eps\sbs{\hat{d}u}=u$.
\item $u\hat{\bullet}_n (v\sps{\prime} \odot_r
v\sps{\prime\prime})= (u\hat{\bullet}_n v\sps{\prime})\odot_r
v\sps{\prime\prime}$.
\end{enumerate} \end{claim}

\begin{proof} The uniqueness of $u \hat{\bullet}_n v$ follows easily by induction
on $v$. We have to show, for every $u \in Y_{n+1}$, the existence of the
partial function $u \hat{\bullet}_n (-)$.

\emph{Case 1:} $\hat{d}u=x \in I$. In this case, $Su=\oc{\hat{d}u}=\oc{x}$
and $\hat{d}u=\hat{c}v$ iff $Tv=\hat{c}v=x$ and we can define $u
\hat{\bullet}_n v = u \odot_r v$, where, of course, $|Su|=\{r\}$. Conditions
\emph{2-4} are easily verified.

\emph{Case 2:} $\hat{d}u=w\sbs{0} \notin I$. We use the \emph{strong}
universal property of $\Omega^+[J^+]$. Let $\gamma: \Omega^+ \el \Omega^+$ be
such that $\gamma\sbs{o}$ is the identity and $\gamma\sbs{t} w=w$ for $w \neq
w\sbs{0}$, while $\gamma\sbs{t} w\sbs{0}=\hat{c} u=Tu$. It is easily seen
that this $\gamma$ is a morphism of object systems. Next, let $\varphi: J^+
\el A(\Omega^+[J^+])$ be defined as $\varphi g=g$ for $g \in J^+ \setminus
\{e\sbs{w\sbs{0}}\}$ and $\varphi e\sbs{w\sbs{0}}=u$. Then $\varphi$ is
\emph{consistent} with $\gamma$, in the sense that $S\varphi g=\gamma Sg$ and
$T\varphi g=\gamma Tg$, hence there is a \emph{unique} morphism $\chi:
\Omega^+[J^+] \el \Omega^+[J^+]$ extending $\gamma$ and $\varphi$. Obviously,
the restriction of $\chi$ to $\Omega[J]$ is the identity. We now define, for
$v$ such that $\hat{c}v=Tv=w\sbs{0}$, $u \hat{\bullet}_n v=\chi v$ and have
to show that conditions \emph{2-4} are met. \emph{3} and \emph{4} are easily
verified and condition~\emph{2} is proven by induction on $v$. We indicate
only the induction step for $\hat{d}$: if $v=v\p \odot_r v\pp$, then
$u\hat{\bullet_n}v= (u\hat{\bullet}_r v\p) \odot_r v\pp$, by
condition~\emph{4}. Hence, $\hat{d}(u \hat{\bullet_n} v)=\hat{d}(u
\hat{\bullet_n} v\p) \repr \hat{d}v\pp$ and, by the induction hypothesis this
equals $\hat{d}v\p \repr \hat{d}v\pp= \hat{d} (v\p \odot_r v\pp) =\hat{d}v$.
\end{proof}

\begin{claim}\label{L:kcomp} For every $k<n$, there is a unique partial
binary operation $\hat{\bullet}_k$ on $Y_{n+1}$, satisfying the following:
\begin{enumerate} \item $u\hat{\bullet}_k v$ is defined iff $\hat{d}^{(k)}u =
\hat{c}^{(k)}v$.
\item $\hat{d}(u
\hat{\bullet}_k v)= \hat{d}u \hat{\bullet}_k \hat{d}v$ and $\hat{c}(u
\hat{\bullet}_k v)= \hat{c}u \hat{\bullet}_k \hat{c}v$ (where, of course, the
composition $\hat{\bullet}_k$ of $n$-cells in $\by$ is the same as
$\bullet_k$ in $\bb$).
\item $\varepsilon\sbs{w} \hat{\bullet}_k \varepsilon\sbs{w
\sps{\prime}}= \varepsilon\sbs{w\hat{\bullet}_k w\sps{\prime}}$.
\item $u\hat{\bullet}_k (v\p \odot_r v\pp)=(u
\hat{\bullet}_k v\p)\odot_r v\pp$.
\item $(u\p \odot_r u\pp)\hat{\bullet}_k v=(u\p \hat{\bullet}_k v)\odot_r u\pp$.
\end{enumerate}\end{claim}

\begin{proof} Again, the uniqueness of $\hat{\bullet}_k$ satisfying
\emph{1-5} is easily established by an induction on $u$ and $v$, so we have
to show only the existence.

It would be nice to produce an argument that uses solely the universal (or
strong universal) property of $\Omega^+[J^+]$, as we did in the proof
of~\ref{L:ncomp}. Unfortunately, we did not find a such, yet. The proof that
we are presenting uses the concrete description of the $J^+$-arrows as
equivalence classes of $\stm^+$-terms, where, of course, $\stm^+$ stands for
the multicomposition language $\stm(\Omega^+,J^+,S,T)$ which is appropriate
for $\Omega^+[J^+]$. Thus, we will define, first, $t \hat{\bullet}_k s$ for
$\stm^+$-\emph{terms} $t,\,s$ satisfying $\hat{d}^{(k)}t=\hat{c}^{(k)}s$,
such that conditions \emph{2-5} will be met (here and in the sequel, we abuse
notation slightly, by letting $\hat{d}t=\hat{d}(\eqv{t})$ and so on). Then we
will show that $\approx$ is a congruence relation with respect to
$\hat{\bullet}_k$ and conclude by setting $u\hat{\bullet}_k v=
t\hat{\bullet}_k s$ for $u=\eqv{t},\,v=\eqv{s}$.

We will define, by recursion on the $\stm^+$-term $t$, the partial function
$t \hat{\bullet}_k (-)$. Assume that $\hat{d}^{(k)}t=\hat{c}^{(k)}s$.

If  $t$ is an identity or a predet, i.e. $t=\varepsilon_w$ for $w\in B_n$, we
define $t\hat{\bullet}_k s$ by recursion on $s$:
\[t \hat{\bullet}_k s =\begin{cases} \varepsilon\sbs{w \hat{\bullet}_k
w\sps{\prime}} & \text{if } s=\varepsilon\sbs{w\sps{\prime}} \\
\varepsilon\sbs{w \hat{\bullet}_k x}\odot_r f & \text{if } s=f\in
J,\,\hat{c}f=Tf=x,\,
\doc{x}=\{r\}\\ (t\hat{\bullet}_k s\sps{\prime})\odot_r s\sps{\prime\prime} &
\text{if } s=s\sps{\prime} \odot_r s\sps{\prime\prime} \end{cases} \] (where,
in the middle case $s=f\in J$, $\doc{x}$ represents the second summand in
$\doc{w}\dot{\cup}\doc{x} = \doc{w \hat{\bullet}_k x} =|S\varepsilon\sbs{w
\hat{\bullet}_k x}|$).

As we proceed with this recursion, we prove by induction on $s$ that
condition \emph{2} is fulfilled, i.e. $\hat{d}(t\hat{\bullet}_k s) = \hat{d}t
\hat{\bullet}_k \hat{d}s$ and $\hat{c}(t\hat{\bullet}_k s) = \hat{c}t
\hat{\bullet}_k \hat{c}s$. The basis of this induction, i.e. the cases in
which $t$ is an identity or a predet or an indet, are easily handled using
the fact that $\val{\varepsilon_w}^+=1_w$, hence $\hat{d}\varepsilon_w=w$.
Let us turn to the case of $s$ being a multicomposition, which is the
induction step. We have: \[ \hat{d}(t\hat{\bullet}_k s)= \hat{d}(
(t\hat{\bullet}_k s\sps{\prime})\odot_r s\sps{\prime\prime})=
d\val{(t\hat{\bullet}_k s\sps{\prime})\odot_r s\sps{\prime\prime}}^+=
d(\val{(t\hat{\bullet}_k s\sps{\prime})}^+ \circ_r
\val{s\sps{\prime\prime}}^+)= d\val{t\hat{\bullet}_k s\sps{\prime}}^+ \repr
d\val{s\sps{\prime\prime}}^+
\] and the induction hypothesis tells us that $d\val{(t\hat{\bullet}_k s\sps{\prime})}^+
=\hat{d}(t\hat{\bullet}_k s\sps{\prime}) = \hat{d}t \hat{\bullet}_k \hat{d}
s\sps{\prime}= d\val{t}^+ \hat{\bullet}_k d\val{s\sps{\prime}}^+$, so that we
can continue the evaluation of $\hat{d}(t\hat{\bullet}_k s)$, keeping in mind
that, $\hat{\bullet}_k$ is the same as the ordinary $\bullet_k$ for
$\by$-cells of dimension $\leqslant n$ :
\[ \hat{d}(t\hat{\bullet}_k s)= (d\val{t}^+ \hat{\bullet}_k
d\val{s\sps{\prime}}^+)\repr d\val{s\sps{\prime\prime}}^+=d\val{t}^+
\hat{\bullet}_k (d\val{s\sps{\prime}}^+ \repr
d\val{s\sps{\prime\prime}}^+)=d\val{t}^+ \hat{\bullet}_k
d\val{s\sps{\prime}\odot_r s\sps{\prime\prime}}^+= \hat{d}t \hat{\bullet}_k
\hat{d}s.\]

The proof that the same is true for the codomain function$\hc$ is similar and
somewhat simpler. It uses the fact that $\hc s=Ts=T(s\p\odot_r s\pp)= Ts\p =
\hc s\p$.

This completes the definition of the $t\hat{\bullet}_k(-)$ function when $t$
is an identity or a predet.

\medskip

If $t$ is an indet $f\in J$, $Tf=\oc{x}$ then we know that
$\vdash\sbs{\stm^+} t=1_x\odot_r f= \varepsilon_x \odot_r t$ and, as we have
already defined the partial function $\varepsilon_x \hat{\bullet}_k (-)$, we
may let $t\hat{\bullet}_k s= (\varepsilon_x \hat{\bullet}_k s)\odot_r t$.

\medskip

Finally, if $t$ is a multicomposition, $t=t\sps{\prime} \odot_r
t\sps{\prime\prime}$ then we let $t\hat{\bullet}_k s= (t\sps{\prime}
\hat{\bullet}_k s)\odot_r t\sps{\prime\prime}$.

\medskip

We leave the reader the verification of condition \emph{2} in these other two
cases.

\medskip

Conditions \emph{3-5} are obviously met for the $\hat{\bullet}_k$ operation
thus defined for $\stm^+$-terms. It remains to show that $\approx$ is a
congruence relation with respect to this operation.

To show that $\vdash t=t\sbs{1}$ implies $\vdash t\hat{\bullet}_k s=t\sbs{1}
\hat{\bullet}_k s$, we proceed by induction on the proof of $t=t_1$.

If $t=t_1$ is an $\stm^+$-axiom, we have to examine five cases (as there are
two kinds of identity axioms). These cases range from trivial to very easy,
\emph{except} (somewhat surprisingly) for the left identity axioms of the
form $t=1_x \odot_r t$. We have to show that $\vdash t \hb_k s=(1_x \hb_k s)
\odot_r t$ and we do this by induction on $t$. Notice that, in this case, $t$
has to be a $J$-arrow, as $Tt=x \in I$. If $t$ is an indet $f \in J$, then we
have by definition that $t \hb_k s=(\eps_x \hb_k s) \odot t$, so there is
nothing to prove (remember that $\eps_x=1_x$). If $t$ is an identity, it has
to be $1_x$ and $t \hb_k s=(1_x \hb_k s) \odot_r t$ becomes an instance of a
right identity axiom. Finally, if $t= t\p \odot_q t\pp$, where $q \neq r$,
then we have: \[\vdash (1_x \hb_k s) \odot_r t=(1_x \hb_k s) \odot_r (t\p
\odot_q t\pp)= ((1_x \hb_k s) \odot_r t\p) \odot_q t\pp,\] by an instance of
the associativity axiom. However, by the induction hypothesis we also have
\[\vdash (1_x \hb_k s) \odot_r t\p =t\p \hb_k s,\]
from which we infer, using the congruence rule, \[\vdash ((1_x \hb_k s)
\odot_r t\p) \odot_q t\pp= (t\p \hb_k s) \odot_q t\pp = t \hb_k s, \] the
last equality holding by the definition of $t\hb_k s$ for the case of $t$
being a multicomposition.

If $t=t_1$ is the conclusion of an inference rule of $\stm^+$ then the
desired equality  follows immediately from the induction hypothesis.

\medskip

The proof that $\vdash s=s\sbs{1}$ implies $\vdash t \hb_k s= t \hb_k
s\sbs{1}$ is similar, once we established that, for $s=f \in J$, with
$Tf=x,\, \doc{x}=\{r\}$, we have $\vdash t \hb_k s = (t \hb_k \eps_x) \odot
_r s$. This is done by induction on $t$ and presents no difficulties.

The proof of the claim is now complete.
\end{proof}

\begin{claim}\label{L:cataxioms} The structure $\by$ that we just described,
is an $(n+1)$-category.
\end{claim}
\begin{proof} We have to verify the axioms for $(n+1)$-cells only.

\medskip

\emph{Verifying the exchange law} $(u\sbs{1}\p \hat{\bullet}_k u\sbs{1}\pp)
\hat{\bullet}_l (u\sbs{2}\p \hat{\bullet}_k u\sbs{2}\pp)= (u\sbs{1}\p
\hat{\bullet}_l u\sbs{2}\p) \hat{\bullet}_k (u\sbs{1}\pp \hat{\bullet}_l
u\sbs{2}\pp)$, when $l<k \leqslant n$ and the expression on the left is
defined (which implies that so is the one on the right).  We have to
distinguish two cases:

\emph{Case 1:} $k=n$. We reason by induction on $u\pp\sbs{1},\, u\pp\sbs{2}$.
If they are both $\eps$'s, i.e. $u\pp\sbs{i}=\eps_{\hd u\p\sbs{i}}$ for
$i=1,2$, then, by~\ref{L:ncomp}, part \emph{3}, the left side of the desired
equality is nothing but $u\p\sbs{1} \hb_l u\p\sbs{2}$; as to the right side,
it is $(u\p\sbs{1} \hb_l u\p\sbs{2}) \hb_n (\eps_{hd u\p\sbs{1}} \hb_l
\eps_{\hd u\p\sbs{2}})$ and is seen to be equal to the same, because
by~\ref{L:kcomp}, parts \emph{3} and \emph{2}, \[\eps_{\hd u\p\sbs{1}} \hb_l
\eps_{\hd u\p\sbs{2}}= \eps_{\hd u\p\sbs{1} \hb_l \hd u\p\sbs{2}}=
\eps_{\hd(u\p\sbs{1} \hb_l u\p\sbs{2})}. \] If any of the $u\pp$s is a
multicomposite, then the exchange axiom follows from the induction
hypothesis, using the connection between the $\hb$ and $\odot$ operations, as
displayed in~\ref{L:ncomp}, part \emph{4} and~\ref{L:kcomp}, parts
\emph{4,5}.

\emph{Case 2:} $k<n$. If all four cells are $\eps$'s, i.e.
$u\sbs{i}\p=\eps\sbs{w_i\p},\, u\sbs{i}\pp=\eps\sbs{w_i\pp}$ then, by part~
\emph{3} of~\ref{L:kcomp}, all we have to show is \[ \eps\sbs{(w\sbs{1}\p
\hat{\bullet}_k w\sbs{1}\pp) \hat{\bullet}_l (w\sbs{2}\p \hat{\bullet}_k
w\sbs{2}\pp)} = \eps\sbs{(w\sbs{1}\p \hat{\bullet}_l w\sbs{2}\p)
\hat{\bullet}_k (w\sbs{1}\pp \hat{\bullet}_l w\sbs{2}\pp)}\] and this follows
by the exchange law in $\bb=\by_n$. Otherwise, if any of the cells is a
$\odot$-composite, then the equality follows easily from the induction
hypothesis, using again the connections between $\hat{\bullet}$ and $\odot$.

\medskip

\emph{The verification of the associative law} is similar, and somewhat
simpler.

\medskip

\emph{The identity laws:}

To verify the left identity law for $\hb_n$, we have to show that $\eps_{\hc
v} \hb_n v=v$. We do this by induction on $v$. If $v$ is an $\eps$, we must
have $\hd v=\hc v$ and $v=\eps_{\hc v}$ and the desired conclusion follows
by~\ref{L:ncomp}, part \emph{3}. If $v=v\p \odot_r v\pp$, then $\hc
v=Tv=T(v\p \odot_r v\pp) = Tv\p =\hc v\p$ and using the induction hypothesis
as well as part \emph{4} of~\ref{L:ncomp}, we conclude that
\[\eps_{\hc v} \hb_n v= (\eps_{\hc v\p} \hb_n v\p) \odot_r v\pp= v\p \odot_r
v\pp =v.\]

The right identity law for $\hb_n$ is part \emph{3} of~\ref{L:ncomp}.

The left identity law for $\hb_k$, with $k<n$, is $\eps_w \hb_k v =v$,
provided that $w = 1\sps{(n)}_a$ where $a= c\sps{(k)} \hc v$. The proof is by
induction on $v$. If $v=\eps_{w\p}$, then $\eps_w \hb_k v = \eps_{w \hb_k
w\p}$ and as $w\p=\hc v$, we have that $a = d\sps{(k)} w\p$, hence $w \hb_k
w\p=w\p$, by the left identity law in $\bb=\by\sbs{n}$, and the desired law
follows. If $v=v\p \odot_r v\pp$, then the conclusion follows easily from the
induction hypothesis, once we notice that $\hc v= Tv=T(v\p \odot v\pp)=
Tv\p=\hc v\p$.

The right identity law is $u \hb_k \eps_{w\p} =u$, where $w\p = 1\sps{(n)}_a$
with $a=d\sps{(k)} \hd u$. The proof, by induction on $u$ is similar, except
that for the induction step $u= u\p \odot_r u\pp$, we have to notice that
$\hd u= d \val{u}^+= d(\val{u\p}^+ \circ_r \val{u\pp}^+) =\hd \val{u\p}^+
\repr \hd \val{u\pp}^+$ and hence, $d\sps{(k)} \hd u= d\sps{(k)} \hd u\p$.
\end{proof}

\medskip

As $\by$ is an $(n+1)$-category whose $n$th truncation is free over its
$(n-1)$th truncation, we may define in it generalized whiskering operations
$w \hrepr -$ for $w \in Y_n$, as described in section~\ref{S:placedcomp}.
Once we did that, we can also define partial placed composition operations
$\hat{\circ}_r$ by the formula \[ u \hat{\circ}_r v= u\hb_n (\hd u \hrepr v)
\text{ for } u,v \in Y_{n+1},\, r\in \doc{u},\, \hc v= \oc{\hd u}(r),\] as in
definition~\ref{D:placedcomp}. Not surprisingly, $\hat{\circ}_r$ turns out to
be the same with the multicomposition operation $\odot_r$ of $\Omega^+[J^+]$.

\begin{claim}\label{L:circriny} If $\by$ is the $(n+1)$-category described
above, then: \begin{enumerate} \item If $w\in Y_n$, $r\in \doc{w}$ and $v \in
Y_{n+1}$ are such that $w\hrepr v$ is defined, then $w\hrepr v= \eps_w
\odot_r v$.
\item For $u,v \in Y_{n+1}$ and $r\in
\doc{u}$ such that $u \odot_r v$ is defined, we have
$u\odot_r v =u \hb_n (\hd u \hrepr v)$. Hence, $\odot_r=\hat{\circ}_r$.
\end{enumerate}\end{claim}

\begin{proof} Part \emph{1}: by induction on the $n$-\emph{cell} $w$.

If $w=x \in I$, then $w\hrep_r v=v= 1_x \odot_r v= \eps_w \odot_r v$.

$w$ cannot be an identity cell, as $\doc{w} \neq \emptyset$.

Finally, if $w=w\p \hb_k w\pp$, assume, e.g., that $r \in \doc{w\p}$. Then
$w\hrepr v = (w\p \hrepr v) \hb_k w\pp=(\eps\sbs{w\p} \odot_r v) \hb_k w\pp$,
where the last equality holds by the induction hypothesis. In these
equalities, $\hb_k$ represents a whiskering, which means that $w\pp$ is just
short for $\eps\sbs{w\pp}$ (which is the identity cell over $w\pp$ in $\by$).
Taking this into consideration, we can go on and conclude that $w\hrepr v=
(\eps\sbs{w\p} \odot_r v) \hb_k \eps\sbs{w\pp}=(\eps\sbs{w\p} \hb_k
\eps\sbs{w\pp}) \odot_r v= \eps\sbs{w\p \hb_k w\pp} \odot_r v = \eps_w
\odot_r v$.

Part \emph{2}: $u \hb_n (\hd u \hrepr v)= u \hb_n (\eps\sbs{\hd u} \odot_r
v)= (u \hb_n \eps\sbs{\hd u}) \odot_r v= u \odot_r v$. \end{proof}

\medskip

Following our plan for the proof of~\ref{L:main}, we now show the following.

\begin{claim}\label{L:yisfree} The $(n+1)$-category $\by$ is freely generated
by $J \subset Y_{n+1}= A (\Omega^+[J^+])$ over $\bb = \by_n$.
\end{claim}

\begin{proof} Let $\bz$ be an $\omega$-category extending $\bb$ and $\varphi:
J \el Z_{n+1}$ a map such that $d\varphi f= \hd f$, $c\varphi f=\hc f$ for $f
\in J$ (here and in the sequel, $d$ and $c$ represent the domain/codomain
functions $d_{\bz},\,c_{\bz}$ of the $\omega$-category $\bz$). We have to
show the existence of a unique $\omega$-functor $G$ extending both, the
identity functor on $\bb$ and $\varphi$. This amounts to specifying the
function that sends each element $u \in Y_{n+1}=A(\Omega^+[J^+])$ to $Gu$,
which is an $(n+1)$-cell of $\bz$ and proving that there is just one such
function that makes $G$ into an $\omega$-functor.

At this point, it is useful to remember that the $(n+1)$-cells of $\bz$ are
the arrows of the extended multicategory $\dcc^+_{\bz}$ which is based on the
object system $\Omega^+$ as well. Our proof will proceed as follows.

\emph{First}, we extend the function $\varphi$ to $\varphi^+: J^+ \el
Z_{n+1}$, by sending the predets to the corresponding identity cells. By the
universal property of $\Omega^+[J^+]$, there is a unique morphism of
multicategories $\chi: \Omega^+[J^+] \el \dcc^+_{\bz}$, which is the identity
on $\Omega^+$ and extends $\varphi^+$.

\emph{Next}, we show that the function $\chi_a$, operating on arrows,
preserves domains, codomains, identity cells as well as $\omega$-categorical
compositions (i.e. $\chi_a (u \hb_k v)= \chi_au \bullet_k \chi_av$ for
$k\leqslant n$, where $\bullet_k$ is the composition in $\bz$). This last
fact follows readily from claims~\ref{L:ncomp},~\ref{L:kcomp},
proposition~\ref{L:bulletviacirc} and the fact that $\chi_a$ preserves
multicomposition. Hence, by setting $Gu= \chi_a u$ for $u \in Y_{n+1}$, we
get an $\omega$-functor as desired.

\emph{Finally}, claim~\ref{L:circriny} implies that any $G$ as above
preserves multicomposition, hence it originates from the unique morphism
$\chi$ that we just described. This proves the uniqueness of $G$.

In the rest of this claim's proof we are elaborating on these three steps.

If we define $\varphi^+ f= \varphi f$ for $f \in J$ and $\varphi e_w=1_w (\in
Z_{n+1})$ for $w \in B_n \setminus J$, we get a function that preserves
sources and targets. Indeed, $S\varphi^+ f= S\varphi f= \oc{d\varphi f}=
\oc{\hd f}=Sf$ for $f \in J$  and $S\varphi^+ e_w=S1_w=\oc{d1_w}=\oc{w}=Se_w$
for $w \in B_n \setminus J$ (notice the ambiguous use of $S$ as denoting
source in $\Omega^+[J^+]$ as well as in $\dcc^+_{\bz}$). A similar
computation shows that $\varphi^+$ preserves targets. The conclusion is that
we can apply the universal property of $\Omega^+[J^+]$ and infer the
existence of the morphism $\chi$ mentioned above.

We have to show that for $u \in A(\Omega^+[J^+])=Y_{n+1}$, $d\chi_a u= \hd u$
and $c\chi_a u = \hc u$. We do this by induction on the \emph{arrow} $u$. If
$u=f \in J$, then $\chi_a u= \varphi f$ and there is nothing to prove. If $u=
\eps_w$ for $w \in B_n$, then $\chi_a u= 1_w$ and $d\chi_a u=w = \hd \eps_w=
\hd u$ and similarly for codomains. As to the induction step: if $u= u\p
\odot_r u\pp$, then $\chi_a u=\chi_a u\p \circ_r \chi_a u\pp$ (where
$\circ_r$ is $(n+1)$-cell placed composition in $\bz$). The induction
hypothesis is that $d\chi_a u\p= \hd u\p,\, d\chi_a u\pp= \hd u\pp$, hence
$d\chi_a u= d\chi_a u\p \repr d\chi_a u\pp= \hd u\p \repr \hd u\pp= \hd (u\p
\odot_r u\pp) = \hd u$, where $\repr$ is cell replacement in \emph{both}
$\bz$ and $\bx$, as we have $\bz_n= \bx_n$. The preservation of codomains is
proven by a similar, but simpler, computation.

It is very easy to see that $\chi_a$ preserves identities. We still have the
task of proving that $\chi_a (u \hb_k v) =\chi_a u \bullet_k \chi_a v$, for
$k \leqslant n$.

For $k=n$, we prove this by induction on $v$. If $v$ is an identity or a
predet, then we must have $v= \eps\sbs{\hd u}$ and the equality is trivial.
If $v=v\p \odot_r v\pp$, then $\chi_a (u \hb_n v)= \chi_a((u \hb_n v\p)
\odot_r v\pp)= \chi_a (u \hb_n v\p) \circ_r \chi_a v\pp$. By using the
induction hypothesis $\chi_a (u \hb_n v\p)=\chi_a u \bullet_n \chi_a v\p$ and
then proposition~\ref{L:bulletviacirc}, we can go on and conclude that
$\chi_a (u \hb_n v)= (\chi_a u\bullet_n \chi_a v\p) \circ_r \chi_a v\pp=
\chi_a u \bullet_n (\chi_a v\p \circ_r \chi_a v\pp)= \chi_a u \bullet_n
\chi_a (v\p \odot_r v\pp)= \chi_a u \bullet_n \chi_a v$.

If $k<n$, then we show by induction on $u$ that $(u \hb_k v) = \chi_a u
\bullet_k \chi_a v$, for all $v$ for which the left hand side is defined (and
hence, so is the right). For $u=\eps_w$, this is done by induction on $v$,
much in the style of the calculation that we just completed  (the main
difference being that this time we use~\ref{L:kcomp}, rather
than~\ref{L:ncomp}). For $u=u\p \odot_r u\pp$, we use~\ref{L:kcomp} again, as
well as the induction hypothesis for $u\p$ and~\ref{L:bulletviacirc}.

By letting $Ga=a$ for $a$ a cell of $\bb=\by_n=\bz_n$ and $Gu= \chi_a u$ for
$u \in Y_{n+1}$, we complete the proof of the \emph{existence} of $G$.

To show \emph{uniqueness}, assume that $G: \by \el \bz$ is an
$\omega$-functor as desired. We have to prove that $G$ must be induced by the
morphism $\chi$ as described above. For this, suffices to show that $G$
preserves multicomposition, meaning that $G(u \odot_r v)= Gu \circ_r Gv$.
This is quite trivial, though: on one hand, we know from~\ref{L:circriny}
that the multicompositions $\odot_r$ are the same as the cell replacements
$\hat{\circ}_r$ in the $\omega$-category $\by$; on the other hand, any
$\omega$-functor like $G$, between two extensions of the $n$-category $\bb$
which extends the identity on $\bb$, clearly preserves placed compositions
between $(n+1)$-cells.

The proof of~\ref{L:yisfree} is now complete.\end{proof}

It follows that $\by$ is isomorphic to $\bx = \bb[J]$ by a unique isomorphism
that extends the identity functions on $\bb$ and $J$. We are now able to
infer immediately the following fact that we stated when outlining the proof
of~\ref{L:strongmain}.

\begin{claim}\label{L:omegaj+isc+y} The multicategories $\Omega^+[J^+]$ and
$\dcc^+_{\by}$ are identical.\end{claim}

\begin{proof} Obviously, the two multicategories have the same object system
$\Omega^+$, the same set of arrows $A(\Omega^+[J^+])=Y_{n+1}$ and the same
source and target functions $Su= \oc{\hd u},\, Tu= \hc u$. Further, they have
the same identity arrows $1_x$, for $x \in I$. By~\ref{L:circriny}, they also
have the same multicomposition operations $\odot_r=\hat{\circ}_r$.
\end{proof}

\emph{Concluding the proof of~\ref{L:strongmain}}: The unique
$\omega$-functor $K: \by \to \bx$ extending the identity maps on both $\bb$
and $J$ is an isomorphism that induces an isomorphism of multicategories
$\kappa: \dcc^+_{\by}=\Omega^+[J+] \to \dcc^+_{\bx}$. In addition, $\kappa$
maps the indets $e\sbs{w},\, w \in B_n \setminus I$, which are also identity
cells in $\by$, to the corresponding identity cells $1_w$ in $\bx$. Hence,
$\kappa$ must be the canonical morphism $\val{-}^+$. \qed

We now mention one more remarkable fact. The elements of the set $A(\Omega^+
[J^+] = Y_{n+1}$ are, at the same time, the arrows of the \emph{free}
multicategory $\Omega^+[J^+]$ and the $(n+1)$-cells of the \emph{free}
extension $\by$ of the $n$-category $\bb$. Therefore, we can define on this
set \emph{two} replacement operations, the multicategorical $\drep_r$
(cf.~\ref{L:arrowrep}) and the $(n+1)$-categorical $\hrep_r$
(cf.~\ref{L:reprexists}). Are these operations the same? Certainly not,
because we might encounter $u,v$ and $r \in \doc{u}$ such that, for $f=
\oc{u}(r)$, we have $Tv =Tf$ (which also means that $\hc v = \hc f$) and $Sv
=Sf$ (which is the same as $\doc{\hd v} =
\doc{\hd f}$), but $\hd v \neq \hd f$. In such a case, $u \drep_r v$ is
defined, while $u \hrep_r v$ is not. However, when both expressions are
defined, they are the same.

\begin{claim}\label{L:drepisrep} If $u,v \in Y_{n+1}$ and $r \in \doc{u}$
are such that $u \hrep_r v$ is defined, then $u \hrep_r v = u \drep_r v$.
\end{claim}

\begin{proof} By induction on the \emph{arrow} $u$ of $\Omega^+[J^+]$. If $u$
is an indet, then both expressions equal $v$. If $u = u\p \odot_q u\pp$ then,
by part \emph{2} of~\ref{L:circriny}, $u= u\p \hb_n (\hd u\p \hrep_q u\pp)$.
If $r \in \doc{u\pp}$ then \[ u \hrep_r v= u\p \hb_n ((\hd u\p \hrep_q u\pp)
\hrep_r v) = u\p \hb_n (\hd u\p \hrep_q (u\pp \hrep_r v)) = u\p \hb_n
(\epsilon\sbs{\hd u\p} \odot_q (u\pp \drep_r v)) \] where the second equality
follows by~\ref{L:repqrepr}, while the third uses the induction hypothesis
for $u\pp$ as well as part \emph{1} of~\ref{L:circriny}. Employing parts
\emph{4,3} of~\ref{L:ncomp}, we go on and conclude \[ =(u\p \hb_n
\epsilon\sbs{\hd u\p}) \odot_q (u\pp \drep_r v) = u\p \odot_q (u\pp \drep_r
v) = (u\p \odot_q u\pp) \drep_r v = u \drep_r v.\] The case $r \in \doc{u\p}$
is similar and simpler. It uses the identity $\hd (u\p \hrep_r v) =\hd u\p$
(cf. condition \emph{1} of~\ref{L:reprexists}). \end{proof}

Of course, the same claim is true for the isomorphic $(n+1)$-category $\bx$
as well. By this we mean that the operation of $(n+1)$-cell replacement in
the free $(n+1)$-category $\bx$ is the \emph{same} with arrow replacement in
the free multicategory $\dcc_{\bx}$, whenever the former is defined.

\medskip

We stated in the introduction that our results imply that, \emph{under
certain conditions}, the free extension $\bx=\bb[J]$ can be construed as a
\emph{term model}. We conclude this section by outlining a proof of this
fact.

\begin{proposition}\label{L:termmodel} Under the assumptions of this section,
there is a primitive recursive function $(-)^{\nu}: \stt(\stc) \el
\stt(\stc)$ which associates with every $\stc$-term $t$ another $\stt$-term
$t^{\nu}$ such that for all $t,s \in \stt(\stc)$, we have that $\vdash t=s$
iff $t\sps{\nu} = s\sps{\nu}$. \end{proposition}

This means that, in the construction of the free extension $\bx=\bb[J]$, we
can substitute the term $t^{\nu}$ for the equivalence class $\eqv{t}$.

\begin{proof} (Sketch) By following our proofs of~\ref{L:ncomp}
and~\ref{L:kcomp}, it is not hard to see that there exists a primitive
recursive function $t \mapsto t\p$ that associates with any $\stc$-term $t$ a
$\stm^+$-term $t\p$, such that $\val{\eqv{t\p}}^+= \eqv{t}$ (hint: one clause
in the recursive definition of $(-)\p$ is $(t\sbs{1} \bullet_k t\sbs{2})\p =
t\sbs{1}\p \hb_k t\sbs{2}\p$, with $t\sbs{1}\p \hb_k -$ defined as in the
proof of~\ref{L:kcomp}).

Next, another primitive recursive function takes any $\stm^+$-term $s$ to a
$\stc$-term $s^{\sharp}$ such that $\val{s}^+= \eqv{s^{\sharp}}$.

Finally, take $t^{\nu}= ((t\p)\sps{\star})^{\sharp}$, where
$(t\p)\sps{\star}$ is the unique \emph{normal} $\stm^+$-term equivalent to
$t\p$ (cf. the discussion that follows proposition~\ref{L:submc}).
\end{proof}

\section{Computads and multitopic sets}\label{S:computads}

The notion of \emph{computad} that we are going to present, was first defined
by Street. A computad is a special kind of $\omega$-category which is
obtained by starting with a $0$-category, i.e. a barren set, taking a free
extension of it which is a $1$-category, i.e. an ordinary category, then
taking a free extension of it which is a $2$-category and so on, \emph{ad
infinitum}. The precise definition is very simply stated.

\begin{definition}\label{D:computad} An $\omega$-category $\ba$ is called a
\emph{computad} if for every $n < \omega$, $\ba_{n+1}$ is a free extension of
$\ba_n$. \end{definition}

Thus, if $\ba$ is a computad then there exists, for every $n < \omega$, a set
$I_{n+1} \subset A_{n+1}$ of $(n+1)$-indets, such that
$\ba_{n+1}=\ba_n[I_{n+1}]$. For the sake of uniformity, we also set
$I\sbs{0}=A\sbs{0}$ and refer, sometimes, to $0$-cells as \emph{$0$-indets}.
A simple proof by induction, using theorem~\ref{L:indetisindec}, shows that
for each $n > 0$, $\ba_n$ is well behaved (cf.
definition~\ref{D:wellbehaved}) and that an $n$-cell $u$ is an $n$-indet iff
it is a non-identity cell indecomposable in the sense of~\ref{D:indec}. Thus,
the sets of indets of a computad are \emph{uniquely} determined.

\begin{definition}\label{D:compfunct} An $\omega$-functor $F: \ba \el \ba\p$
between computads $\ba$ and $\ba\p$ is called a \emph{computad functor} iff
it preserves indets namely, $Fu$ is an indet whenever $u$ is. The category
$Comp$, whose objects are the computads and arrows the computad functors,
will be called the \emph{category of computads}. \end{definition}

Obviously, $Comp$ is a non-full subcategory of the category $\omega Cat$ of
$\omega$-categories.

It is not hard to see that a computad functor preserves not only
$\omega$-categorical, but also computad structure:

\begin{proposition}\label{L:compfunct} Assume that $F: \ba \el \ba\p$ is a
computad functor and $u$ an $n$-cell of $\ba$, $n>0$.
\begin{enumerate} \item There is a bijection $\theta: \doc{u} \el \doc{Fu}$
such that, for $r \in \doc{u}$, $F(\oc{u}(r))= \oc{Fu}(\theta r)$. We will
always assume, as we may, that due to an appropriate reparametrization,
$\theta$ is the identity.
\item $F$ preserves the generalized whiskering operations. This
means that $F(u \repr v)= (Fu) \repr (Fv)$ whenever $r \in \doc{u}$ and $u
\repr v$ is defined. \item $F$ preserves the placed composition operations,
meaning that $F(u \circ_r v)=(Fu) \circ_r (Fv)$, whenever $r \in \doc{u}$ and
$u \circ_r v$ is defined. \end{enumerate}
\end{proposition}

\begin{proof} As $\ba_n$ is a free extension of $\ba_{n-1}$, we can prove
statements by induction on $n$-cells. Parts \emph{1,2} are easily seen by
induction on $u$ and then, part \emph{3} follows immediately, because
$\circ_r$ is defined, in~\ref{D:placedcomp}, in terms of operations that are
preserved by $F$, namely categorical composition, generalized whiskering and
the domain function.
\end{proof}

In view of the results of section~\ref{S:comparing}, we take a special
interest in the case in which all indets are many-to-one.

\begin{definition} A \emph{many-to-one} computad is one in which the codomain
of any $(n+1)$-indet is an $n$-indet, for all $n \in \omega$. The full
subcategory $m/1Comp$ of $Comp$, whose objects are the many-to-one computads,
will be called the \emph{category of many-to-one computads}. \end{definition}

As we  learned from corollary~\ref{L:cxyieldsx}, if $\ba$ is a many-to-one
computad then for each $n$, the many-to-one $(n+1)$-cells of $\ba$ determine
the structure of all $(n+1)$-cells. Let us pursue this line of thought and
take a closer look at the set of all many-to-one cells of $\ba$. Following
our practice, we consider all $0$-cells to be \emph{indets} and, for
convenience, we declare them to be many-to-one cells. All $1$-cells are
many-to-one, but for $n \geqslant 2$, only \emph{some} $n$-cells are many to
one.

The set of many-to-one cells of a many-to-one computad $\ba$ is \emph{not}
closed under the \mbox{$\omega$-categorical} composition operations
$\bullet_k$ and yet, this set enjoys remarkable closure properties. First of
all, if $u$ is a many-to-one cell, then so are its domain and codomain
(assuming, of course, that $u$ has positive dimension). Indeed, $cu$ is an
indet, hence is many-to-one, and $du$ is parallel to $cu$, hence is
many-to-one as well. Next, the many-to-one cells are closed under the
\emph{placed} composition operations. Thus, the many-to-one cells form a
complex structure that deserves a special name. We arrive thus, in a natural
way, to the notion of \emph{multitopic set} that was introduced
in~\cite{HMP3}.

Given a many-to-one computad $\ba$ define, for $n>0$, $\dcc_n=\dcc_{\ba_n}$.
In other words, $\dcc_n$ is the multicategory whose arrows are the
many-to-one $n$-cells of $\ba$, and whose objects (and object types) are the
$(n-1)$-indets. By~\ref{L:main}, $\dcc_n$ is a free multicategory generated
by the $n$-indets. For the sake of completeness, we also let $\dcc_0$ be the
barren set $A_0$ of $0$-cells, viewed as a free multicategory (as indicated
in the ``important example'' following~\ref{L:arrowrep}). Thus, we have a
sequence $(\dcc_n)\sbs{n \in \omega}$, of free multicategories, such that the
generating a-indets of $\dcc_n$ are at the same time the objects (and object
types) of $\dcc_{n+1}$. There is an additional structural item that links
these multicategories, as we have the domain and codomain functions $d,c:
A(\dcc_{n+1}) \el A(\dcc_n)$. The structure $S=S_{\ba}$ consisting of the
sequence $(\dcc_n)\sbs{n\in \omega}$ and the functions $d,\,c$, will be
called the \emph{multitopic set associated with the many-to-one computad
$\ba$}.

We now reproduce the definition of the abstract notion of multitopic set
from~\cite{HMP3}.

We start with a preliminary definition that will describe the connection
between the multicategories $\dcc_n$ and $\dcc_{n+1}$ mentioned above.

\begin{definition}\label{D:freemcext} Given a free multicategory $\dcc= \Omega [J]$, we say that
$\widehat{\dcc}$ is a \emph{free extension} of $\dcc$ via the functions $d$
and $c$ iff the following conditions are met: \begin{enumerate} \item
$\widehat{\Omega} =\Omega (\widehat{\dcc}) = (J)$. In other words,
$\widehat{\dcc}$ is based on the simple object system whose objects are the
a-indets that generate $\dcc$.
\item $\widehat{\dcc}= \widehat{\Omega}[\widehat{J}]$, meaning that $\widehat{\dcc}$ is freely
generated by a set of a-indets $\widehat{J} \subset A(\widehat{\dcc})$. \item
$d$ and $c$ are functions $d: A(\widehat{\dcc}) \el A(\dcc),\,c:
A(\widehat{\dcc}) \el J$, such that for $u \in A(\widehat{\dcc})$, $Su=
\oc{du}$ and $Tu= cu$. Furthermore, $du \parallel cu$, meaning that $Sdu =
Scu,\,
 Tdu = Tcu$. Also, for $x \in J$, $d1_x=c1_x = x$. \item For $u,v
\in A(\widehat{\dcc})$ and $r \in |Su|$ such that the multicomposition $u
\widehat{\odot}_r v$ is defined in $\widehat{\dcc}$, we have $d (u
\widehat{\odot}_r v)= du \drep_r dv$ and $c(u \widehat{\odot}_r v) = cu$
(where $\drep_r$ is the replacement operation in $\dcc$ as defined by
theorem~\ref{L:arrowrep}).
\end{enumerate}
\end{definition}

We are now ready to define:

\begin{definition}\label{D:ms} A \emph{multitopic set} $S$ consists of
sequences $\dcc_n=\dcc_n (S)$ of multicategories and $d_n = d_n(S),\,c_n=
c_n(S)$ of functions, $n \in \omega$, such that the following conditions are
met:
\begin{enumerate} \item $\dcc_0$ is a barren set viewed as a free
multicategory. \item $\dcc_{n+1}$ is a free extension of $\dcc_n$ via the
functions $d_n,\,c_n$, for all $n \in \omega$. \item For $n \in \omega$, we
have $d_n d_{n+1} = d_n c_{n+1},\,c_n d_{n+1} = c_n c_{n+1}$
(\emph{globularity conditions}). \end{enumerate}
\end{definition}

\begin{remark} If $S$ is a multitopic set, then each
$\dcc_n=\dcc_n(S)$ is a multicategory based on a \emph{simple} object system,
as it follows from definition~\ref{D:freemcext}.
\end{remark}

If $\ba$ is a many-to-one computad, then the structure $S_{\ba}$ is a
multitopic set in the sense of this definition, when $d_n,\,c_n$ are the
domain/codomain functions of the $\omega$-category $\ba$ restricted to the
set $A(\dcc_{n+1})$ of the many-to-one $(n+1)$-cells of $\ba$. This is easily
seen, thanks to the remark following claim~\ref{L:drepisrep} (applied to the
free $(n+1)$-category $\bx=\ba_{n+1}$). As we shall see in the next section,
\emph{every} multitopic set is (isomorphic to) some $S_{\ba}$.

Following the notation of~\cite{HMP3}, we shall write $d=d_n(S),\,c=c_n(S)$,
as the subscripts are understood from the context. Thus, the globularity
conditions become \mbox{$dd=dc$} and $cd=cc$.

Other notations and terminology from~\cite{HMP3} that we will use are as
follows. The set of generating a-indets of $\dcc_n(S)$ will be $C_n=C_n(S)$
(its elements are called ``\emph{$n$-cells}'' in~\cite{HMP3}, but we shall
not adopt this terminology here, as it would be confusing in our context,
that mentions so often $n$-cells in $\omega$-categories). The set of arrows
of $\dcc_n$ is $P_n=P_n(S)$ and its members are called \emph{$n$-pasting
diagrams}, because they can be naturally given a diagrammatic representation
(cf.~\cite{HMP1}). Notice that $P_0=C_0$.

\medskip

A multitopic set $S$ is called \emph{$n$-dimensional} iff $C_k(S)= \emptyset$
for all $k > n$; this condition implies that all pasting diagrams of
dimension $>n$ are \emph{identities}. An $n$-dimensional multitopic set is
determined by the \emph{finite} sequence $\langle \dcc_k \rangle_{k \leqslant
n}$ of its first $n+1$ components. If $S$ is any multitopic set, its $n$th
truncation will be $n$-dimensional multitopic set $S_n$ with $\dcc_k(S_n) =
\dcc_k(S)$, for $k \leqslant n$. Obviously, for a many-to-one computad $\ba$,
the $n$th truncation of $S_{\ba}$ is $S_{\ba_n}$.

\medskip

Next, we define the obvious notion of \emph{morphism} of multitopic sets.

\begin{definition}\label{D:msmorphism} A morphism $\Phi: S \el S\p$ between
multitopic sets $S$ and $S\p$ is a sequence $\langle \phi_n \rangle_{n<
\omega}$ of maps $\phi_n:P_n \el P\p_n$ (where here and in the sequel,
unprimed notations, like $P_n$ refer to components of $S$, while their primed
counterparts, like $P_n\p$, refer to $S\p$), that preserve the multitopic
structure, meaning that for each $n<\omega$:
\begin{enumerate}
\item $\phi_n$ maps indets to indets, i.e., $\phi_n x \in C\p_n$
whenever $x \in C_n$. \item If $\tilde{\phi}_n$ is the restriction of
$\phi_n$ to $C_n$, then the pair $\chi=(\tilde{\phi}_n, \phi_{n+1})$ is a
morphism of multicategories from $\dcc_{n+1}$ to $\dcc\p_{n+1}$. \item For $u
\in P_{n+1}$, we have $d \phi_{n+1} u= \phi_n du$ and $c \phi_{n+1} u =
\phi_n cu$ (notice the context sensitivity of the notation for the
domain/codomain functions: $d,c$ refer to $S\p$ on the left sides of the
equations, and to $S$ on the right).
\end{enumerate}
\end{definition}

\begin{notation} For a morphism $\Phi$ as above and for $u \in P_n$, we
denote $\Phi u = \phi_n u$. Thus, $\Phi$ can be viewed as one single,
dimension preserving, function from the pasting diagrams of $S$ to those of
$S\p$. \end{notation}

Obviously, if $S$ is an $n$-dimensional multitopic set and $\Phi:S \el S\p$
is a morphism then the components $\phi_k$ of $\Phi$ for $k>n$ are trivial,
and $\Phi$ is determined by its first $n+1$ components and we write $\Phi =
\langle \phi_k \rangle_{k \leqslant n}$. One useful instance of this is the
following: if $\Phi = \langle \phi_n \rangle_{n< \omega}:S \el S\p$ is a
morphism of multitopic sets, the so is $\Phi_n:S_n \el S\p_n$, where $\Phi_n
= \langle \phi_k \rangle_{k \leqslant n}$. $\Phi_n$ will be called the
\emph{$n$th truncation} of $\Phi$.

\medskip

\begin{remark} Morphisms of multitopic sets are determined by their values on indets.
These values can be chosen arbitrarily, subject to certain restrictions that
insure the preservation of domains/codomains. More explicitly, a stepwise
process of building a multitopic morphism goes as follows. We start by
choosing $\phi_0: C_0 \el C\p_0$ \emph{arbitrarily}. Assuming that we have
already constructed $\phi_k$, for $k \leqslant n$ such that $\langle \phi_k
\rangle_{k \leqslant n}$ is a morphism from $S$ to $S\p$, we start the
construction of $\phi_{n+1}$ by choosing a function $\tilde{\phi}_{n+1}:
C_{n+1} \el C\p_{n+1}$ \emph{arbitrarily}, subject to the restriction that
$d\p \tilde{\phi}_{n+1} f= \phi_n df$ and similarly for the codomain
function. There is a \emph{unique} morphism $\chi: \dcc_{n+1} \el
\dcc\p_{n+1}$ such that $\chi x = \phi_n x$ and $\chi f = \tilde{\phi}_{n+1}
f$ for $x \in C_n,\, f \in C_{n+1}$. We define $\phi_{k+1} u= \chi u$ for $u
\in P_{k+1}$. Then $\phi_{k+1}$ extends $\tilde{\phi}_{n+1}$, and we know
that it satisfies condition 3 of~\ref{D:msmorphism} for $u \in C_{n+1}$.
Using~\ref{L:arrowreppres}, we can show that the same condition is fulfilled
for \emph{all} $u \in P_{n+1}$.\end{remark}

\medskip

The composition of morphisms of multitopic sets is again such a morphism.
Also, for a multitopic set S, the sequence of identity maps $id_n:P_n(S) \el
P_n(S)$ is a morphism from $S$ to itself. Hence we may define a new category:

\begin{definition}\label{D:mltSet} The category $mltSet$, whose objects are the
multitopic sets and arrows their morphisms, is called the \emph{category of
multitopic sets}. \end{definition}

Can we extend the function $\ba \mapsto S_{\ba}$ to a functor? We can, and
actually, much more is true.

\begin{theorem}\label{L:sAisfff} The function that associates the multitopic set $S_{\ba}$
to any many-to-one computad $\ba$ can be extended to a functor $S_{-}:
m/1Comp \el mltSet$ which is full and faithful. \end{theorem}

\begin{proof}Given a computad functor $F: \ba \el \ba\p$ between many-to-one computads
$\ba$ and $\ba\p$, we have to define a morphism $S_F: S_{\ba} \el S_{\ba\p}$
of multitopic sets. We set $S_F = \langle \phi_n \rangle_{n < \omega}$ where
$\phi_n$ is the restriction of $F$ to the set of many-to-one $n$-cells of
$\ba$, which is the same with the set $P_n=P_n(S_{\ba})$ of the $n$-pasting
diagrams of $\dcc_n=\dcc_{\ba_n}$. As $F$ is a computad map, it maps indets
to indets and, therefore, condition \emph{1} of~\ref{D:msmorphism} is
fulfilled. Conditions \emph{2-3} are also satisfied, as it follows
by~\ref{L:compfunct}. Thus, $S_F = \langle \phi_n \rangle_{n< \omega}$ is,
indeed, a morphism of multitopic sets, according to~\ref{D:msmorphism}. The
functoriality of $F \mapsto S_F$ is readily verified.

The functor $S_-$ is \emph{faithful}. Indeed, if $S_F= S_G$ then we show by
induction on $n$ that $F_n= G_n$, where $F_n,G_n: \ba_n \el \ba\p_n$ are the
restrictions of $F,G$ to the $n$th truncation of $\ba$. The case $n=0$ is
trivial, because $F_0,G_0$ are both the $0$th component of $S_F=S_G$. If
$n>0$, then $F_n,G_n$ extend $F_{n-1},G_{n-1}$ respectively and by the
induction hypothesis, $F_{n-1}=G_{n-1}$. Thus, as $\ba_n$ is a free extension
of $\ba_{n-1}$, to infer that $F_n$ and $G_n$ are equal, we have only to show
that they are equal on the set  of \emph{$n$-indets}, which equals $C_n
\subset P_n$. This is clear, however, as the restrictions of $F_n, G_n$ to
$P_n$ are, both, equal to the $n$th component of $S_F=S_G$.

Finally, we can show that $S_-$ is \emph{full}. Given a morphism $\Phi:
S_{\ba} \el S_{\ba\p}$ we define by induction the sequence $\langle F_n
\rangle_{n<\omega}$ of truncations of an $\omega$-functor $F: \ba \el \ba\p$
such that $S_F=\Phi$. We start by letting $F_0=\phi_0$. Once we have $F_n$,
we let $F_{n+1}$ be the unique $(n+1)$-functor $H: \ba_{n+1} \el \ba\p_{n+1}$
that extends $F_n$ and satisfies $Hf = \phi_{n+1}f$ for $f$ an $(n+1)$-indet
(by~\ref{L:compfunct}, it follows that $Hu = \phi_{n+1} u$ for $u$ any
many-to-one $(n+1)$-cell of $\ba$).
\end{proof}

\section{Multitopic sets are equivalent to many-to-one computads}
\label{S:equivalence}

\begin{definition}\label{D:assignment} We say that $\Sigma$ is an
\emph{assignment} of a multitopic set $S$ into a many-to-one computad $\ba$,
and denote this as $\Sigma: S \el \ba$, iff $\Sigma: S \el S_{\ba}$ is a
morphism of multitopic sets.
\end{definition}

\begin{remark} If $\Sigma: S \el \ba$ is an assignment and $F: \ba \el \ba\p$
is a computad functor in $m/1Comp$ then the composite function $\Theta =
F\Sigma$ is an assignment $\Theta: S \el \ba\p$. \end{remark}

\medskip

Roughly speaking, an assignment $\Sigma$ is determined by its values on the
a-indets that generate $S$. By this we mean that once we know the $n$th
component $\sigma_n$, $\sigma_{n+1}$ is uniquely determined by the values
$\sigma_{n+1}f \in A_{n+1}$ for $f \in C_{n+1}(S)$. These values can be
chosen arbitrarily, apart from the conditions that domains and codomain
should be preserved (i.e. $d \sigma_{n+1}f=\sigma_n df$ and similarly for
codomains).

As we shall see, theorem~\ref{L:main} implies that every multitopic set is
(isomorphic to) $S_{\ba}$, for some many-to-one computad $\ba$. Actually, we
prove somewhat more:

\begin{proposition}\label{L:SisSA} For every multitopic set $S$ there is a
many-to-one computad $\fr{S}$ and an assignment $\ff{S}: S \el \fr{S}$ such
that:
\begin{enumerate} \item $\ff{S}$ is an isomorphism of multitopic sets. \item
For any assignment $\Sigma: S \el \bb$ into a many-to-one computad $\bb$,
there is a \emph{unique} computad functor $F: \fr{S} \el \bb$ such that
$\Sigma = F\ff{S}$.
\end{enumerate}
\end{proposition}

Before proving~\ref{L:SisSA}, let us state two important corollaries. The
first one is the main result of this article.

\begin{theorem}\label{L:mainresult} The categories $m/1Comp$ and $mltSet$ are
equivalent. Actually, the functor $S_-: m/1Comp \el mltSet$ is an equivalence
of categories. \end{theorem}

\begin{proof} By~\ref{L:sAisfff}, $S_-$ is full and faithful.
By~\ref{L:SisSA} part~\emph{1}, $S_-$ is essentially surjective on objects,
i.e., for every object $S$ of $mltSet$, there is an object $\ba$ of $m/1Comp$
 such that $S$ is isomorphic to $S_{\ba}$ (we mean, of course, that
$\ba=\fr{S}$). These conditions mean that $S_-$ is an equivalence of
categories. \end{proof}

The second corollary states that $\fr{-}$ and $\ff{-}$ are functorial. To
explain the functoriality of the second of these functions, we have to define
one more category.

\begin{definition}\label{Ass} The category $Ass$ of \emph{assignments} is
defined as follows. The \emph{objects} are the assignments $\Sigma: S \el
\ba$ from multitopic sets to many-to-one computads. An \emph{arrow} with
domain $\Sigma: S \el \ba$ and codomain $\Sigma\p: S\p \el \ba\p$ will be a
pair $(\Phi, F)$ consisting of a morphism $\Phi: S \el S\p$ and a computad
functor $F: \ba \el \ba\p$, such that the following diagram commutes: $$\bfig
\square[{S}`{\ba}`{S\p}`{\ba\p};{\Sigma}`{\Phi}`{F}`{\Sigma\p}] \efig$$
\end{definition}

Thus, if $S$ is a multitopic set, then $\ff{S}: S \el \fr{S}$ is an object of
the category $Ass$.

\begin{theorem}\label{L:frisfunct} $\fr{-}$ and $\ff{-}$ can
be expanded to functors $\fr{-}: mltSet \el m/1Comp$ and $\ff{-}: mltSet \el
Ass$ such that, for any morphism $\Phi: S \el S\p$ in $mltSet$, we have
$\ff{\Phi} = (\Phi, \fr{\Phi})$.
\end{theorem}

\begin{remark} The last condition means that the following diagram commutes: $$\bfig
\square[{S}`{\fr{S}}`{S\p}`{\fr{S\p}};{\ff{S}}`{\Phi}`{\fr{\Phi}}`{\ff{S\p}}]
\efig$$ \end{remark}

\begin{proof} We have to define arrows $\fr{\Phi},\,\ff{\Phi}$ in $m/1Comp$
and $Ass$, respectively. The composite function $\Sigma= \ff{S\p} \Phi$ is an
assignment from $S$ to the many-to-one computad $\fr{S\p}$. By~\ref{L:SisSA}
part~\emph{2}, there is a \emph{unique} computad functor $F: \fr{S} \el
\fr{S\p}$ such that $\Sigma = F \ff{S}$. We now define the arrows
$\fr{\Phi}=F: \fr{S} \el \fr{S\p}$ of $m/1Comp$ and $\ff{\Phi}= (\Phi,F):
\ff{S} \el \ff{S\p}$ of $Ass$. It is easy to verify that we have thus defined
the desired functors. \end{proof}

\noindent\emph{Proof of~\ref{L:SisSA}.} We define, by induction, the
truncations of $\fr{S}= \ba$ and of $\ff{S}= \Phi$. To be more precise, we
will define sequences $\langle \ba_n \rangle_{n < \omega}$ and $\langle
\phi_n \rangle_{n<\omega}$ such that the following conditions are fulfilled:
\begin{description}
\item[a.] $\ba_n$ is an $n$-dimensional many-to-one computad.
\item[b.] $\ba_{n+1} = \ba_n[C_{n+1}]$, which means that $\ba_{n+1}$ is a
free extension of $\ba_n$ generated by a set of many-to-one $(n+1)$-indets
which is identical with the set $C_{n+1}=C_{n+1}(S)$ of a-indets that
generate $\dcc_{n+1} = \dcc_{n+1}(S)$ over $\dcc_n$, as indicated in
definitions~\ref{D:freemcext} and~\ref{D:ms}.
\item[c.] $\phi_n: P_n \el A_n$ and $\Phi_n= \langle \phi_k \rangle_{k
\leqslant n}$ is an isomorphism $\Phi_n: S_n \el S_{\ba_n}$ of
$n$-dimensional multitopic sets such that $\phi_n x =x$ for $x \in C_n$.
\item[d.] Condition \emph{2} of~\ref{L:SisSA} is fulfilled with
$S_n,\,\ba_n$ and $\Phi_n$ replacing $S,\, \fr{S}$ and $\ff{S}$,
respectively.
\end{description}

Once this is done, we will take $\fr{S}$ and $\ff{S}$ having $\langle \ba_n
\rangle_{n<\omega}$ and $\langle \Phi_n \rangle_{n< \omega}$ as sequences of
truncations.

As the basis of the induction, we set $\ba_0 = P_0=C_0$ and take $\phi_0$ to
be the identity function.

Assume that we defined already $\ba_n$ and $\Phi_n= \langle \phi_k \rangle_{k
\leqslant n}$.

\emph{Defining $\ba_{n+1}$.} Let us define functions $d\p, c\p: C_{n+1} \el
A_n$ by letting $d\p f= \phi_n df$ and $c\p f= \phi_n cf$, for $f \in
C_{n+1}$. The functions $d,c: C_{n+1} \el P_n$ are closely related to their
primed counterparts. Indeed, as $\phi_n$ is the identity on indets, we have
$\oc{d\p f} = \oc{df}$; moreover, $c\p f= cf$, as $cf \in C_n$ and hence,
$\phi_n cf =cf$. Because of these considerations, \emph{we shall denote these
newly defined functions by $d,c$, rather than $d\p, c\p$}. Using the fact
that, by induction hypothesis, $\Phi_n$ is an isomorphism between the
multitopic sets $S_n$ and $S_{\ba_n}$, we infer that $df,\,cf$ are
\emph{parallel} as $n$-cells of $\ba_n$ and therefore, $C_{n+1}$ together
with the functions $d,c:C_{n+1} \el A_n$ becomes a set of
\emph{$(n+1)$-indets} over $\ba_n$. We now define $\ba_{n+1}= \ba[C_{n+1}]$,
and thus fulfill condition \textbf{b.} above.

\emph{Defining $\phi_{n+1}$.} By~\ref{D:freemcext} and~\ref{D:ms}, we have
$\dcc_{n+1} = \Omega[C_{n+1}]$, where $\Omega$ is the \emph{simple} object
system with set of objects $C_n$. The same $\Omega$ is also the object system
of the multicategory $\dcc_{\ba_{n+1}}$ whose arrows are the many-to-one
$(n+1)$-cells of $\ba_{n+1}$. The indets $f \in C_{n+1}$ are arrows of
$\dcc_{n+1}$ as well as of $\dcc_{\ba_{n+1}}$, and have the \emph{same}
source and target, $Sf = \oc{df}$ and $Tf = cf$, in both multicategories. At
this point of the proof, we use our main technical result~\ref{L:main} and
conclude that the \emph{canonical} morphism $\val{-}: \Omega[C_{n+1}] \el
\dcc_{\ba{n+1}}$ (i.e. the unique morphism that is the identity on both,
$C_n$ and $C_{n+1}$) is an \emph{isomorphism}. We \emph{define} $\phi_{n+1}:
P_{n+1} \el A_{n+1}$ by $\phi_{n+1} u= \val{u}$.

\emph{Verifying condition} \textbf{c.} The pair $(\tilde{\phi_n},
\phi_{n+1})$ (cf. the notation used in~\ref{D:msmorphism}) is the same with
$\val{-}$, hence it is an isomorphism of multicategories. We have to prove,
in addition, that $d\phi_{n+1}u= \phi_n du$, for \emph{all} $u \in P_{n+1}$.
We show this by induction on $(n+1)$ pasting diagrams. To begin with, this is
given for $u \in C_{n+1}$ and immediate for identities. For the induction
step, we use the fact that $\phi_{n+1}$ preserves multicomposition and infer:
\[d \phi_{n+1} (u \odot_r v) = d(\phi_{n+1} u \circ_r \phi_{n+1} v)) = d
\phi_{n+1} u \repr d\phi_{n+1} v= \] Using the induction hypothesis as well
as the fact that, by proposition~\ref{L:arrowreppres}, $\phi_n$ preserves
arrow replacement, we go on and conclude  \[= \phi_n du \repr \phi_n dv =
\phi_n(du \drep_r dv) = \phi_n d(u \odot_r v).\]

\emph{Verifying condition} \textbf{d.} Given an assignment $\Sigma: S_{n+1}
\el \bb$, let $\Sigma\p: S_n \el \bb$ be its restriction to $S_n$. By the
induction hypothesis, we have a computad functor $F\p: \ba_n \el \bb$ such
that $\Sigma\p = F\p \Phi_n$. $F\p$ has a unique extension $F: \ba_{n+1} =
\ba_n[C_{n+1}] \el \bb$ such that $F f= \Sigma f$ for $f \in C_{n+1}$. To
show that the assignments $\Sigma$ and $F \Phi_{n+1}$ from $S_{n+1}$ to $\bb$
are equal, we have only to show that they induce the same multicategory
morphism from $\dcc_{n+1}=\Omega[C_{n+1}]$ to $\dcc_{\bb_{n+1}}$. To this
end, it suffices to show that they are equal on the indets in $C_n$ and
$C_{n+1}$ and this is readily seen. Indeed, for $x \in C_n$, this follows
from $\Sigma\p =F\p \Phi_n$, while for $f \in C_{n+1}$, $\Sigma f = F f = F
\phi_{n+1} f$. Thus, condition \emph{2} of~\ref{L:SisSA} is established. \qed
\section{Concluding remarks}\label{S:concluding} A noteworthy result
of~\cite{HMP3} says that the category $mltSet$ of multitopic sets is a
\emph{presheaf} category, i.e. it is equivalent to the category $Set
^{Mlt\sps{op}}$ of the \emph{contravariant} functors from a certain category
$Mlt$, called the category of \emph{multitopes}, into the category of sets
$Set$. Thus, from our main result~\ref{L:mainresult}, we infer that \emph{the
category $m/1Comp$ of many-to-one multitopic sets is a presheaf category} as
well. This is a remarkable fact, since it is known that the category $Comp$
of all computads is \emph{not} a presheaf category, as shown in~\cite{MZ}.

The objects of $Mlt$, as described in~\cite{HMP3}, are the same as the
pasting diagrams of the \emph{terminal} multitopic set. An alternative
description of $Mlt$ was given recently by the third named author of this
paper, cf.~\cite{Z}.

\medskip

As a corollary of our proposition~\ref{L:termmodel}, we infer that \emph{the
word problem for many-to-one computads is solvable}. The meaning of this
statement is, roughly, as follows. A computad $\ba$ is determined by the
sequence $(I_n)\sbs{n \in \omega}$ of sets of indets of the various
dimensions. One can set up a large language which has terms denoting the
cells of $\ba$. This language has a hierarchical structure, being built in
consecutive stages. In the initial stage we have a language $\stc_0$ whose
terms are the indets $x \in I_0$. Once the $n$th stage language $\stc_n$ is
defined, we take the next one to be $\stc_{n+1} = \stc(\ba_n, I_{n+1}, d,c)$
whose terms are defined as in definition~\ref{D:cterms}, with one difference:
the values of the domain/codomain functions $dt, ct$ of a $\stc_{n+1}$-term
$t$ are \emph{$\stc_n$-terms}, rather than $n$-cells of $\ba$. The meaning of
$\stc_{n+1}$-terms is clear, once the semantics of $\stc_n$ is understood.
Each $\stc_n$ comes with its deduction system, similar to the one defined
in~\ref{D:axioms}. The word problem for $\sta$ is to find an algorithm for
deciding whether $t=s$ is $\stc_n$-provable or not, for given $\stc_n$ terms
$t,s$. As we mentioned already, \ref{L:termmodel} implies that we have such
an algorithm, actually a \emph{primitive recursive} one, for $\ba$ a
many-to-one computad.

After a first draft of the present work has been completed, the second named
author proved that \emph{the word problem for arbitrary computads is
solvable} as well., cf.~\cite{M}. His algorithm is very different from the
present one. It is \emph{not} based on the existence of term models and
actually, we do not know if a result similar to~\ref{L:termmodel} is true for
arbitrary, not necessarily many-to-one, free extensions.

\medskip

\textbf{Acknowledgement.} We thank Michael Barr for creating his new diagram
package, which we used for drawing the few diagrams of this work.


\begin{thebibliography}{10}

\bibitem{BD1}
John~C. Baez and James Dolan, \emph{Higher-dimensional algebra and
topological
  quantum field theory}, J. Math. Phys. \textbf{36} (1995), no.~11, 6073--6105.

\bibitem{BD2}
John~C. Baez and James Dolan, \emph{Higher-dimensional algebra. {III}.
{$n$}-categories and the
  algebra of opetopes}, Adv. Math. \textbf{135} (1998), no.~2, 145--206.

\bibitem{B}
M.~A. Batanin, \emph{Computads for finitary monads on globular sets}, Higher
  category theory (Evanston, IL, 1997), Contemp. Math., vol. 230, Amer. Math.
  Soc., Providence, RI, 1998, pp.~37--57.

\bibitem{Gr}
George Gr{\"a}tzer, \emph{Universal algebra}, second ed., Springer-Verlag,
New
  York, 1979.

\bibitem{HMP1}
Claudio Hermida, Michael Makkai, and John Power, \emph{On weak higher
  dimensional categories. {I}. 1}, J. Pure Appl. Algebra \textbf{154} (2000),
  no.~1-3, 221--246, Category theory and its applications (Montreal, QC, 1997).

\bibitem{HMP2}
Claudio Hermida, Michael Makkai, and John Power, \emph{On weak
higher-dimensional categories. {I}.2}, J. Pure Appl.
  Algebra \textbf{157} (2001), no.~2-3, 247--277.

\bibitem{HMP3}
Claudio Hermida, Michael Makkai, and John Power, \emph{On weak
higher-dimensional categories. {I}. 3}, J. Pure Appl.
  Algebra \textbf{166} (2002), no.~1-2, 83--104.

\bibitem{Lambek}
Joachim Lambek, \emph{Deductive systems and categories. {II}. {S}tandard
  constructions and closed categories}, Category Theory, Homology Theory and
  their Applications, I (Battelle Institute Conference, Seattle, Wash., 1968,
  Vol. One), Springer, Berlin, 1969, pp.~76--122.

\bibitem{Leinster}
Tom Leinster, \emph{A survey of definitions of {$n$}-category}, Theory Appl.
  Categ. \textbf{10} (2002), 1--70 (electronic).

\bibitem{CWM}
Saunders Mac~Lane, \emph{Categories for the working mathematician}, second
ed.,
  Graduate Texts in Mathematics, vol.~5, Springer-Verlag, New York, 1998.

\bibitem{M1}
Michael Makkai, \emph{The multitopic omega-category of all multitopic
  omega-categories}, Report, McGill University,
  http://www.math.mcgill.ca/makkai/, 1999.

\bibitem{M}
Michael Makkai, \emph{The word problem for computads}, Report, McGill
University,
  http://www.math.mcgill.ca/makkai/, 2005.

\bibitem{MZ}
Michael Makkai and Marek Zawadowski, \emph{The category of 3-computads is not
  cartesian closed}, J. Pure Appl. Algebra \textbf{to appear} (2008).

\bibitem{P}
Thorsten Palm, \emph{Dendrotopic sets}, Galois theory, {H}opf algebras, and
  semiabelian categories, Fields Inst. Commun., vol.~43, Amer. Math. Soc.,
  Providence, RI, 2004, pp.~411--461.

\bibitem{S1}
Ross Street, \emph{Limits indexed by category-valued {$2$}-functors}, J. Pure
  Appl. Algebra \textbf{8} (1976), no.~2, 149--181.

\bibitem{S2}
Ross Street, \emph{Categorical structures}, Handbook of algebra, Vol.\ 1,
  North-Holland, Amsterdam, 1996, pp.~529--577.

\bibitem{Z}
Marek Zawadowski, \emph{Multitopes are the same as principal ordered face
  structures}, Report, Warsaw University, http://duch.mimuw.edu.pl/\~{
  }zawado/papers.htm, 2008.

\end{thebibliography}
\end{document}